\documentclass[reqno]{amsart}

\usepackage[T1]{fontenc}
\usepackage[utf8]{inputenc}
\usepackage{graphicx}
\usepackage{mathtools}
\usepackage{tikz}
\usetikzlibrary{arrows, backgrounds, calc, chains, decorations, patterns, positioning, shapes}

\usepackage{geometry}
\IfFormatAtLeastTF{2026-06-01}{}{\usepackage{thmtools}}
\usepackage[pdfpagelabels]{hyperref}
\usepackage{cleveref}
\usepackage{subcaption}
\usepackage{parskip}

\usepackage{amsfonts}
\usepackage{amsmath}
\usepackage{amsrefs}
\usepackage{amssymb}
\usepackage{amstext}
\usepackage{amsthm}

\usepackage{silence}
\theoremstyle{plain}
\newtheorem{theorem}{Theorem}[section]

\newtheorem{corollary}[theorem]{Corollary}
\newtheorem{lemma}[theorem]{Lemma}
\newtheorem{proposition}[theorem]{Proposition}

\theoremstyle{definition}
\newtheorem{definition}[theorem]{Definition}

\theoremstyle{remark}
\newtheorem{remark}[theorem]{Remark}
\newtheorem{example}[theorem]{Example}

\makeatletter%
\DeclareRobustCommand*{\bfseries}{%
    \not@math@alphabet\bfseries\mathbf
    \fontseries\bfdefault\selectfont
    \boldmath
}
\makeatother

\renewcommand{\th}{^\mathsf{th}}

\DeclareMathOperator{\D}{D}
\DeclareMathOperator{\Exp}{Exp}

\DeclareMathOperator{\LD}{LD}

\DeclareMathOperator{\R}{R}
\DeclareMathOperator{\DP}{DP}
\DeclareMathOperator{\area}{area}
\DeclareMathOperator{\dinv}{dinv}
\DeclareMathOperator{\touch}{touch}

\DeclareMathOperator{\Des}{Des}
\DeclareMathOperator{\cc}{c}

\newcommand{\N}{\mathbb{N}}
\newcommand{\Z}{\mathbb{Z}}
\newcommand{\Q}{\mathbb{Q}}

\newcommand{\A}{\mathbb{A}}%\newcommand{\A}{\mathcal{A}}
\newcommand{\Ht}{\widetilde{H}}
\newcommand{\z}{\widehat{z}}

\let\oldTheta\Theta
\let\Theta\undefined
\DeclareMathOperator{\Theta}{\oldTheta}

\makeatletter
\providecommand*{\shuffle}{%
  \mathbin{\mathpalette\shuffle@{}}%
}
\newcommand*{\shuffle@}[2]{%
  \sbox0{$#1\vcenter{}$}%
  \kern .15\ht0 % side bearing
  \rlap{\vrule height .25\ht0 depth 0pt width 2.5\ht0}%
  \raise.1\ht0\hbox to 2.5\ht0{%
    \vrule height 1.75\ht0 depth -.1\ht0 width .17\ht0 %
    \hfill
    \vrule height 1.75\ht0 depth -.1\ht0 width .17\ht0 %
    \hfill
    \vrule height 1.75\ht0 depth -.1\ht0 width .17\ht0 %
  }%
  \kern .15\ht0 % side bearing
}
\makeatother

\newcommand{\qbinom}[3][q]{\genfrac{[}{]}{0pt}{}{#2}{#3}_{#1}}

\newcommand{\<}{\langle}
\renewcommand{\>}{\rangle}

\definecolor{valleys}{RGB}{0, 100, 0}
\definecolor{rises}{RGB}{128, 0, 0}
\definecolor{rho}{RGB}{255, 127, 0}
\definecolor{rhoprime}{RGB}{106, 90, 205}
\definecolor{ud}{RGB}{0, 0, 0}

\makeatletter
\pgfkeys{
    /tikz/sharp angle/.code={%
        \pgfsetarrowoptions{sharp >}{#1}%
        \pgfsetarrowoptions{sharp <}{-#1}%
    },
    /tikz/sharp > angle/.code={%
        \pgfsetarrowoptions{sharp >}{#1}%
    },
    /tikz/sharp < angle/.code={%
        \pgfsetarrowoptions{sharp <}{#1}%
    },
    /tikz/sharp protrude/.code=\csname if#1\endcsname\qrr@tikz@sharp@z@-0.05\p@\else\qrr@tikz@sharp@z@\z@\fi,
    /tikz/sharp protrude/.default=true
}

\newdimen\qrr@tikz@sharp@z@
\qrr@tikz@sharp@z@\z@
\pgfarrowsdeclare{sharp >}{sharp >}{%
    \edef\pgf@marshal{\noexpand\pgfutil@in@{and}{\pgfgetarrowoptions{sharp >}}}%
    \pgf@marshal
    \ifpgfutil@in@
    \edef\pgf@tempa{\pgfgetarrowoptions{sharp >}}
    \expandafter\qrr@tikz@sharp@parse\pgf@tempa\@qrr@tikz@sharp@parse
    \else
    \qrr@tikz@sharp@parse\pgfgetarrowoptions{sharp >}and-\pgfgetarrowoptions{sharp >}\@qrr@tikz@sharp@parse
    \fi
    \pgfmathparse{max(\pgf@tempa,\pgf@tempb,0)}%
    \let\qrr@tikz@sharp@max\pgfmathresult
    \pgfmathsetlength\pgf@xa{.5*\pgflinewidth * tan(\qrr@tikz@sharp@max)}%
    \pgfarrowsleftextend{+\pgf@xa}%
    \pgfarrowsrightextend{+\pgf@xa}%
}{%
    \edef\pgf@marshal{\noexpand\pgfutil@in@{and}{\pgfgetarrowoptions{sharp >}}}%
    \pgf@marshal
    \ifpgfutil@in@
    \edef\pgf@tempa{\pgfgetarrowoptions{sharp >}}
    \expandafter\qrr@tikz@sharp@parse\pgf@tempa\@qrr@tikz@sharp@parse
    \else
    \qrr@tikz@sharp@parse\pgfgetarrowoptions{sharp >}and-\pgfgetarrowoptions{sharp >}\@qrr@tikz@sharp@parse
    \fi
    \pgfmathsetlength\pgf@ya{.5*\pgflinewidth * tan(max(\pgf@tempa,\pgf@tempb,0))}%
    \pgfmathsetlength\pgf@xa{-.5*\pgflinewidth * tan(\pgf@tempa)}%
    \pgfmathsetlength\pgf@xb{-.5*\pgflinewidth * tan(\pgf@tempb)}%
    \advance\pgf@xa\pgf@ya
    \advance\pgf@xb\pgf@ya
    \ifdim\pgf@xa>\pgf@xb
    \pgftransformyscale{-1}%
    \pgf@xc\pgf@xb
    \pgf@xb\pgf@xa
    \pgf@xa\pgf@xc
    \fi
    \pgfpathmoveto{\pgfqpoint{\qrr@tikz@sharp@z@}{.5\pgflinewidth}}%
    \pgfpathlineto{\pgfqpoint{\pgf@xa}{.5\pgflinewidth}}%
    \pgfpathlineto{\pgfqpoint{\pgf@ya}{+0pt}}%
    \pgfpathlineto{\pgfqpoint{\pgf@xb}{-.5\pgflinewidth}}%
    \pgfpathlineto{\pgfqpoint{\qrr@tikz@sharp@z@}{-.5\pgflinewidth}}%
    \pgfusepathqfill
}
\pgfarrowsdeclare{sharp <}{sharp <}{%
    \edef\pgf@marshal{\noexpand\pgfutil@in@{and}{\pgfgetarrowoptions{sharp <}}}%
    \pgf@marshal
    \ifpgfutil@in@
    \edef\pgf@tempa{\pgfgetarrowoptions{sharp <}}
    \expandafter\qrr@tikz@sharp@parse\pgf@tempa\@qrr@tikz@sharp@parse
    \else
    \expandafter\qrr@tikz@sharp@parse\pgfgetarrowoptions{sharp <}and-\pgfgetarrowoptions{sharp <}\@qrr@tikz@sharp@parse
    \fi
    \pgfmathparse{max(\pgf@tempa,\pgf@tempb,0)}%
    \let\qrr@tikz@sharp@max\pgfmathresult
    \pgfmathsetlength\pgf@xa{.5*\pgflinewidth * tan(\qrr@tikz@sharp@max)}%
    \pgfarrowsleftextend{+\pgf@xa}%
    \pgfarrowsrightextend{+\pgf@xa}%
}{%
    \edef\pgf@marshal{\noexpand\pgfutil@in@{and}{\pgfgetarrowoptions{sharp <}}}%
    \pgf@marshal
    \ifpgfutil@in@
    \edef\pgf@tempa{\pgfgetarrowoptions{sharp <}}
    \expandafter\qrr@tikz@sharp@parse\pgf@tempa\@qrr@tikz@sharp@parse
    \else
    \expandafter\qrr@tikz@sharp@parse\pgfgetarrowoptions{sharp <}and-\pgfgetarrowoptions{sharp <}\@qrr@tikz@sharp@parse
    \fi
    \pgfmathsetlength\pgf@ya{.5*\pgflinewidth * tan(max(\pgf@tempa,\pgf@tempb,0))}%
    \pgfmathsetlength\pgf@xa{-.5*\pgflinewidth * tan(\pgf@tempa)}%
    \pgfmathsetlength\pgf@xb{-.5*\pgflinewidth * tan(\pgf@tempb)}%
    \advance\pgf@xa\pgf@ya
    \advance\pgf@xb\pgf@ya
    \ifdim\pgf@xa>\pgf@xb
    \pgftransformyscale{-1}%
    \pgf@xc\pgf@xb
    \pgf@xb\pgf@xa
    \pgf@xa\pgf@xc
    \fi
    \pgfpathmoveto{\pgfqpoint{\qrr@tikz@sharp@z@}{.5\pgflinewidth}}%
    \pgfpathlineto{\pgfqpoint{\pgf@xa}{.5\pgflinewidth}}%
    \pgfpathlineto{\pgfqpoint{\pgf@ya}{+0pt}}%
    \pgfpathlineto{\pgfqpoint{\pgf@xb}{-.5\pgflinewidth}}%
    \pgfpathlineto{\pgfqpoint{\qrr@tikz@sharp@z@}{-.5\pgflinewidth}}%
    \pgfusepathqfill
}
\def\qrr@tikz@sharp@parse#1and#2\@qrr@tikz@sharp@parse{\def\pgf@tempa{#1}\def\pgf@tempb{#2}}
\makeatother

\title{Leaving the Hall: explicit formulas for Negu\textcommabelow{t} operators}

\author{Michele D'Adderio}
\address{
    Dipartimento di Matematica \newline \indent
    Università di Pisa \newline \indent
    Pisa, PI, 56127, Italia
}
\email{michele.dadderio@unipi.it}

\author{Giovanni Interdonato}
\address{
    Classe di Scienze \newline \indent
    Scuola Normale Superiore \newline \indent
    Pisa, PI, 56126, Italia
}
\email{giovanni.interdonato@sns.it}

\author{Alessandro Iraci}
\address{
    Facoltà di Ingegneria e Informatica \newline \indent
    Università Pegaso \newline \indent
    Napoli, NA, 80143, Italia
}
\email{alessandro.iraci@unipegaso.it}

\author{Roberto Pagaria}
\address{
    Dipartimento di Matematica \newline \indent
    Università di Bologna \newline \indent
    Bologna, BO, 40126, Italia
}
\email{roberto.pagaria@unibo.it}


\begin{document}

\begin{abstract}
    Recent major breakthroughs in $q,t$-combinatorics include the introduction of the Dyck path algebra $\A_{q,t}$ by Carlsson and Mellit and of the Catalanimals by Blasiak et al., both of which led, among other things, to independent proofs of different extensions of the rational shuffle conjecture of Bergeron et al.

    The first main contribution of this paper is a simple, explicit formula inside the algebra $\A_{q,t}$ for the Negu\textcommabelow{t} operators, yielding a direct, elementary connection between the original operators of the rational shuffle conjecture and the corresponding Catalanimals. Our formula bypasses the elliptic Hall algebra, turning these operators into transparent, workable tools whose action we can compute exactly and efficiently on any symmetric function, not just constants.

    Our second main contribution consists of a series of explicit formulas relating the Negu\textcommabelow{t} operators to the Theta operators introduced by D'Adderio et al. To prove these formulas, we provide an extension of the aforementioned Theta operators to the entire algebra $\A_{q,t}$, allowing us to obtain a series of new combinatorial results. The algebraic computations underlying this extension have been formalized in Lean.

    To showcase the power of our results, we give a proof, also partially formalized in Lean, of the Theta conjecture of D'Adderio et al., first stated in 2019.
\end{abstract}

\maketitle
\tableofcontents

\section{Introduction}

In the 1990s, Garsia and Haiman set out to prove the Schur positivity of the (modified) Macdonald polynomials by showing them to be the bigraded Frobenius characteristic of certain Garsia--Haiman modules \cite{GarsiaHaiman1993GradedRepresentationModel}. Their prediction was confirmed in 2001, when Haiman used the algebraic geometry of the Hilbert scheme to prove that the dimension of their modules is $n!$ \cite{Haiman2001nFactorial}, thus proving the $n!$ theorem. In the course of these developments, it became clear that remarkable connections were to be found between Macdonald polynomial theory and the representation theory of the symmetric group. For example, during their quest for Macdonald positivity, Garsia and Haiman introduced the $\mathfrak{S}_n$-module of \emph{diagonal harmonics}, i.e.\ the coinvariants of the diagonal action of $\mathfrak{S}_n$ on polynomials in two sets of $n$ variables, and they conjectured that its Frobenius characteristic is given by $\nabla e_n$, where $\nabla$ is the \emph{nabla} operator on symmetric functions introduced in \cite{BergeronGarsiaHaimanTesler1999IdentitiesPositivityConjectures}, which acts diagonally on Macdonald polynomials. Haiman proved this conjecture in 2002 \cite{Haiman2002HilbertScheme}.

Over the years, this subject has revealed itself to be extremely fruitful and to have striking connections to other fields of mathematics, including elliptic Hall algebras \cites{SchiffmannVasserot2011EllipticHallAlgebra,BHMPS2023ShuffleAnyLine}, affine Hecke algebras \cite{CarlssonMellit2018ShuffleConjecture}, Springer fibers \cite{Mellit2020SpringerFibers}, the homology of torus knots \cites{GorskyNegut2015RefinedKnotInvariants, Mellit2022HomologyTorusKnots}, and the shuffle algebra of symmetric functions \cite{Negut2014ShuffleAlgebra}. This rich framework is a strength of the field and often leads to theorems whose proofs draw on different theories.

The combinatorial side of this story solidified when Haglund et al.\ formulated their \emph{shuffle conjecture} \cite{HaglundHaimanLoehrRemmelUlyanov2005ShuffleConjecture}: they predicted a combinatorial formula for $\nabla e_n$ in terms of labeled Dyck paths, which are lattice paths using north and east steps going from $(0,0)$ to $(n,n)$ and staying weakly above the line connecting these two points (called the \emph{main diagonal}). This remarkable conjecture, which resisted all attempts at proof for many years, opened the way to several refinements and extensions. Notably, Haglund, Morse, and Zabrocki conjectured a \emph{compositional} refinement of the shuffle conjecture, which also specified all the points where the Dyck paths return to the main diagonal \cite{HaglundMorseZabrocki2012CompositionalShuffleConjecture}. Not long after, Gorsky and Negu\textcommabelow{t} (cf.\ \cites{GorskyNegut2015RefinedKnotInvariants,Negut2014ShuffleAlgebra}) formulated a \emph{rational shuffle conjecture}, whose combinatorial side involved rectangular labeled Dyck paths, i.e.\ labeled lattice paths going from $(0,0)$ to $(m,n)$ and staying weakly above the line connecting these two points. The symmetric function side is obtained by acting on the constant symmetric function $1$ with suitable elements of the elliptic Hall algebra $\mathcal{E}$ of Burban and Schiffmann. Soon after, in \cite{BergeronGarsiaSergelXin2015CompositionalShuffleConjectures} Bergeron et al.\ formulated a \emph{compositional} refinement of the rational shuffle conjecture, at the same time making the symmetric function side more explicit, i.e.\ relying only on basic facts that do not require any understanding of the elliptic Hall algebra $\mathcal{E}$.

The first breakthrough occurred in \cite{CarlssonMellit2018ShuffleConjecture}, where Carlsson and Mellit proved the compositional shuffle conjecture. For their proof, they introduced their algebra $\A_{q,t}$, which turned out to be a powerful tool for proving several other conjectures. For example, in \cite{Mellit2021Rational} Mellit used the algebra $\A_{q,t}$ to prove the compositional rational shuffle conjecture as well.

Around the time of the announcement of the proof of Carlsson and Mellit \cite{CarlssonMellit2018ShuffleConjecture}, Haglund, Remmel, and Wilson formulated their \emph{Delta conjecture} \cite{HaglundRemmelWilson2018DeltaConjecture}: this is actually a pair of conjectures for the symmetric function $\Delta'_{e_{n-k-1}}e_n$ in terms of decorated Dyck paths, where $k$ decorations are placed on either \emph{rises} or \emph{valleys} of the path. The symmetric function operator $\Delta'_f$ (also introduced in \cite{BergeronGarsiaHaimanTesler1999IdentitiesPositivityConjectures}) acts diagonally on the Macdonald polynomials and, in a sense, generalizes $\nabla$. In \cite{DAdderioIraciVandenWyngaerd2021ThetaOperators}, a compositional refinement of the Delta conjecture was proposed using a new tool: the \emph{Theta operators} $\Theta_f$. Finally, the rise version of the compositional Delta conjecture was proved in \cite{DAdderioMellit2022CompositionalDelta}, using once again the $\A_{q,t}$ algebra. Remarkably, the valley version of the Delta conjecture is still wide open (cf.\ \cite{DAdderioIraci2023} for some partial results).

It should be noted that in \cite[Conjecture~9.1]{DAdderioIraciVandenWyngaerd2021ThetaOperators}, the authors tried to merge the two versions of the Delta conjecture into what they called the \emph{Theta conjecture}, suggesting a combinatorial interpretation for $\left. \Theta_{e_k} \Theta_{e_l} \nabla e_{n-k-l} \right\rvert_{q=1}$ in terms of labeled Dyck paths supporting decorations on both rises and valleys, via a formula that specializes to the two versions of the Delta conjecture when $k$ or $l$ is set to $0$. The suggested formula still lacks a dinv statistic, which justifies the specialization $q=1$. Due to its conjectural connection to the coinvariants of the diagonal action of $\mathfrak{S}_n$ on polynomials in two sets of $n$ commuting variables and two sets of $n$ anticommuting variables (cf. \cite[Section 8]{DAdderioIraciVandenWyngaerd2021ThetaOperators}), we consider finding such a statistic an outstanding open problem in $q,t$-combinatorics (cf.\ \cite{IraciNadeauVandenWyngaerd2024Smirnov} for some partial progress).

Soon after the proof of the rise version of the compositional Delta conjecture in \cite{DAdderioMellit2022CompositionalDelta}, Blasiak et al.\ provided an independent proof of the rational shuffle conjecture (which encompasses the shuffle conjecture) in \cite{BHMPS2023ShuffleAnyLine}, extending these formulas to labeled Dyck paths staying weakly above ``any line''. This initiated a remarkable series of articles by the same authors \cites{BHMPS2023ProofExtendedDelta,BHMPS2024LLT,BHMPS2025LoehrWarrington,BHMPS2025RaisingOperatorFormula}, which settle and extend several outstanding conjectures in the field. The proofs in \cite{BHMPS2023ShuffleAnyLine} are divided into two steps. The first step shows that the combinatorial side of the formulas coincides with a certain Hall-Littlewood series. The second step consists of proving that this Hall-Littlewood series matches the symmetric-function side of the earlier conjectures: to achieve this, the authors have to use an action of the elliptic Hall algebra $\mathcal{E}$, via an implicit isomorphism with the shuffle algebra studied by Negu\textcommabelow{t} in \cite{Negut2014ShuffleAlgebra}. In particular, it is shown that certain elements $\D_\gamma$ in $\mathcal{E}$ (called Negu\textcommabelow{t} elements in \cites{BHMPS2023ShuffleAnyLine,BHMPS2023ProofExtendedDelta} and \cite{BenDaliBonzomDoulega2026PathOperators}) act as the relevant operators on the constant symmetric function $1$. Generalizations of these Hall-Littlewood series are called \emph{Catalanimals} in \cites{BHMPS2024LLT,BHMPS2025LoehrWarrington,BHMPS2025RaisingOperatorFormula}.

The first main contribution of the present article is to connect the two approaches to the rational shuffle conjecture that we recalled above: we provide a simple, explicit formula on symmetric functions for the operators corresponding to the Negu\textcommabelow{t} elements $\D_\gamma$, thereby identifying them inside the $\A_{q,t}$ algebra of Carlsson and Mellit. The advantage of our formula is threefold: (1) since we reprove in an elementary way that $\D_\gamma$ acting on $1$ gives the Hall-Littlewood series (Catalanimal) obtained in \cite{BHMPS2023ShuffleAnyLine}, our results connect this series to the original operators of the rational shuffle conjecture, completely bypassing the elliptic Hall algebra $\mathcal{E}$; (2) our formula enables the exact evaluation of $\D_\gamma$ on arbitrary symmetric functions (rather than constants alone), yielding an algorithm (computable, e.g., in SageMath) requiring a number of operator steps that is linear, rather than exponential, in the length of $\gamma$; (3) our formula allows us to commute the action of $\D_\gamma$ with other operators, yielding new formulas.

The second main contribution of this paper is related to point (3) above: we derive explicit commutation relations between the Negu\textcommabelow{t} operators $\D_\gamma$ and the Theta operators $\Theta_{e_k}$ introduced in \cite{DAdderioIraciVandenWyngaerd2021ThetaOperators}, and exploit these relations to prove new formulas about both. To achieve this, we first extend the definition of the Theta operators to the entire algebra $\A_{q,t}$, broadening their scope; in this context, the Theta operators are easier to compute and manipulate, allowing us to obtain a series of new combinatorial results in the form of explicit formulas. The computations used to define this extension have been formalized in Lean 4 using Harmonic's \emph{Aristotle} \cite{Achim2025Aristotle}, an AI system for formal verification.

As a notable application of our new formulas, we prove the Theta conjecture \cite[Conjecture~9.1]{DAdderioIraciVandenWyngaerd2021ThetaOperators} using (among other tools) a bijective argument that we also formalized in Lean 4 using \emph{Aristotle}.

Our formulas potentially have many more applications to open problems in the literature. For example, the decorated framework of the Delta conjectures and the rational framework of the rational shuffle conjecture have recently been unified in \cite{IraciPagariaPaolini2026FallingStars}, where a decorated version of the rational shuffle theorem was first conjectured and then proved, thanks to techniques introduced in \cite{GillespieGorskyGriffin2025Skewing}.

Our formulas can be adapted to this context as well, and in principle they can be used to prove several of the conjectures appearing in these papers. For the moment, we consider these applications outside the scope of this paper and leave them for future work.

This paper is structured as follows. In \Cref{sec:symmetric_functions} we introduce the symmetric function tools we need and in \Cref{sec:Aqt} we do the same with the Dyck path algebra $\A_{q,t}$; in \Cref{sec:Dgamma}, we give an alternative definition for the $\D_\gamma$ operators as elements of (an extension of) $\A_{q,t}$, and prove that they satisfy all the properties we expect; in \Cref{sec:Theta}, we define an endomorphism $\Theta$ of $\A_{q,t}$, show that it extends the usual $\Theta_f$ operators on symmetric functions, and derive its commutation relations with the $\D_\gamma$ operators; finally, in \Cref{sec:theta_conjecture}, we demonstrate the power of these new operators by proving the Theta conjecture.

\section{Symmetric function tools}
\label{sec:symmetric_functions}

We denote by $\Lambda$ the graded algebra of symmetric functions with coefficients in $\mathbb{Q}(q,t)$, and by $\<\, , \>$ the \emph{Hall scalar product} on $\Lambda$, defined by declaring that the Schur functions form an orthonormal basis.

The standard bases of symmetric functions that will appear in our calculations are the monomial $\{m_\lambda\}_{\lambda}$, complete $\{h_{\lambda}\}_{\lambda}$, elementary $\{e_{\lambda}\}_{\lambda}$, power $\{p_{\lambda}\}_{\lambda}$, and Schur $\{s_{\lambda}\}_{\lambda}$ bases.

For $f \in \Lambda$, we denote by $f^\perp$ the operator adjoint to multiplication by $f$ with respect to the Hall scalar product, that is, for every $g, h \in \Lambda$, we have $\< f^\perp g, h \> = \< g, fh \>$.

For a partition $\mu \vdash n$, we denote by \[ \Ht_\mu \coloneqq \Ht_\mu[X] = \Ht_\mu[X; q,t] = \sum_{\lambda \vdash n} \widetilde{K}_{\lambda \mu}(q,t) s_{\lambda} \] the \emph{(modified) Macdonald polynomials}, where \[ \widetilde{K}_{\lambda \mu} \coloneqq \widetilde{K}_{\lambda \mu}(q,t) = K_{\lambda \mu}(q,1/t) t^{n(\mu)} \] are the \emph{(modified) Kostka coefficients} (see \cite[Chapter~2]{Haglund2008Book} for more details).

Macdonald polynomials form a basis of the algebra of symmetric functions $\Lambda$. This basis is a modification of the one introduced by Macdonald \cite{Macdonald1995Book}.

If we identify the partition $\mu$ with its Young diagram, i.e.\ with the collection of cells $\{(i,j)\mid 1\leq i\leq \mu_j, 1\leq j\leq \ell(\mu)\}$, then for each cell $c\in \mu$ we define the \emph{arm}, \emph{leg}, \emph{co-arm}, and \emph{co-leg} (denoted, respectively, by $a_\mu(c), l_\mu(c), a_\mu'(c), l_\mu'(c)$) to be the numbers of cells in $\mu$ that are strictly to the right, below, to the left, and above $c$ in $\mu$, respectively (see \Cref{fig:notation}).

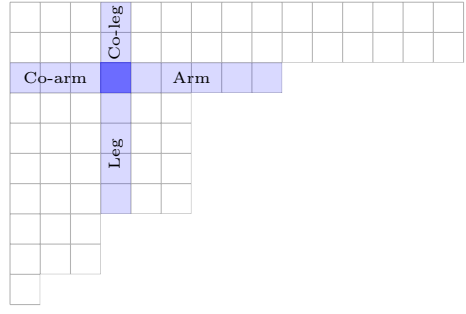
\begin{figure}
    \centering
    \begin{tikzpicture}[scale=0.4]
        \draw[gray,opacity=.6](0,0) grid (15,10);
        \fill[white] (1,-0.1)|-(3,1) |- (6,3) |- (9,7) |- (15.1,8) |- (1,-.1);
        \fill[blue, opacity=.15] (0,7) rectangle (9,8) (3,3) rectangle (4,10);
        \fill[blue, opacity=.5] (3,7) rectangle (4,8);
        \draw (6,7.5) node {\tiny{Arm}} (3.5,5) node[rotate=90] {\tiny{Leg}} (3.5, 9) node[rotate = 90] {\tiny{Co-leg}} (1.5,7.5) node {\tiny{Co-arm}} ;
    \end{tikzpicture}
    \caption{Arm, leg, co-arm, and co-leg of a cell of a partition.}
    \label{fig:notation}
\end{figure}

For every partition $\mu$, we define the following constants:

\[ B_{\mu} \coloneqq B_{\mu}(q,t) = \sum_{c \in \mu} q^{a_{\mu}'(c)} t^{l_{\mu}'(c)}, \qquad
    \Pi_{\mu} \coloneqq \Pi_{\mu}(q,t) = \prod_{c \in \mu / (1)} (1-q^{a_{\mu}'(c)} t^{l_{\mu}'(c)}). \]

Notice that in the definition of $\Pi_\mu$, the exponents $a_{\mu}'$ and $l_{\mu}'$ are evaluated with respect to the shape $\mu$, not $\mu / (1)$; the latter notation only indicates that the cell $(1,1)$ should be omitted from the product (otherwise the product would be $0$).

We will make extensive use of the \emph{plethystic notation} (cf.\ \cite[Chapter~1]{Haglund2008Book}).

\begin{definition}
    For every expression $X$ for which plethystic substitution is defined, we define the \emph{plethystic exponential} $\Exp[X]$  as
    \[\Exp[X] \coloneqq \sum_{n\geq 0} h_n[X].\]
\end{definition}

\begin{definition}
    For every expression $Y$ for which plethystic substitution is defined, we define the \emph{(plethystic) translation operator} $\mathcal{T}_Y$ by setting, for every $F[X] \in \Lambda$,
    \[ \mathcal{T}_Y F[X] \coloneqq F[X+Y]. \]
\end{definition}

\begin{definition}
    For every expression $Z$ for which plethystic substitution is defined, we define the \emph{(plethystic) multiplication operator} $\mathcal{P}_Z$ by setting, for every $F[X]\in \Lambda$,
    \[ \mathcal{P}_Z F[X] \coloneqq \Exp[XZ] F[X]. \]
\end{definition}

These operators are related by the following identity, which is straightforward to check:
\begin{equation}
    \label{eq:TPExp}
    \mathcal{T}_Y \mathcal{P}_Z= \Exp[YZ] \mathcal{P}_Z \mathcal{T}_Y.
\end{equation}

We recall the definitions of the following linear operators on $\Lambda$.

\begin{definition}[{\cite[Equation~(3.11)]{BergeronGarsia1999ScienceFiction}}]
    \label{def:nabla}
    We define the linear operator $\nabla \colon \Lambda \rightarrow \Lambda$ on the eigenbasis of Macdonald polynomials as \[ \nabla \Ht_\mu = e_{\lvert \mu \rvert}[B_\mu] \Ht_\mu. \]
\end{definition}

\begin{definition}
    \label{def:pi}
    We define the linear operator $\mathbf{\Pi} \colon \Lambda \rightarrow \Lambda$ on the eigenbasis of Macdonald polynomials as \[ \mathbf{\Pi} \Ht_\mu = \Pi_\mu \Ht_\mu \] where we conventionally set $\Pi_{\varnothing} \coloneqq 1$.
\end{definition}

\begin{definition}
    \label{def:delta}
    For $f \in \Lambda$, we define the linear operators $\Delta_f, \Delta'_f \colon \Lambda \rightarrow \Lambda$ on the eigenbasis of Macdonald polynomials as \[ \Delta_f \Ht_\mu = f[B_\mu] \Ht_\mu, \qquad \qquad \Delta'_f \Ht_\mu = f[B_\mu-1] \Ht_\mu. \]
\end{definition}

Observe that, on the vector space of homogeneous symmetric functions of degree $n$, denoted by $\Lambda^{(n)}$, the operator $\nabla$ equals $\Delta_{e_n}$. Also note that \[ \mathbf{\Pi} = \sum_{n=0}^\infty (-1)^n \Delta'_{e_n}. \] 

We will use the symmetric functions $E_{n,k}$ (cf.\ \cite[Section~1.3]{DAdderioRomero2023ThetaIdentities}), which can be defined as
\[E_{n,k}[X]=q^k\sum_{r=0}^k q^{\binom{r}{2}} \qbinom{k}{r}(-1)^re_n\left[X\frac{1-q^{-r}}{1-q}\right],\]
where $\qbinom{k}{r}$ denotes the usual $q$-binomial coefficient.

\begin{remark}
    \label{rem:Enk}
    It is well known (cf.\ \cite[Theorem 7.4]{DAdderioRomero2023ThetaIdentities}) that $\nabla E_{n,n}=\widetilde{H}_{(n)}$, and that $\widetilde{H}_{(n)}$ depends only on $q$ (and not on $t$), so that $\nabla E_{n,n} \rvert_{q=1} = \widetilde{H}_{(n)} \rvert_{q=1}$ equals $e_1^n$, like any modified Macdonald polynomial of degree $n$ at $q=t=1$.
\end{remark}

From now on, let $M \coloneqq (1-q)(1-t)$.

\begin{definition}[{\cite[Equation~(28)]{DAdderioIraciVandenWyngaerd2021ThetaOperators}}]
    \label{def:theta}
    For any symmetric function $f \in \Lambda^{(n)}$, we define the \emph{Theta operators} on $\Lambda$ as follows: for every $F \in \Lambda^{(m)}$ we set
    \begin{equation*}
        \Theta_f F  \coloneqq
        \left\{\begin{array}{ll}
            0                                                           & \text{if } n \geq 1 \text{ and } m=0 \\
            f \cdot F                                                   & \text{if } n=0 \text{ and } m=0      \\
            \mathbf{\Pi} f \left[\frac{X}{M}\right] \mathbf{\Pi}^{-1} F & \text{otherwise}
        \end{array}
        \right. ,
    \end{equation*}
    and we extend the definition by linearity to all $f, F \in \Lambda$.
\end{definition}

It is clear that $\Theta_f$ is linear. In addition, if $f$ is homogeneous of degree $k$, then $\Theta_f$ raises degree by $k$:
\[\Theta_f \Lambda^{(n)} \subseteq \Lambda^{(n+k)} \qquad \text{ for } f \in \Lambda^{(k)}. \]

Finally, we need to refer to \cite[Algorithm~4.1]{BergeronGarsiaSergelXin2015CompositionalShuffleConjectures} (see also \cite[Definition~1.1, Theorem~2.5]{BergeronGarsiaSergelXin2016PlethysticOperators}).

\begin{definition}
    Let $m, n > 0$. Let $a,b,c,d \in \mathbb{N}$ be such that $a+c=m$, $b+d=n$, and $ad-bc = \gcd(m,n)$. We recursively define $Q_{m,n}$ as an operator on $\Lambda$ by \[ Q_{m,n} = \frac{1}{M} \left( Q_{c,d} Q_{a,b} - Q_{a,b} Q_{c,d} \right), \] with the base cases \[ Q_{1,0} = \D_0 = \mathsf{id} - M \Delta_{e_1} \quad \text{ and } \quad Q_{0,1} = - e_1^\bullet \] (where $f^\bullet$ denotes multiplication by $f$).
\end{definition}

\begin{definition}
    \label{def:hall-algebra-op}
    For a coprime pair $(a,b)$ and $F \in \Lambda^{(d)}$, we define $F_{ad,bd}$ as follows. Let \[ f = \sum_{\lambda \vdash d} c_\lambda(q,t) \left( \frac{qt}{qt-1} \right)^{\ell(\lambda)} h_\lambda \left[ \frac{1-qt}{qt} X \right]. \] Then, we define \[ F_{ad,bd} \coloneqq \sum_{\lambda \vdash d} c_\lambda(q,t) \prod_{i=1}^{\ell(\lambda)} Q_{\lambda_i a, \lambda_i b}, \]
    which is an operator acting on symmetric functions.
\end{definition}

As conjectured in \cite{BergeronGarsiaSergelXin2016PlethysticOperators} and proved in \cite{Mellit2021Rational}, $(-1)^{d(b+1)} f_{ad,bd}$ coincides with the operator 

\begin{remark}
    \label{rmk:eha-notation}
    Comparing the normalization of the $Q_{m,n}$-operators in Algorithm~3.1 of~\cite{BergeronGarsiaSergelXin2015CompositionalShuffleConjectures} with the Schiffmann--Vasserot realization of the elliptic Hall algebra used in~\cite{BHMPS2023ShuffleAnyLine}, one obtains, for $F \in \Lambda^{(d)}$,
    \[
        F_{da,db}= (-1)^{d(b+1)} F[-MX^{a,b}],
    \]
    where $F[-MX^{a,b}]$ is the operator arising from the shuffle algebra \cite{Negut2014ShuffleAlgebra}.

    Indeed, each homogeneous generator of degree $r$ contributes a factor $(-1)^{r(b+1)}$. Since every monomial occurring in the expansion of $F$ has total degree $d$, the resulting factor is
    \[
        \prod_i (-1)^{\lambda_i(b+1)}
        =(-1)^{(b+1)\sum_i\lambda_i}
        =(-1)^{d(b+1)}.
    \]
\end{remark}

Throughout the rest of this paper, we will use $F[-MX^{a,b}]$ as notation for these expressions, unless otherwise specified.

\section{The Dyck path algebra}
\label{sec:Aqt}

The Carlsson--Mellit algebra, or Dyck path algebra $\A_{q,t}$, was first introduced in \cite{CarlssonMellit2018ShuffleConjecture} as a groundbreaking tool for proving the shuffle theorem. Its action on the space of symmetric functions $\Lambda$ was later modified in \cite{Mellit2021Rational} to prove the rational shuffle theorem as well. In this section, we first recall the definition and some basic facts about it, and then extend it to accommodate elements of negative degree and define an action on Laurent polynomials with coefficients in $\Lambda$.

\subsection{The algebra \texorpdfstring{$\A_q$}{Aq}}

\begin{definition}[{\cite[Definition~3.1]{CarlssonMellit2018ShuffleConjecture}}]
    \label{def:Aq}
    The algebra $\A_q$ (over $\Q(q)$) is the path algebra of the quiver with vertex set $\N$, arrows $d_+$ from $k$ to $k+1$, arrows $d_-$ from $k+1$ to $k$, and loops $T_1, \dots, T_{k-1}$ from $k$ to $k$ (for $k > 0$) subject to the following relations:
    \begin{align}
        (T_i-1)(T_i+q) = 0 & \label{eq:skein} \tag{S} \\
        T_i T_{i+1} T_i = T_{i+1} T_i T_{i+1} & \label{eq:braid} \tag{B} \\
        T_i T_j = T_j T_i & \qquad \text{ for } \lvert i-j \rvert > 1 \label{eq:commuting} \tag{C} \\
        d_- T_i = T_i d_- & \qquad \text{ for } 2 \leq i \leq k-2 \label{eq:r1} \tag{R1} \\
        d_+ T_i = T_{i+1} d_+ & \label{eq:r2} \tag{R2} \\
        T_1 d_+^2 = d_+^2 & \label{eq:r3} \tag{R3} \\
        d_-^2T_{k-1} = d_-^2 & \label{eq:r4} \tag{R4} \\
        d_- [d_+,d_-] T_{k-1} = q [d_+,d_-] d_- & \qquad \text{ for } k \geq 2 \label{eq:r5} \tag{R5} \\
        T_1 [d_+,d_-] d_+ = q d_+ [d_+,d_-] & \qquad \text{ for } k \geq 1 \label{eq:r6} \tag{R6}
    \end{align}
    where in each identity $k$ denotes the index of the vertex where the corresponding paths begin.
    
    The idempotent corresponding to the empty path from vertex $k$ will be denoted by $\epsilon_k$.
\end{definition}

\begin{definition}
    We use the following shorthands from \cite{CarlssonMellit2018ShuffleConjecture}. If $i \leq j$, set
    \begin{align*}
        & T_{i \nearrow j} = T_i T_{i+1} \cdots T_{j-1}, && \quad T_{j \searrow i} = T_{j-1} T_{j-2} \cdots T_{i}, && \quad T_{i  \nearrow i} = T_{i \searrow i} = \mathsf{id}, \\
        & T_{i \nearrow j}^* = T_i^{-1} T_{i+1}^{-1} \cdots T_{j-1}^{-1}, && \quad T_{j \searrow i}^* = T_{j-1}^{-1} T_{j-2}^{-1} \cdots T_{i}^{-1}, && \quad T_{i \nearrow i}^* = T_{i \searrow i}^* = \mathsf{id}.
    \end{align*}
    If $i > j$, we set $T_{i \nearrow j} \coloneqq T_{i \searrow j}^*$ and $T_{j \searrow i} \coloneqq T_{j \nearrow i}^*$.
\end{definition}

\begin{proposition}[{\cite[Lemma~5.5]{CarlssonMellit2018ShuffleConjecture}}]
    For $k > 0$, define loops $y_1, \dots, y_k$ from $k$ to $k$ as
    \[ y_1 \coloneqq \frac{[d_+,d_-]}{q^{k-1}(q-1)}T_{k \searrow 1} \quad \text{ and } \quad y_{i+1} \coloneqq q T_i^{-1} y_i T_i^{-1}. \]
    Then the following identities hold.
    \begin{align}
        [y_i, T_j] & = 0 \quad \text{ for } i \neq j, j+1  \label{eq:yT_comm} \\
        [d_-, y_i] & = 0 \label{eq:yd-_comm} \\
        [y_i, T^*_{i+1 \searrow 1} d_+] & = 0 \label{eq:yd+_comm} \\
        [y_i,y_j] & = 0. \label{eq:yy_comm}
    \end{align}
\end{proposition}

\subsection{The algebra \texorpdfstring{$\A_{q,t}$}{Aqt}}

\begin{definition}[{\cite[Definition~3.9]{Mellit2021Rational}}]
    \label{def:Aqt}
    The algebra $\A_{q,t}$ (over $\Q(q,t)$) is the quotient of the free product of the algebras $\A_q$ and $\A_{q^{-1}}$ by the relations given below. To distinguish elements of the second algebra from the corresponding elements of the first algebra, we mark the former with $*$. Set $z_i = y_i^*$, and $\z_i = z_i/qt$. The relations are
    \[ \epsilon_k^* = \epsilon_k, \quad T_i^* = T_i^{-1}, \quad d_-^* = d_-, \]
    and
    \begin{align}
        \z_{i+1} d_+ & = d_+ \z_i \label{eq:Q1} \tag{Q1} \\
        y_{i+1} d_+^* & = d_+^* y_i \label{eq:Q1*} \tag{Q1*} \\
        \z_1 d_+ & = - q^k y_1 d_+^*. \label{eq:Q2} \tag{Q2}
    \end{align}
\end{definition}

Explicitly, the $z_i$ are defined as
\[ z_1 \coloneqq \frac{q^k [d_+^*,d_-]}{(1-q)}T_{k \searrow 1}^* \quad \text{ and } \quad z_{i+1} \coloneqq q^{-1} T_i z_i T_i. \]

\begin{remark}
    The algebra $\A_{q,t}$ is $\mathbb{Z}^2$-graded, with
    \[ \deg(d_+) = (0,1), \quad \deg(d_+^*) = (1,0), \quad \deg(d_-) = (0,0), \quad \deg(T_i) = (0,0) \; \forall i. \]
    We will refer to the $\mathbb{Z}^2$-grading as \emph{bidegree}, and to the restriction to the second component of the bigrading simply as \emph{degree}.
\end{remark}

\subsection{The action on Laurent polynomials}

For $k \in \N$, we set 
\[ V_k \coloneqq \Lambda \otimes \Q(q,t)[y_1^{\pm 1}, \dots, y_k^{\pm 1}], \quad V \coloneqq \bigoplus_{k=0}^\infty V_k. \]

Notice that this differs from the definition in \cites{CarlssonMellit2018ShuffleConjecture, Mellit2021Rational} because we consider Laurent polynomials in $y_1, \dots, y_k$ rather than just polynomials. 

Given a Laurent monomial $m$ in a given set of variables, we denote by $\langle m \rangle B$ the coefficient of $m$ in the Laurent series $B$. For example, $\langle x^0\rangle f(x)$ denotes the constant term of the Laurent series $f(x)$ in the variables $x_1,x_2,\dots$. The algebra $\A_{q,t}$ acts on $V$ as follows.

\begin{proposition}[{\cite[Propositions~3.3 and 3.4]{Mellit2021Rational}}]
    There is an action of $\A_{q,t}$ on $V$ defined as follows: for any $F\in V_k$,
    \begin{align}
        T_i F & = \frac{(q-1) y_i F + (y_{i+1} - q y_i) s_i F}{y_{i+1} - y_i}, \\
        d_- F & = \langle y_k^0 \rangle F[X - (q-1) y_k;y_1,\ldots,y_k] \Exp[-y_k^{-1} X], \\
        d_+ F & = -T_{1 \nearrow k+1}\; (y_{k+1} F[X+(q-1) y_{k+1}; y_1, \ldots, y_k]), \\
        d_+^* F & = \eta\left(F[X + (q-1) y_{k+1}; y_1, \ldots, y_k]\right),
    \end{align}
    where $s_i$ is the operator that interchanges the variables $y_i$ and $y_{i+1}$, and $\eta$ is the operator that sends $y_i$ to $y_{i+1}$ for $i<k+1$ and $y_{k+1}$ to $t y_1$. Moreover, the operators $y_i \in \A_q$ act precisely by multiplication by $y_i$.

    In particular, $T_i, d_-, d_+$ define an action of $\A_q$, while $T_i^{-1}, d_-, d_+^*$ define an action of $\A_{q^{-1}}$.
\end{proposition}

\begin{theorem}[{\cite[Theorem~3.3]{Mellit2021Rational}}]
    \label{thm:annihilator}
    The annihilator of $V$ as an $\A_{q,t}$-module is generated by the relations\footnote{Equation \eqref{eq:I2} in \cite{Mellit2021Rational} has the factor $q^k$ missing. The corrected equation reported here can be found in Theorem~3.6 of the arXiv version.}:
    \begin{align}
        (d_- d_+^* -1) d_+^{*k} \epsilon_0 & = 0 \label{eq:I1} \tag{I1} \\
        (d_+ + q^k y_1 d_+^* ) d_+^{*k} \epsilon_0 & = 0. \label{eq:I2} \tag{I2}
    \end{align}
\end{theorem}

The following statement is a generalization of \cite[Proposition~3.24]{Mellit2021Rational}.

\begin{proposition}
    \label{prop:tau}
    The operator
    \[ \tau(u) F = F[X+u] \Exp\left[- u \sum_{i=1}^k y_i^{-1}\right] \]
    extends the translation operator $\mathcal{T}_{u}$ to $V$. Moreover, it commutes with $y_i$ and $d_-$, and satisfies
    \[ \tau(u)\ d_+^* = \Exp[-u/y_1] d_+^* \tau(u). \]
\end{proposition}

\begin{proof}
    The commutation with $y_i$ is clear, as $\tau(u)$ does not act on the constants.
    Notice that, for every $F \in V_k$ with $k \geq 1$,
    \[ d_- F = \langle y_k^0 \rangle \mathcal{P}_{-\frac{1}{y_k}} \mathcal{T}_{-(q-1)y_k}F, \]
    hence
    \begin{align*}
        \tau(u)\ d_- F & = \Exp\left[- u \sum_{i=1}^{k-1} y_i^{-1}\right] \mathcal{T}_{u} d_- F \\
        & = \Exp\left[- u \sum_{i=1}^{k-1} y_i^{-1}\right] \mathcal{T}_{u} \langle y_k^0 \rangle\mathcal{P}_{-\frac{1}{y_k}} \mathcal{T}_{-(q-1)y_k} F \\
        & = \langle y_k^0 \rangle\mathcal{P}_{-\frac{1}{y_k}} \mathcal{T}_{-(q-1)y_k} \Exp\left[- u \sum_{i=1}^{k-1} y_i^{-1}\right] \Exp \left[-u y_k^{-1}\right] \mathcal{T}_{u} F && \text{by \eqref{eq:TPExp}} \\
        & = d_- \Exp\left[- u \sum_{i=1}^{k} y_i^{-1}\right] \mathcal{T}_{u} F \\
        & = d_- \tau(u)\ F,
    \end{align*}
    as desired. Moreover, since
    \[ d_+^* F = \eta(\mathcal{T}_{(q-1)y_{k+1}} F), \]
    $\mathcal{T}_{u}$ commutes with $d_+^*$, and the commutation of $\tau(u)$ with $d_+^*$ follows.
\end{proof}

We will need the following technical lemma.

\begin{lemma}
    \label{lemma:yz_commutation}
    We have the identity \[ \z_1 y_1 = y_1 (d_+^* d_- + q^{1-k} z_1 T_{1 \nearrow k}) T_{k \searrow 1} \] as operators on $V_k$. In particular, when $k=1$, we have \[ \z_1 y_1 = y_1 (d_+^* d_- + z_1). \]
\end{lemma}

\begin{proof}
    We have
    \begin{align*}
        \z_1 y_1 & = \z_1 \frac{d_+ d_- - d_- d_+}{q^{k-1}(q-1)} T_{k \searrow 1} \\
        & = \frac{1}{q^{k-1} (q-1)} \left( \z_1 d_+ d_- - \z_1 d_- d_+ \right) T_{k \searrow 1} \\
        & = \frac{1}{q^{k-1} (q-1)} \left( \z_1 d_+ d_- - d_- \z_1 d_+ \right) T_{k \searrow 1}\\
        & = \frac{1}{q^{k-1} (q-1)} \left( -q^{k-1} y_1 d_+^* d_- + q^k d_- y_1 d_+^* \right) T_{k \searrow 1}\\
        & = \frac{1}{1-q} y_1 \left( d_+^* d_- - q d_- d_+^* \right) T_{k \searrow 1}\\
        & = \frac{1}{1-q} y_1 \left( d_+^* d_- - q d_+^* d_- + q d_+^* d_- - q d_- d_+^* \right) T_{k \searrow 1}\\
        & = \frac{1}{1-q} y_1 \left( (1-q) d_+^* d_- + q (d_+^* d_- - d_- d_+^*) \right) T_{k \searrow 1}\\
        & = y_1 (d_+^* d_- + q^{1-k} z_1 T_{1 \nearrow k}) T_{k \searrow 1},
    \end{align*}
    as expected.
\end{proof}

Although this is not needed for our purposes, it is convenient to make the negative powers of the $y_i$ explicit in the algebra $\A_{q,t}$. This can be achieved in the following way.

\begin{definition}
    \label{def:Aqt_Y}
    We define $\A_{q,t}^\pm$ to be the algebra obtained from $\A_{q,t}$ by adding a loop $Y$ from $k$ to $k$ for $k > 0$, subject to the relation $Y y_1 = y_1 Y = \epsilon_k$.
\end{definition}

The action of $\A_{q,t}$ on $V$ can be extended to $\A_{q,t}^{\pm}$ by setting $Y F = y_1^{-1} F$. 
\Cref{thm:annihilator} extends to $\A_{q,t}^\pm$ as well, since the action of $Y$ is invertible.

\section{The \texorpdfstring{$\D_\gamma$}{Dgamma} operators}
\label{sec:Dgamma}

The $\D_\gamma$ operators arise from certain distinguished elements of the shuffle algebra \cite[Proposition~6.1]{Negut2014ShuffleAlgebra}, which are viewed as operators on symmetric functions as in \cite[Equation~(52)]{BHMPS2023ShuffleAnyLine}.

In their seminal series of papers, Blasiak et al.\ showed that the action of these elements on the constant symmetric function $1$ produces combinatorially interesting symmetric functions. These results eventually led to proofs of the shuffle theorem under any line \cite{BHMPS2023ShuffleAnyLine}, the extended Delta conjecture \cite{BHMPS2023ProofExtendedDelta}, and other impressive results.

However, their results rely on an implicit isomorphism between the shuffle algebra and the positive part of the elliptic Hall algebra $\mathcal{E}$, whose image can be computed only on special elements satisfying additional properties.

In this section, we give an alternative definition of these operators as elements of $\A_{q,t}^\pm$ and show that the two definitions are equivalent. The advantage of our alternative definition is that it is completely explicit, computationally efficient, and can be applied to any symmetric function.

We also reprove some properties of these operators from scratch, bypassing the shuffle algebra entirely.

\subsection{An explicit formula for the \texorpdfstring{$\D_\gamma$}{Dgamma} operators}

Throughout this section, let $\gamma = (\gamma_1, \dots, \gamma_\ell) \in \bigcup_{n > 0} \mathbb{Z}^n$ with $\ell(\gamma) = \ell$. For any such $\gamma$, define $\gamma_+ = (\gamma_1, \dots, \gamma_\ell + 1)$ and $\gamma_- = (\gamma_1 - 1, \dots, \gamma_\ell)$.

\begin{definition}
    \label{def:D_gamma}
    We define the operator $\D_\gamma \colon \Lambda \rightarrow \Lambda$ as \[ \D_\gamma F = d_- (-y_1)^{\gamma_1-1} \z_1 (-y_1)^{\gamma_2} \cdots \z_1 (-y_1)^{\gamma_\ell} \z_1 d_+ F. \] 
\end{definition}

\begin{remark}
    Even when $\gamma \not \geq 0$, the expression in \Cref{def:D_gamma} is well-defined using only the action of $\A_{q,t}$ on $V$; thus, this construction does not rely on \Cref{def:Aqt_Y}.
    
    However, if negative powers of $y_1$ occur, the element $\D_\gamma$ itself can be viewed only as an 
    element of $\A_{q,t}^\pm$, not as an element of $\A_{q,t}$.

    Notice that, using \Cref{lemma:yz_commutation}, we can rewrite the operator $\D_\gamma$ as \[ \D_\gamma F = d_- (-y_1)^{\gamma_1} (d_+^* d_- + z_1) (-y_1)^{\gamma_2} \cdots (d_+^* d_- + z_1) (-y_1)^{\gamma_\ell} d_+^* F, \]
    so it is an element of $\A_{q,t}$ whenever $\gamma \geq 0$, despite $\gamma_1 - 1$ appearing as an exponent in the definition.
\end{remark}

Our goal is to show that these operators coincide with the $\D_\gamma$ operators from \cite{BHMPS2023ShuffleAnyLine}. We do so by proving that they satisfy the same recurrence relations and that these relations completely determine the operators. In particular, when $\gamma = (m)$ has a single part, the operator $\D_m$ agrees with the one defined in \cite{BHMPS2023ShuffleAnyLine} and differs by a sign $(-1)^m$ from the corresponding operator defined in \cite{GarsiaHaimanTesler1999ExplicitPlethysticFormulas}.
 
Given $\alpha=(\alpha_1,\alpha_2,\dots,\alpha_r) \in \Z^r$ and $\beta=(\beta_1,\beta_2,\dots,\beta_s) \in \Z^s$, we denote by $\alpha\beta$ the concatenation
\[ \alpha\beta \coloneqq (\alpha_1,\alpha_2,\dots,\alpha_r,\beta_1,\beta_2,\dots,\beta_s)\in \Z^{r+s}. \]

\begin{theorem}
    \label{thm:dgamma-recursion}
    For $F\in\Lambda$, $m\in\Z$, and nonempty integer sequences $\alpha$ and $\beta$, the operators $\D_\gamma$ satisfy the relations
    \begin{enumerate}
        \item \label{dop:onepart} $\D_{(m)} F = \D_m F = (-1)^m \langle z^m \rangle \Exp[-zX] F[X + M/z]$ for $m \in \mathbb{Z}$;
        \item \label{dop:neg} $\D_{\alpha (m)} F = 0$ if $m < -\deg(F)$;
        \item \label{dop:recursion} $\D_{\alpha \beta} F = \D_\alpha \D_\beta F + qt \D_{\alpha_+ \beta_-} F$.
    \end{enumerate}
\end{theorem}

\begin{proof}
    First, we check \eqref{dop:onepart}. We have
    \begin{align*}
        d_- (-y_1)^{m-1} \z_1 d_+ F & = d_- (-y_1)^m d_+^* F \\
        & = d_- (-y_1)^m F[X + (q-1)ty_1] \\
        & = \langle y_1^0 \rangle \Exp[-y_1^{-1} X] (-y_1)^m F[X + (q-1)ty_1 - (q-1)y_1] \\
        & = \langle y_1^0 \rangle \Exp[-y_1^{-1} X] (-y_1)^m F[X + M y_1] \\
        (z = y_1^{-1}) & = \langle z^0 \rangle \Exp[-zX] (-z)^{-m} F[X + M/z] \\
        & = (-1)^m \langle z^m \rangle \Exp[-zX] F[X + M/z] \\
        & = \D_{m} F
    \end{align*}
    as desired.

    Now we check \eqref{dop:neg}. If $\gamma$ has one part, the claim follows from \eqref{dop:onepart}. Otherwise, using the alternative definition, it is enough to prove that $(d_+^* d_- + z_1) (-y_1)^m d_+^* F = 0$ if $m < - \deg(F)$.
    Using \eqref{eq:Q2} and \Cref{lemma:z1y1a} applied to $d_+^* F$, the contribution of the $d_+^* d_-$ term is
    \[ 
        \left. \left(\langle w_2^{-m} \rangle \mathcal{T}_{(q-1)t w_1} \mathcal{P}_{-\frac{1}{w_2}} \mathcal{T}_{-(q-1)w_2} d_+^* F \right) \right\rvert_{w_1=y_1}
    \]
    while, up to a coefficient, the contribution of the $z_1$ term is
    \[ 
        \left. \left(\langle w_2^{-m} \rangle \left(1-\Exp\left[(q-1)t\frac{w_1}{w_2}\right] \right) \mathcal{T}_{(q-1)tw_1}\mathcal{P}_{-\frac{1}{w_2}} \mathcal{T}_{-(q-1)w_2} d_+^* F \right) \right\rvert_{w_1=y_1}.
    \]
    In either expression, the terms of $\mathcal{T}_{-(q-1)w_2} d_+^* F$ involve only powers of $w_2$ with nonnegative exponents, all at most $\deg(F)$. Since $-m > \deg(F)$, extracting the corresponding coefficient gives $0$, as expected.

    Finally, we check \eqref{dop:recursion}. Let $\alpha \in \mathbb{Z}^r$ and $\beta \in \mathbb{Z}^s$. We have
    \begin{align*}
        \D_{\alpha \beta} F & = d_- (-y_1)^{\alpha_1-1} \z_1 \cdots (-y_1)^{\alpha_r} \z_1 (-y_1)^{\beta_1} \cdots \z_1 (-y_1)^{\beta_s} \z_1 d_+ F \\
        & = d_- (-y_1)^{\alpha_1-1} \z_1 \cdots (-y_1)^{\alpha_r} \z_1 (-y_1) (-y_1)^{\beta_1-1} \cdots \z_1 (-y_1)^{\beta_s} \z_1 d_+ F \\
        & = d_- (-y_1)^{\alpha_1-1} \z_1 \cdots (-y_1)^{\alpha_r} (-y_1) (d_+^* d_- + qt \z_1) (-y_1)^{\beta_1-1} \cdots \z_1 (-y_1)^{\beta_s} \z_1 d_+ F \\
        & = d_- (-y_1)^{\alpha_1-1} \z_1 \cdots (-y_1)^{\alpha_r} (-y_1) d_+^* d_- (-y_1)^{\beta_1-1} \cdots \z_1 (-y_1)^{\beta_s} \z_1 d_+ F \\
        & \quad + qt \cdot d_- (-y_1)^{\alpha_1-1} \z_1 \cdots (-y_1)^{\alpha_r + 1} \z_1 (-y_1)^{\beta_1-1} \cdots \z_1 (-y_1)^{\beta_s} \z_1 d_+ F \\
        & = d_- (-y_1)^{\alpha_1-1} \z_1 \cdots (-y_1)^{\alpha_r} \z_1 d_+ d_- (-y_1)^{\beta_1-1} \cdots \z_1 (-y_1)^{\beta_s} \z_1 d_+ F \\
        & \quad + qt \cdot d_- (-y_1)^{\alpha_1-1} \z_1 \cdots (-y_1)^{\alpha_r + 1} \z_1 (-y_1)^{\beta_1-1} \cdots \z_1 (-y_1)^{\beta_s} \z_1 d_+ F \\
        & = \D_\alpha \D_\beta F + qt \D_{\alpha_+ \beta_-} F,
    \end{align*}
    as desired.
\end{proof}

\begin{proposition}
    \label{prop:dgamma-uniquely-determined}
    $\D_\gamma F$ is uniquely determined by the relations \eqref{dop:onepart}, \eqref{dop:neg}, and \eqref{dop:recursion}.
\end{proposition}

\begin{proof}
    Given $\gamma \in \bigcup_{n > 0} \mathbb{Z}^n$ and $F \in \Lambda$, we will actually show that $\D_\gamma F$ is completely determined by the relations \eqref{dop:onepart}, \eqref{dop:neg}, and \eqref{dop:recursion} with $\alpha = (\gamma_1, \dots, \gamma_{\ell-1})$ and $\beta = (\gamma_\ell)$; then \Cref{thm:dgamma-recursion} ensures that \eqref{dop:recursion} is satisfied in full generality, since the left-hand side does not depend on the choice of the split of $\gamma$ into $\alpha$ and $\beta$.

    Let $\deg F = n$. We proceed by double induction on $\ell(\gamma)$ and $\gamma_\ell$. If $\ell(\gamma) = 1$, then \eqref{dop:onepart} uniquely determines $\D_\gamma F$; if $\gamma_\ell < -n$, then by \eqref{dop:neg} we have $\D_\gamma F = 0$, so the base cases are covered. If $\ell(\gamma) > 1$ and $\gamma_\ell \geq -n$, by \eqref{dop:recursion} we have \[ \D_\gamma F = \D_{(\gamma_1, \dots, \gamma_{\ell-1})} \D_{(\gamma_\ell)} F + qt \D_{(\gamma_1, \dots, \gamma_{\ell-1}+1, \gamma_\ell-1)} F, \] and now the right-hand side of the equation is uniquely determined by induction.
\end{proof}

In the rest of this work, we reprove every statement using our definition. The following result is an example of a statement that is known for the operators from \cite{BHMPS2023ShuffleAnyLine}, but which we reprove for our operators.

\begin{lemma}
    \label{lem:Dgamma_q1}
    For any $\gamma \in \mathbb{N}^n$, upon specialization at $q=1$, $\D_\gamma$ acts by multiplication by $\left( \D_\gamma \cdot 1 \right)_{q=1} $.
\end{lemma}

\begin{proof}
    We proceed by double induction on $\ell(\gamma)$ and $\gamma_{\ell(\gamma)}$, letting the last component vary in $\mathbb{Z}$ rather than $\mathbb{N}$. Let $F \in \Lambda$ be regular at $q=1$ (i.e.\ its coefficients when expanded in the Schur basis have no poles at $q=1$, so that the specialization is well-defined).
    
    If $\gamma = (m)$, we have
    \begin{align*}
        \left( \D_{(m)} F \right)_{q=1} & = (-1)^m \langle z^m \rangle \Exp[-zX] F[X + M/z] \rvert_{q=1} \\
        & = (-1)^m \langle z^m \rangle \Exp[-zX] F[X] \rvert_{q=1} \\
        & = e_m F \rvert_{q=1}
    \end{align*}
    if $m \geq 0$ and $0$ otherwise, so the base case for the length holds. If $\gamma_{\ell(\gamma)} < - \deg(F)$, then $\D_\gamma F = 0$ by \Cref{thm:dgamma-recursion}. Finally, if $\gamma = (\gamma', m)$, we have
    \[ \left. \D_\gamma F \right\rvert_{q=1} = \left. \left( \D_{\gamma'} \D_{(m)} + qt \D_{(\gamma'_+, m-1)} \right) F \right\rvert_{q=1} \]
    and the claim follows by induction.
\end{proof}

\subsection{Identification with the \texorpdfstring{Negu\textcommabelow{t}}{Negut} operators}

In this subsection, we show that the operators $\D_\gamma$ defined in \Cref{def:D_gamma} coincide with the operators $\D_\gamma$ defined in \cite[Equation~(52)]{BHMPS2023ShuffleAnyLine}.
Until this identification is proved, $\D_\gamma$ in this subsection denotes the latter operators. In particular, the two lemmas below are proved independently of \Cref{def:D_gamma}.

Notice that the results in this paper are completely independent from this identification and hold for the operators in \Cref{def:D_gamma}: outside of this subsection, we do not rely on any results from \cite{BHMPS2023ShuffleAnyLine} or \cite{Negut2014ShuffleAlgebra} related to the $\D_\gamma$ operators.

\begin{theorem}
    \label{thm:same-operators}
    The operators $\D_\gamma$ defined in \cite[Equation~(52)]{BHMPS2023ShuffleAnyLine} satisfy the relations \eqref{dop:onepart}, \eqref{dop:neg}, and \eqref{dop:recursion}.
    In particular, they coincide with the operators $\D_\gamma$ from \Cref{def:D_gamma}.
\end{theorem}

Relation~\eqref{dop:onepart} is \cite[Proposition~3.3.4]{BHMPS2023ShuffleAnyLine}, and relation~\eqref{dop:recursion} is \cite[Equation~(47)]{BHMPS2023ProofExtendedDelta}. Relation~\eqref{dop:neg} is more involved, and its proof is going to take the rest of this subsection.

\begin{lemma}
    \label{lem:on-one}
    If $\gamma \in \Z^\ell$ and $\gamma_\ell<0$, then $\D_{\gamma} \cdot 1 = 0$.
\end{lemma}

\begin{proof}
    By \cite[Corollary~3.7.2 and Equation~(46)]{BHMPS2023ShuffleAnyLine},
    \[
        (\omega(\D_\gamma\cdot 1))(x_1,\dots,x_\ell)
        =\left[\boldsymbol{\sigma}\left(
        x^\gamma\frac{\prod_{i+1<j}(1-qt x_i/x_j)}
        {\prod_{i<j}(1-q x_i/x_j)(1-t x_i/x_j)}
        \right)\right]_{\mathrm{pol}},
    \]
    where $\boldsymbol{\sigma}$ is Weyl symmetrization and the denominator factors are expanded geometrically in $x_i/x_j$ for $i<j$. Every monomial $x^\mu$ inside this Weyl symmetrization has exponent vector
    \[ \mu=\gamma+\sum_{i<j} a_{ij}(\mathbf e_i-\mathbf e_j),\qquad a_{ij}\geq 0. \]
    Hence $\mu_\ell=\gamma_\ell-\sum_{i<\ell}a_{i\ell}<0$.

    To see that such a monomial has zero polynomial part after symmetrization, set $\rho=(\ell-1,\dots,1,0)$. If $\mu+\rho$ has repeated coordinates, then $\boldsymbol{\sigma}(x^\mu)=0$. Otherwise, Weyl straightening gives $\boldsymbol{\sigma}(x^\mu)=\pm\chi_\lambda$, where $\lambda+\rho$ is the decreasing rearrangement of $\mu+\rho$. In particular,
    \[ \lambda_\ell=\min_i(\mu_i+\rho_i)\leq\mu_\ell<0, \]
    so $\chi_\lambda$ has zero polynomial part. Thus the displayed specialization of $\omega(\D_\gamma\cdot 1)$ vanishes. The Schur-length bound in the same corollary says that this symmetric function is determined by its specialization to $\ell$ variables, and therefore $\D_\gamma\cdot 1=0$.
\end{proof}

\begin{lemma}
    \label{lem:commutator}
    For every $\gamma=(\gamma_1,\dots,\gamma_\ell)\in\Z^\ell$ and $k\geq 1$,
    \begin{equation}        
        \label{eq:commutator}
        [\D_{\gamma}, p_k^\bullet] = (-1)^k (1-q^k)(1-t^k) \sum_{i=1}^{\ell} \D_{\gamma+k\mathbf{e}_i},
    \end{equation}
    where $p_k^\bullet$ denotes multiplication by $p_k$ and $\mathbf{e}_i$ is the $i$-th standard basis vector of $\Z^\ell$.
\end{lemma}

\begin{proof}
    If $\gamma = (m)$, by \cite[Lemma~3.3.3]{BHMPS2023ShuffleAnyLine} we have
    \[ \left[\bigl(\omega p_k[-X/M]\bigr)^\bullet, \D_m \right] = -\D_{m+k}. \]
    Since
    \[ \omega p_k[-X/M] = \frac{(-1)^k}{(1-q^k)(1-t^k)} p_k[X], \]
    we can rewrite this as
    \begin{equation}
        [\D_m, p_k^\bullet] = (-1)^k (1-q^k)(1-t^k) \D_{m+k}, \label{eq:one-part-final}
    \end{equation}
    which is the desired relation in the one-part case. Now let
    \[ \psi \colon \mathcal{S} \longrightarrow \mathcal{E}^+ \]
    be the shuffle-algebra isomorphism of \cite[Proposition~3.5.1]{BHMPS2023ShuffleAnyLine}. We use the tensor-algebra quotient presentation from \cite[Section~3.5]{BHMPS2023ShuffleAnyLine}, and write $\psi(\phi)$ for the image of the class of a Laurent polynomial $\phi$. In this presentation, the monomial $x_1^{a_1}\cdots x_\ell^{a_\ell}$ acts as the ordered product $\D_{a_1}\cdots\D_{a_\ell}$. Set
    \[ c_k=(-1)^k(1-q^k)(1-t^k),\qquad P_k(x)=\sum_{i=1}^\ell x_i^k. \]
    The Leibniz rule for commutators and \eqref{eq:one-part-final} give
    \[
        [\D_{a_1}\cdots\D_{a_\ell},p_k^\bullet]
        =c_k\sum_{i=1}^\ell
        \D_{a_1}\cdots\D_{a_{i-1}}\D_{a_i+k}\D_{a_{i+1}}\cdots\D_{a_\ell}.
    \]
    Since tensor monomials span the Laurent polynomials, it follows by linearity that
    \[ [\psi(\phi),p_k^\bullet]=c_k\,\psi(P_k\phi) \]
    as operators on $\Lambda$, for every Laurent polynomial $\phi$ in $\ell$ variables. Multiplication by $P_k$ is well-defined on the quotient: its symmetry implies
    $\mathbf H_{q,t}^{(\ell)}(P_k\phi)=P_k\mathbf H_{q,t}^{(\ell)}(\phi)$, so it preserves the kernel defining $\mathcal S$.

    For completeness, the distinguished element $\D_{\gamma}$ is represented
    in~\cite[(51)--(52)]{BHMPS2023ShuffleAnyLine} by the rational shuffle expression
    \[ R_{\gamma}(x)
        = \frac{x_1^{\gamma_1}\cdots x_\ell^{\gamma_\ell}} {\prod_{j=1}^{\ell-1}(1-qt\,x_j/x_{j+1})}.
    \]
    More precisely, the construction in \cite[Section~3.6]{BHMPS2023ShuffleAnyLine}, using \cite[Proposition~6.1]{Negut2014ShuffleAlgebra}, provides a Laurent polynomial $\eta_{\gamma}$ such that
    \[
        \mathbf{H}_{q,t}^{(\ell)}(\eta_{\gamma}) =\mathbf{H}_{q,t}^{(\ell)}(R_{\gamma}),
        \qquad \D_{\gamma}=\psi(\eta_{\gamma}).
    \]
    Applying this construction also to each $\gamma+k\mathbf e_i$, choose corresponding Laurent polynomials $\eta_{\gamma+k\mathbf e_i}$. Because $p_k(x_1,\dots,x_\ell)$ is symmetric, it can be pulled through the symmetrization operator, so
    \begin{align*}
        \mathbf{H}_{q,t}^{(\ell)}(p_k(x_1,\dots,x_\ell) \eta_{\gamma})
        & = p_k(x_1,\dots,x_\ell) \mathbf{H}_{q,t}^{(\ell)}(\eta_{\gamma})\\
        & = p_k(x_1,\dots,x_\ell) \mathbf{H}_{q,t}^{(\ell)}(R_{\gamma})\\
        & = \mathbf{H}_{q,t}^{(\ell)}(p_k(x_1,\dots,x_\ell) R_{\gamma})\\
        & = \sum_{i=1}^{\ell} \mathbf{H}_{q,t}^{(\ell)}(R_{\gamma + k\mathbf{e}_i}).
    \end{align*}
    Thus the Laurent polynomial $P_k\eta_\gamma-\sum_i\eta_{\gamma+k\mathbf e_i}$ lies in the kernel defining $\mathcal S$. Applying $\psi$ and the commutator formula for Laurent polynomials proves~\eqref{eq:commutator}.
\end{proof}

\begin{proof}[Proof of \Cref{thm:same-operators}]
    We already know that the operators $\D_\gamma$ from \cite{BHMPS2023ShuffleAnyLine} satisfy relations~\eqref{dop:onepart} and~\eqref{dop:recursion}. By \Cref{prop:dgamma-uniquely-determined}, it suffices to show that they also satisfy relation~\eqref{dop:neg}, that is, 
    \[ m < -\deg(F) \implies \D_{\alpha (m)} F = 0. \]
    It suffices to prove this when $F \in \Lambda^{(n)}$ is homogeneous; we proceed by strong induction on $n$. If $n=0$, the assertion reduces to \Cref{lem:on-one}.

    Now let $n>0$ and suppose that the assertion holds in every degree smaller than $n$. The power-sum monomials $p_\lambda$, for $\lambda\vdash n$, form a basis of $\Lambda^{(n)}$, so it is enough by linearity to take $F=p_kG$ with $1\leq k\leq n$ and $G\in\Lambda^{(n-k)}$. By \Cref{lem:commutator}, for any $\gamma\in\Z^\ell$,
    \begin{equation}
        \label{eq:induction}
        \D_{\gamma}(p_kG) = p_k \D_{\gamma} G + (-1)^k (1-q^k)(1-t^k) \sum_{i=1}^{\ell}\D_{\gamma + k \mathbf{e}_i} G.
    \end{equation}
    If $\gamma_\ell=m<-n$, then the last component of $\gamma$ and of every $\gamma+k\mathbf e_i$ is at most
    \[ m+k<-n+k=-(n-k). \]
    All operators on the right-hand side of \eqref{eq:induction} therefore annihilate $G$ by the induction hypothesis.

    The identification with \Cref{def:D_gamma} now follows.
\end{proof}

\subsection{Computational implementation and benchmarks}
\label{sec:benchmarks}

A major practical advantage of our new \Cref{def:D_gamma} is that it provides a direct linear algorithm for computing $\D_\gamma F$ for any symmetric function $F \in \Lambda$ and any vector of integers $\gamma \in \mathbb{Z}^\ell$.

In contrast to recursive implementations based on \Cref{thm:dgamma-recursion} (which generate an evaluation tree of depth $\ell-1$ with $O(2^{\ell-1})$ leaf operations), our formula evaluates $\D_\gamma F$ via $O(\ell)$ sequential operator steps in the extended module $V_1 = \Lambda \otimes \mathbb{Q}(q,t)[y_1^{\pm 1}]$.

We evaluated both implementations (\texttt{D\_new} for \Cref{def:D_gamma} and \texttt{D\_old} for \Cref{thm:dgamma-recursion}) in SageMath on CoCalc. Table~\ref{tab:benchmarks_a} reports execution times for evaluating $\D_\gamma \cdot 1$ for compositions $\gamma \in \mathbb{Z}^\ell$ as the length $\ell(\gamma)$ increases. Table~\ref{tab:benchmarks_b} reports timings for evaluating small compositions on various Schur basis elements $F \in \Lambda$.

\begin{table}[ht]
    \centering
    \small
    \begin{tabular}{|l|c|c|r|r|r|}
        \hline
        Composition $\gamma$ & Length $\ell$ & Input $F$ & \texttt{D\_new} & \texttt{D\_old} & Speedup \\
        \hline
        $(3)$ & 1 & $1$ & $0.0029$\,s & $0.0285$\,s & $9.8\times$ \\
        $(2, 1)$ & 2 & $1$ & $0.0116$\,s & $0.0583$\,s & $5.0\times$ \\
        $(0, 3)$ & 2 & $1$ & $0.0261$\,s & $0.1334$\,s & $5.1\times$ \\
        $(1, 1, 2)$ & 3 & $1$ & $0.0610$\,s & $0.9231$\,s & $15.1\times$ \\
        $(2, -1, 2)$ & 3 & $1$ & $0.0383$\,s & $0.2084$\,s & $5.4\times$ \\
        $(1, 2, 1, 1)$ & 4 & $1$ & $0.1078$\,s & $5.7311$\,s & $53.1\times$ \\
        $(3, 0, -1, 2)$ & 4 & $1$ & $0.0541$\,s & $0.9364$\,s & $17.3\times$ \\
        $(1, 1, 3, 0, 2)$ & 5 & $1$ & $0.7124$\,s & $> 300$\,s & $>420\times$ \\
        $(1, -2, 3, 0, 1)$ & 5 & $1$ & $0.1809$\,s & $3.4109$\,s & $18.9\times$ \\
        \hline
    \end{tabular}
    \caption{Benchmark Part A: fixed input $F = 1$, varying composition length $\ell(\gamma)$.}
    \label{tab:benchmarks_a}
\end{table}

\begin{table}[ht]
    \centering
    \small
    \begin{tabular}{|l|c|c|r|r|r|}
        \hline
        Composition $\gamma$ & Length $\ell$ & Input $F$ & \texttt{D\_new} & \texttt{D\_old} & Speedup \\
        \hline
        $(0, 2)$ & 2 & $s_1$ & $0.0394$\,s & $0.1272$\,s & $3.2\times$ \\
        $(0, 2)$ & 2 & $s_2$ & $0.0819$\,s & $0.4007$\,s & $4.9\times$ \\
        $(2, -1, 1)$ & 3 & $s_1$ & $0.0514$\,s & $0.2102$\,s & $4.1\times$ \\
        $(2, -1, 1)$ & 3 & $s_{2,1}$ & $0.2758$\,s & $4.2279$\,s & $15.3\times$ \\
        $(3, -1, 2)$ & 3 & $s_2$ & $0.2338$\,s & $11.2977$\,s & $48.3\times$ \\
        $(3, -1, 2)$ & 3 & $s_{2,2}$ & $0.8745$\,s & $216.8674$\,s & $248.0\times$ \\
        $(1, 0, -1, 2)$ & 4 & $s_{2,1} $ & $0.9265$\,s & $35.8861$\,s & $38.7\times$ \\
        \hline
    \end{tabular}
    \caption{Benchmark Part B: small compositions evaluated on nontrivial input symmetric functions $F \in \Lambda$.}
    \label{tab:benchmarks_b}
\end{table}

The code is available at \cite[{\texttt{benchmarks.ipynb}}]{Iraci2026Code}. The implementation of the $\A_{q,t}$ algebra is based on code originally written by Anton Mellit.

\subsection{Commutation with skewing operators}

By \Cref{prop:tau}, for $F \in \Lambda$, we have
\[ \tau(u) F = F[X+u] = \sum_{k=0}^\infty u^k h_k^\perp F \quad \text{ and } \quad \tau(-u) F = F[X-u] = \sum_{k=0}^\infty (-u)^k e_k^\perp F. \]

We can derive the following commutation relations between skewing operators and $\D_\gamma$ operators.

\begin{theorem}
    For every $n \geq 0$ and $\gamma\in \mathbb{Z}^\ell$, we have
    \begin{equation}
        e_n^\perp \D_\gamma = \sum_{\beta \in \mathbb{N}^\ell} \D_{\gamma-\beta} e_{n-\lvert \beta \rvert}^\perp
    \end{equation}
    and
    \begin{equation}
        h_n^\perp \D_\gamma = \sum_{\beta \in \{0,1\}^\ell} \D_{\gamma-\beta} h_{n-\lvert \beta \rvert}^\perp
    \end{equation}
    where the difference is taken componentwise.
\end{theorem}

\begin{proof}
    From \Cref{prop:tau}, we can deduce that $\tau(u)\ \z_1 = \Exp[-u/y_1] \z_1 \tau(u)$. Since
    \[ \Exp[-u/y_1] = (1 - u y_1^{-1}) \quad \text{ and } \quad \Exp[u/y_1] = (1 - u y_1^{-1})^{-1}, \]
    we have
    \begin{align*}
        \tau(u) \D_\gamma & = \tau(u)\ d_- (-y_1)^{\gamma_{1}-1} \z_1 (-y_1)^{\gamma_{2}} \z_1 \cdots (-y_1)^{\gamma_{\ell}} \z_1 d_+ \\
        & = d_- (-y_1)^{\gamma_{1}-1} (1 - uy_1^{-1}) \z_1 \cdots (-y_1)^{\gamma_{\ell}} (1 - uy_1^{-1}) \z_1 d_+ \tau(u) \\
    \end{align*}
    and
    \begin{align*}
        \tau(-u) \D_\gamma & = \tau(-u)\ d_- (-y_1)^{\gamma_{1}-1} \z_1 (-y_1)^{\gamma_{2}} \z_1 \cdots (-y_1)^{\gamma_{\ell}} \z_1 d_+ \\
        & = d_- (-y_1)^{\gamma_{1}-1} (1 - uy_1^{-1})^{-1} \z_1 \cdots (-y_1)^{\gamma_{\ell}} (1 - uy_1^{-1})^{-1} \z_1 d_+ \tau(-u) \\
    \end{align*}
    and the result follows by taking the coefficient of $u^n$.
\end{proof}

\subsection{Action of \texorpdfstring{$\D_\gamma$}{Dgamma} on \texorpdfstring{$1$}{1}}
\label{section:Dgammaon1}

The goal of this subsection is to prove directly from our explicit definition of $\D_\gamma$ that $\D_\gamma \cdot 1$ equals the Hall-Littlewood series stated in \cite[Corollary~3.7.2]{BHMPS2023ShuffleAnyLine}. We need to recall some definitions and basic facts from \cite[Section~2]{BHMPS2023ShuffleAnyLine}.

Consider the action of the symmetric group $\mathfrak{S}_\ell$ on the Laurent power series in variables $x_1,x_2,\dots,x_\ell$ with coefficients in $\mathbb{Q}(q,t)$ by permuting the variables. Given a Laurent polynomial $f(x) = f(x_1,x_2,\dots,x_\ell) \in \mathbb{Q}(q,t)[x_1^{\pm 1}, x_2^{\pm 1}, \dots, x_\ell^{\pm 1}]$, set
\[
    \textbf{$\sigma$}\, f(x_1,x_2,\dots,x_\ell) \coloneqq \sum_{w\in \mathfrak{S}_\ell}w\left(\frac{f(x)}{\prod_{i<j}(1-x_j/x_i)}\right),
\]
and
\[
    \mathbf{H}_{q,t}^{(\ell)}(f(x_1,x_2,\dots,x_\ell)) \coloneqq \textbf{$\sigma$}\, \left( \frac{f(x) \prod_{i<j}(1-qtx_i/x_j) }{\prod_{i<j}(1-qx_i/x_j)(1-tx_i/x_j)}\right).
\]

A vector $\lambda=(\lambda_1,\lambda_2,\dots,\lambda_\ell) \in \mathbb{Z}^\ell$ is called \emph{dominant} if $\lambda_1\geq \lambda_2\geq\cdots \geq \lambda_\ell$. Given $\lambda\in \mathbb{Z}^\ell$, set $x^\lambda \coloneqq x_1^{\lambda_1}x_2^{\lambda_2}\cdots x_\ell^{\lambda_\ell}$, and for a dominant vector $\lambda\in \mathbb{Z}^\ell$ set
\[\chi_\lambda \coloneqq \textbf{$\sigma$}(x^\lambda).\]
When $\lambda_\ell\geq 0$, $\chi_\lambda$ coincides with the Schur polynomial $s_\lambda(x_1,x_2,\dots,x_\ell)$.

Any symmetric formal Laurent series $f(x) = f(x_1,x_2,\dots,x_\ell)$ can be written uniquely as 
\[ f(x) = \sum_{\lambda\text{ dominant}} c_\lambda \chi_\lambda \] 
for some coefficients $c_\lambda\in \mathbb{Q}(q,t)$. We denote by $f(x)_{\text{pol}}$ the \emph{polynomial part} of $f(x)$, i.e.\ 
\[ f(x)_{\text{pol}} \coloneqq \sum_{\lambda\text{ dominant s.t.\ }\lambda_\ell\geq 0} c_\lambda \chi_\lambda, \]
which is a symmetric power series in $x_1,x_2,\dots,x_\ell$.

The following lemma is a straightforward consequence of \cite[Lemma~2.3.1, Equation~(21)]{BHMPS2023ShuffleAnyLine}.
\begin{lemma}
    \label{lem:poly_part}
    For any Laurent series $f(x)$ in $x_1,x_2,\dots,x_\ell$, we have
    \[\boldsymbol{\sigma} (f(x))_{\text{pol}}=\langle w^0 \rangle f(w_1,w_2,\dots,w_\ell) \Exp[\overline{W}X]\prod_{i<j}(1-w_i/w_j),\]
    where $X=x_1+x_2+\cdots +x_\ell$ and $\overline{W}=w_1^{-1}+w_2^{-1}+\cdots +w_\ell^{-1}$.
\end{lemma}

We can now state the main theorem of this section: cf.\ \cite[Corollary~3.7.2]{BHMPS2023ShuffleAnyLine}.
\begin{theorem}
    \label{thm:Don1}
    For any $\gamma=(\gamma_1,\dots,\gamma_{\ell})\in \mathbb{Z}^{\ell}$, we have
    \[\omega \D_\gamma \cdot 1 = \mathbf{H}_{q,t}^{(\ell)}\left(\frac{x^\gamma}{\prod_{i=1}^{\ell-1}(1-qt x_i/x_{i+1})}\right)_{\mathrm{pol}}.\]
\end{theorem}

The remainder of this section is devoted to the proof of this result. We start with a lemma.

\begin{lemma}
    \label{lemma:z1y1a}
    For $F = F[X;y_1]\in V_1$ and $a \in \mathbb{Z}$, the term $\z_1 y_1^a F[X;y_1]$ is equal to
    \[ \left. \frac{1}{t(1-q)}\left(\langle w_2^{-a} \rangle \left(1-\Exp\left[(q-1)t\frac{w_1}{w_2}\right] \right) \mathcal{T}_{(q-1)tw_1}\mathcal{P}_{-\frac{1}{w_2}} \mathcal{T}_{-(q-1)w_2}F[X;w_2]\right)\right\rvert_{w_1=y_1}. \]
\end{lemma}

\begin{proof}
    First, observe that for every $F = F[X;y_1]\in V_1$,
    \[ d_- F =\langle y_1^0 \rangle\mathcal{P}_{-\frac{1}{y_1}} \mathcal{T}_{-(q-1)y_1}F, \]
    and for every $G = G[X] \in V_0$,
    \[ d_+^* G = G[X + (q-1) t y_1] = \mathcal{T}_{(q-1)t y_1} G. \]
    Therefore, since
    \[ \z_1 y_1^a F[X;y_1] = \frac{1}{t(1-q)}(d_+^*d_- - d_- d_+^*) y_1^a F[X;y_1], \]
    we compute
    \begin{align*}
        d_+^*d_- y_1^a F[X;y_1] & = \mathcal{T}_{(q-1)t y_1} \langle y_1^0 \rangle\mathcal{P}_{-\frac{1}{y_1}} \mathcal{T}_{-(q-1)y_1} y_1^a F[X;y_1] \\
        & = \left. \left(\mathcal{T}_{(q-1)t w_1} \langle w_2^0 \rangle\mathcal{P}_{-\frac{1}{w_2}} \mathcal{T}_{-(q-1)w_2} w_2^a F[X;w_2]\right) \right\rvert_{w_1=y_1} \\
        & = \left. \left(\langle w_2^{-a} \rangle \mathcal{T}_{(q-1)t w_1} \mathcal{P}_{-\frac{1}{w_2}} \mathcal{T}_{-(q-1)w_2}F[X;w_2]\right) \right\rvert_{w_1=y_1}.
    \end{align*}
    Similarly,
    \begin{align*}
        & d_-d_+^* y_1^a F[X;y_1] = d_- y_2^a d_+^* F[X;y_1] && \text{by \eqref{eq:Q1*}} \\
        & = d_- y_2^a \eta \mathcal{T}_{(q-1) y_2} F[X;y_1] \\
        & = d_- y_2^a \eta F[X+(q-1) y_2;y_1] \\
        & = d_- y_2^a  F[X+(q-1) ty_1;y_2] \\
        & = d_- y_2^a  \mathcal{T}_{(q-1)t y_1}F[X;y_2] \\
        & = \langle y_2^0\rangle \mathcal{P}_{-\frac{1}{y_2}} \mathcal{T}_{-(q-1)y_2} y_2^a \mathcal{T}_{(q-1)t y_1}F[X;y_2] \\
        & =\langle y_2^0\rangle \Exp\left[(q-1)t\frac{y_1}{y_2}\right]\mathcal{T}_{(q-1)t y_1} \mathcal{P}_{-\frac{1}{y_2}} \mathcal{T}_{-(q-1)y_2} y_2^a F[X;y_2] && \text{by \eqref{eq:TPExp}} \\
        & = \left. \left(\langle w_2^{-a} \rangle \Exp\left[(q-1)t\frac{w_1}{w_2}\right] \mathcal{T}_{(q-1)t w_1} \mathcal{P}_{-\frac{1}{w_2}} \mathcal{T}_{-(q-1)w_2} F[X;w_2]\right) \right\rvert_{w_1=y_1}.
    \end{align*}
    Putting these pieces together gives the result.
\end{proof}

We have the following corollary.

\begin{corollary}
    \label{cor:Dgamma1plethystic}
    Given $F = F[X;y_1]\in V_1$, we have
    \begin{align*}
        & d_- y_1^{\gamma_{1}-1} \z_1 y_1^{\gamma_{2}} \cdots \z_1 y_1^{\gamma_{\ell-1}} \z_1 y_1^{\gamma_{\ell}+1} F \\
        & \quad = \langle w^{-\gamma} \rangle \prod_{i=1}^{\ell-1}\left(\frac{1}{1-qtw_{i}/w_{i+1}}\right) \prod_{i<j} \Exp\left[-\frac{w_{i}}{w_j} M\right] \mathcal{P}_{-\overline{W}} \mathcal{T}_{MW} \mathcal{T}_{-t(q-1)w_\ell}F[X;w_\ell],
    \end{align*}
    where $ W = w_1+w_2+\cdots +w_{\ell}$ and $\overline{W} = {w_1}^{-1}+{w_2}^{-1}+\cdots +{w_\ell}^{-1}$.
\end{corollary}

\begin{proof}
    Iterating \Cref{lemma:z1y1a}, we compute
    \begin{align*}
        d_- y_1^{\gamma_{1}-1} \z_1 y_1^{\gamma_{2}} \cdots \z_1 y_1^{\gamma_{\ell-1}} \z_1 y_1^{\gamma_{\ell}+1} F & = \langle w^{-\gamma} \rangle \frac{w_\ell}{w_1} \frac{1}{t^{\ell-1}(1-q)^{\ell-1}}\prod_{i=1}^{\ell-1} \left(1-\Exp\left[(q-1)t \frac{w_{i}}{w_{i+1}} \right] \right)
        \\
        & \quad \times
        \mathcal{P}_{-\frac{1}{w_1}} \mathcal{T}_{Mw_1} \cdots \mathcal{P}_{-\frac{1}{w_{\ell-1}}} \mathcal{T}_{Mw_{\ell-1}}\mathcal{P}_{-\frac{1}{w_\ell}} \mathcal{T}_{-(q-1)w_\ell}F[X;w_\ell].
    \end{align*}
    Now, iterating \eqref{eq:TPExp} and the fact that
    \[1-\Exp\left[(q-1)t\frac{w_i}{w_{i+1}}\right]=t\frac{w_i}{w_{i+1}}\cdot \frac{(1-q)}{1-qtw_i/w_{i+1}}, \]
    we deduce that the previous expression is equal to
    \[
        \langle w^{-\gamma} \rangle \prod_{i=1}^{\ell-1}\left(\frac{1}{1-qtw_{i}/w_{i+1}}\right) \prod_{i<j} \Exp\left[-\frac{w_{i}}{w_j} M\right] \mathcal{P}_{-\overline{W}} \mathcal{T}_{MW} \mathcal{T}_{-t(q-1)w_\ell} F[X;w_\ell]
    \]
    as stated.
\end{proof}

We are finally ready to prove the main result of this section.

\begin{proof}[Proof of \Cref{thm:Don1}]
    Applying \Cref{cor:Dgamma1plethystic} to $d_+^*(1)=1 \in V_1$ and then applying $\omega$, we get 
    \begin{align*}
    \omega \D_\gamma \cdot 1 & = \omega \left( d_- (-y_1)^{\gamma_{1}-1} \z_1 (-y_1)^{\gamma_{2}} \z_1 (-y_1)^{\gamma_{3}} \cdots \z_1 (-y_1)^{\gamma_{\ell-1}} \z_1 (-y_1)^{\gamma_{\ell}+1} d_+^* (1) \right) \\
        & = (-1)^{\lvert \gamma \rvert} \cdot \omega \left( \langle w^{-\gamma} \rangle \prod_{i=1}^{\ell-1} \left(\frac{1}{1-qtw_{i}/w_{i+1}}\right) \prod_{i<j} \Exp\left[-\frac{w_{i}}{w_j} M\right] \mathcal{P}_{-\overline{W}}  \mathcal{T}_{MW} \mathcal{T}_{-t(q-1)w_\ell} 1 \right).
    \end{align*}
    Now observe that $\mathcal{T}_Y \cdot 1 = 1$, while $\mathcal{P}_Z \cdot 1 = \Exp[XZ]$ for all expressions $Y$ and $Z$ for which the operators are defined. Therefore, the last expression becomes
    \begin{align*}
        & (-1)^{\lvert \gamma \rvert} \cdot \omega \left( \langle w^{-\gamma} \rangle \prod_{i=1}^{\ell-1} \left(\frac{1}{1-qtw_{i}/w_{i+1}}\right) \prod_{i<j} \Exp\left[-\frac{w_{i}}{w_j} M\right] \Exp\left[-\overline{W}X\right] \right) \\
        & \quad = \langle w^0 \rangle \left(\frac{w^\gamma}{\prod_{i=1}^{\ell-1}(1-qtw_{i}/w_{i+1})}\right) \prod_{i<j}\Exp\left[-\frac{w_{i}}{w_j} M\right] \Exp\left[\overline{W}X\right],
    \end{align*}
    because the constant term is symmetric in the variables $x_1,x_2,\dots$. The sign in the plethystic exponential $\Exp[-\overline{W}X]$ cancels out with $(-1)^{\lvert \gamma \rvert}$ and $\omega$, and the $w^\gamma$ appears because of the shift in the coefficient extraction.
    Now, since
    \[ \Exp\left[-\frac{w_i}{w_j}M\right] = \frac{(1-qtw_i/w_j)(1-w_i/w_j)}{(1-qw_i/w_j)(1-tw_i/w_j)} \]
    the last expression reduces to
    \begin{align*}
        &   \langle w^0 \rangle \left(\frac{w^\gamma}{\prod_{i=1}^{\ell-1}(1-qtw_{i}/w_{i+1})} \cdot \frac{\prod_{i<j}(1-qtw_i/w_j)}{\prod_{i<j}(1-qw_i/w_j)(1-tw_i/w_j)}\right) \Exp\left[\overline{W}X\right] \prod_{i<j} \left(1-\frac{w_{i}}{w_j}\right),
    \end{align*}
    and now by \Cref{lem:poly_part} this gives precisely
    \[
        \textbf{$\sigma$} \left(\frac{w^\gamma}{\prod_{i=1}^{\ell-1}(1-qtw_{i}/w_{i+1})}\cdot \frac{\prod_{i<j}(1-qtw_i/w_j)}{\prod_{i<j}(1-qw_i/w_j)(1-tw_i/w_j)}\right)_{\text{pol}}
    \]
    which is equal to
    \[ 
        \mathbf{H}_{q,t}^{(\ell)}\left(\frac{x^\gamma}{\prod_{i=1}^{\ell-1}(1-qtx_{i}/x_{i+1})} \right)_{\mathrm{pol}}
    \]
    by definition. This concludes the proof.
\end{proof}

\subsection{Connection to the rational shuffle theorem}
\label{section:D1}
We are now able to connect the approaches of Mellit and Blasiak et al.\ by outlining a proof of the rational shuffle theorem which relies on Mellit's algebraic construction \cite{Mellit2021Rational} (but not on the combinatorial recursion in terms of braids) and on the combinatorial interpretation of
\[ \mathbf{H}_{q,t}^{(\ell)}\left(\frac{x^\gamma}{\prod_{i=1}^{\ell-1}(1-qt x_i/x_{i+1})}\right)_{\mathrm{pol}} \]
by Blasiak et al.\ (but not on any theory of the elliptic Hall algebra).
We start by recalling the following basic proposition.

\begin{proposition}
    \label{prop:SL2}
    The monoid $\mathsf{SL}_2(\mathbb{Z})_+$, consisting of matrices with nonnegative integer entries and determinant $1$, is freely generated by the elements \[ N = \begin{pmatrix} 1 & 1 \\ 0 & 1 \end{pmatrix} \quad \text{ and } \quad S = \begin{pmatrix} 1 & 0 \\ 1 & 1 \end{pmatrix}. \]
\end{proposition}

\begin{definition}
    \label{def:lowest_path}
    Let $A \in \mathsf{SL}_2(\mathbb{Z})_+$. Let $U = (0,1)$ denote a vertical (up) step and $R = (1,0)$ a horizontal (right) step on the lattice grid $\mathbb{Z}^2$. By \Cref{prop:SL2}, $A$ admits a unique decomposition as a product of positive powers of $N$ and $S$. Using the assignments
    \[ N U = U R, \quad N R = R, \quad S U = U, \quad S R = U R, \]
    we define $P_A$ to be the lattice path from $(0,0)$ to $U A^t$ obtained by acting with $A$ on $U$.
\end{definition}

\begin{example}
    Let 
    \[ A_\gamma = \begin{pmatrix} 2 & 5 \\ 1 & 3 \end{pmatrix}  \in \mathsf{SL}_2(\mathbb{Z})_+. \]
    To obtain the factorization, consider $U A_\gamma^t = (5,3)$. We use the Euclidean algorithm:
    \[ (5,3) \rightarrow (2,3) \rightarrow (2,1) \rightarrow (1,1) \rightarrow (0,1), \]
    where decreasing the first component corresponds to $N^{-1}$ and decreasing the second component corresponds to $S^{-1}$. Reversing the steps, we obtain $A_\gamma = NSNN$, which is the desired factorization. To build $P_{A_\gamma}$, we act on $U = (0,1)$ in the same way, obtaining (see \Cref{fig:lowest_path_53})
    \[ U \stackrel{N}{\rightarrow} UR \stackrel{N}{\rightarrow} URR \stackrel{S}{\rightarrow} UURUR \stackrel{N}{\rightarrow} URURRURR. \]
\end{example}

\begin{figure}
    \centering
    \begin{tikzpicture}[scale=.72]
        \draw[step=1.0, gray!60, thin] (0,0) grid (5,3) (0,0) -- (5,3);
        \draw[blue!60, line width = 1.6 pt] (0,0) -- (0,1) -- (1,1) -- (1,2) -- (2,2) --  (3,2) -- (3,3) -- (4,3) -- (5,3); 
    \end{tikzpicture}
    \caption{The lowest lattice path above the diagonal from $(0,0)$ to $(5,3)$.}
    \label{fig:lowest_path_53}
\end{figure}
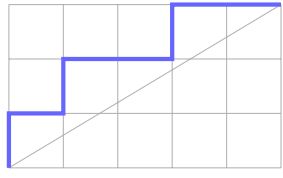

\begin{lemma}
    \label{lem:lowest_path}
    For $A \in \mathsf{SL}_2(\mathbb{Z})_+$, $P_A$ is the lowest lattice path above the diagonal from $(0,0)$ to $U A^t$ in the lattice grid.
\end{lemma}

\begin{proof}
    We proceed by induction on the length of $A$ when expressed as a word in $N$ and $S$. If the length is $0$, then $A = I$ and the statement is trivial.

    If the claim holds for some matrix $A$, then it suffices to check that it holds for $NA$ and $SA$. By induction, $P_A$ is the lowest lattice path above the diagonal from $(0,0)$ to $(a,b) = U A^t$ in the lattice grid.
    
    We have $U (NA)^t = (a+b, b)$. For $1 \leq i \leq b$, the intersections between the line $y=i$ and the diagonals from $(0,0)$ to $(a,b)$ and $(a+b,b)$ are exactly $i$ units apart, and so are the starting points of the vertical steps in row $i$ of $P_A$ and $P_{NA}$. In particular, the horizontal distances between those points are the same, and since those for $P_A$ are less than $1$ (because it is the lowest path above the diagonal), so are those for $P_{NA}$.

    The same argument holds for $SA$ by considering the intersections with the lines $x=i$ for $1 \leq i \leq a$ and looking at the vertical distances instead. The claim follows.
\end{proof}

\begin{proposition}[{\cite[Proposition~3.16]{Mellit2021Rational}}]
    \label{prop:north_south}
    The assignments
    \[ N(\epsilon_k) = S(\epsilon_k) = \epsilon_k, \quad N(T_i) = S(T_i) = T_i, \quad N(d_-) = S(d_-) = d_-, \]
    \[ N(d_+) = - \z_1 d_+, \quad N(d_+^*) = d_+^*, \quad S(d_+) = d_+, \quad S(d_+^*) = -y_1 d_+^*, \]
    extend to endomorphisms $N, S$ of $\A_{q,t}$ which preserve the kernel of the action on $V$.
\end{proposition}

Via these assignments and \Cref{prop:SL2}, we obtain an action of $\mathsf{SL}_2(\mathbb{Z})_+$ on $\A_{q,t}$ which preserves the kernel of the action on $V$.

\begin{remark}
    \label{rmk:intertwine}
    For $A \in \mathsf{SL}_2(\mathbb{Z})_+$ and $L \in \A_{q,t}$ let $A(L)$ denote the image of $L$ under the action of $A$ via the action from \Cref{prop:north_south}, and let $A^{\text{old}}(L)$ denote the image of $L$ under the action of $A$ via the action from \cite{BergeronGarsiaSergelXin2015CompositionalShuffleConjectures}. If $J$ is the operator on $\A_{q,t}$ which multiplies an element of bidegree $(a,b)$ by $(-1)^{a+b}$, then $J(A(L)) = A^{\text{old}}(J(L))$ \cite[Section~3, p.\ 4143]{Mellit2021Rational}.
    
    In particular, let $L \in \epsilon_0 \A_{q} \epsilon_0$ of degree $d$, corresponding to the multiplication by a symmetric function $F = L \cdot 1$ of degree $d$; its bidegree as an element of $\A_{q,t}$ is then $(0,d)$. It follows that (cf.\ \cite[Conjecture~4.2]{BergeronGarsiaSergelXin2015CompositionalShuffleConjectures})
    \[ J(A(L)) = A^{\text{old}}(J(L)) = (-1)^d A^{\text{old}}(L) = (-1)^d F_{ad,bd}, \]
    where $(a,b) = (0,1) A^t$ (so, $A$ maps the $(0,1)$ component of the elliptic Hall algebra to the $(a,b)$ component).
\end{remark}

We can now reprove \cite[Proposition~3.6.1]{BHMPS2023ShuffleAnyLine}, which is a special case of \cite[Proposition~6.7]{Negut2014ShuffleAlgebra}.

\begin{theorem}
    \label{thm:rational_shuffle}
    Let $m, n$ be positive integers, let $d = \gcd(m,n)$, and let $m = ad$, $n = bd$. For $1 \leq i \leq m$, let $\gamma_i = \lceil ib/a \rceil - \lceil (i-1)b/a \rceil$. Then
    \[ \D_\gamma = e_d \left[ -MX^{a,b} \right]. \] 
\end{theorem}

\begin{proof}
    Since $\gcd(a,b) = 1$, by Bézout's identity there exist unique integers $0 < x \leq a$ and $0 \leq y < b$ such that $xb - ya = 1$. Let
    \[ A_\gamma = \begin{pmatrix} x & a \\ y & b \end{pmatrix}  \in \mathsf{SL}_2(\mathbb{Z})_+. \]

    For $A = A_\gamma$ and $L = d_- (-y_1)^{d-1} d_+$, we have \[ L \cdot 1 = d_- (-y_1)^{d-1} d_+ \cdot 1 = d_- (-y_1)^{d-1} \z_1 d_+ \cdot 1 = \D_d \cdot 1 = e_d, \]
    so by \Cref{rmk:intertwine} we obtain
    \[ (-1)^{m+n} A_\gamma(L) = (-1)^d e_{m,n} = (-1)^d (-1)^{d(b+1)} e_d \left[ -MX^{a,b} \right] = (-1)^{n} e_d \left[ -MX^{a,b} \right], \]
    where the factor $(-1)^{d(b+1)}$ arises from \Cref{rmk:eha-notation}. Since $n = bd$, we have $(-1)^d (-1)^{d(b+1)} = (-1)^{db} = (-1)^n$, and we only need to show that $A_\gamma(L) = (-1)^m \D_\gamma$.
    
    The Euclidean algorithm guarantees that, by construction, $\gamma$ is such that the lowest lattice path above the diagonal from $(0,0)$ to $(m,n)$ in the lattice grid is exactly \[ P_\gamma = U^{\gamma_1} R U^{\gamma_2} R \cdots U^{\gamma_m} R, \]
    which consists of $d$ copies of the lowest lattice path above the diagonal from $(0,0)$ to $(a,b)$. By \Cref{lem:lowest_path}, $P_\gamma = (P_{A_\gamma})^d$.

    By definition, if, in $P_\gamma$, we replace each occurrence of $U$ with $-y_1$ (except the first one, which is replaced by $d_+$), each occurrence of $R$ with $-\z_1$, and add a $d_-$ to the left, we obtain $(-1)^m \D_\gamma$.

    Performing the same replacements in \Cref{prop:north_south} yields the assignments in \Cref{def:lowest_path}; applying them to $(P_{A_\gamma})^d$ yields $A_\gamma(L)$. Since $P_\gamma = (P_{A_\gamma})^d$, the two sequences of replacements yield $A_\gamma(L) = (-1)^m \D_\gamma$, as desired.
\end{proof}

The rational shuffle theorem now follows.

\begin{theorem}
    Let $m, n$ be positive integers, let $d = \gcd(m,n)$, and let $m = ad$, $n = bd$. Then
    \[ e_d \left[ -MX^{a,b} \right] \cdot 1 = \sum_{\pi \in \LD(m,n)} q^{\dinv(\pi)} t^{\area(\pi)} x^\pi, \]
    where the original definitions can be found in \cite[Conjecture~4.2]{BergeronGarsiaSergelXin2015CompositionalShuffleConjectures}.
\end{theorem}

\begin{proof}
    By \cite[Remark~3.25]{Mellit2021Rational}, the left-hand side coincides with the one from \cite[Conjecture~4.2]{BergeronGarsiaSergelXin2015CompositionalShuffleConjectures}; by \Cref{thm:rational_shuffle}, that is the same as $\D_\gamma \cdot 1$ for an appropriate $\gamma$; by \Cref{thm:Don1}, this is equal to
    \[ \omega \left( \mathbf{H}_{q,t}^{(\ell)}\left(\frac{x^\gamma}{\prod_{i=1}^{\ell-1}(1-qt x_i/x_{i+1})}\right)_{\mathrm{pol}} \right); \]
    finally, by the combinatorial argument in \cite[Theorem~5.5.1]{BHMPS2023ShuffleAnyLine}, this is equal to the right-hand side.
\end{proof}

\subsection{Combinatorics of \texorpdfstring{$\D_\gamma$}{Dgamma} when \texorpdfstring{$q=1$}{q=1}}

Recall that, for a word $w = w_1\cdots w_r$, its \emph{descent set} is the set
\[ \Des(w) \coloneqq \left\{ i \in [r-1] \mid w_i>w_{i+1} \right\} = \left\{d_1<\dots<d_\ell\right\} \]
and its \emph{descent composition} is the composition that records the lengths of the successive segments determined by the descents, namely
\[ \left(d_1,d_2-d_1,\dots,d_\ell-d_{\ell-1},r-d_\ell\right). \]
With an abuse of notation, we will denote the descent composition of $w$ by $\Des(w)$ as well, where the context will make it clear whether we are referring to the descent set or the descent composition.

\begin{example}
    For the word $w=34229845$, its descent set is $\{2,5,6\}$ and its descent composition is $(2,3,1,2)$.
\end{example}

Let $\gamma \in \mathbb{N}^\ell$ with $\lvert \gamma \rvert = n$, that is, $\gamma = (\gamma_1,\dots,\gamma_\ell) \vDash n$ is a weak composition of $n$ of length at most $\ell$ (disregarding possible trailing zeros). The expression $\left.\D_\gamma\cdot1\right\rvert_{q=1}$ has a known combinatorial interpretation \cite[Equation~(148)]{BHMPS2023ShuffleAnyLine}, which we recover from our definition.

\begin{definition}
    We define the \emph{path associated to $\gamma$}, denoted by $\delta(\gamma)$, as the lattice path in the grid $n\times\ell$, starting at $(0,0)$ and ending at $(\ell,n)$, that alternates $\gamma_i$ unit vertical steps and one unit horizontal step, for $i=1,\dots,\ell$.
\end{definition}

\begin{definition}
    We define the \emph{set of paths above $\delta\left(\gamma\right)$}, denoted by $\R(\gamma)$, as the set of lattice paths in the grid $n\times\ell$, starting at $(0,0)$ and ending at $(\ell,n)$, staying weakly above the path $\delta(\gamma)$.
\end{definition}

\begin{definition}
    Given an element $\tau\in \R(\gamma)$, we define its \emph{area}, denoted by $\area(\tau)$, as the number of whole squares between $\tau$ and $\delta(\gamma)$.
\end{definition}

\begin{remark}
    \label{rmk:heights}
    A lattice path $\rho$ can be encoded by the vector of the heights of its horizontal steps $(\rho_1,\dots,\rho_\ell)$. Unless otherwise specified, we will assume that the path lies in the $\rho_\ell\times\ell$ grid.

    With this notation, $\delta(\gamma)$ is the path $(\gamma_1,\gamma_1+\gamma_2,\dots,\gamma_1+\dots+\gamma_\ell)$, and a path $\tau=(\tau_1,\dots,\tau_\ell)$ belongs to $\R(\gamma)$ if and only if $\tau \geq \delta(\gamma)$ componentwise.
\end{remark}

\begin{definition}
    Given a lattice path $\rho$, we define its \emph{rise composition $\mu(\rho)$} as the composition that records how many consecutive vertical steps $\rho$ has. Namely, if $(\rho_1,\dots,\rho_\ell)$ is the vector of the heights of the horizontal steps of $\rho$, then $\mu(\rho)$ is the composition obtained from $(\rho_1,\rho_2-\rho_1,\dots,\rho_\ell-\rho_{\ell-1})$ after removing the zeros.
\end{definition}

In particular, $\gamma$ (stripped of the zeros) is the rise composition of $\delta(\gamma)$, and for any $\tau \in \R(\gamma)$, $\mu(\tau) \vDash n$ is a composition of $n$ of length at most $\ell$.

\begin{example}
    For the descent composition $\gamma=(2,3,1,2)$ of the word $w=34229845$ from the previous example, $\delta(\gamma)$ is the path $(2,5,6,8)$ in blue in \Cref{fig:small-path-example}. The red path in the same figure is an element $\tau=(4,6,6,8)$ of $\R(\gamma)$, with $\area(\tau)=3$ and rise composition $\mu(\tau)=(4,2,2)$.
\end{example}

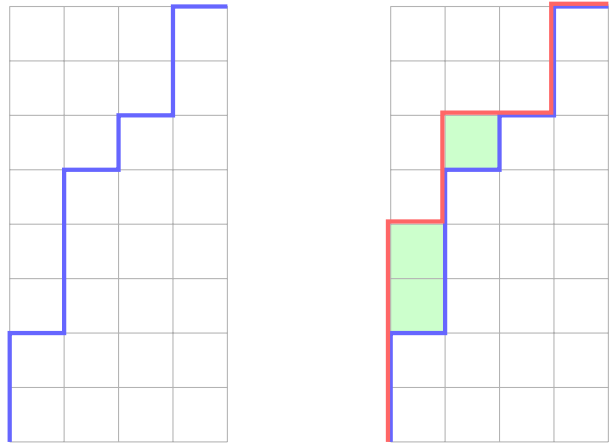
\begin{figure}[ht]
    \centering
    \begin{tikzpicture}[scale=.72]
        \draw[step=1.0, gray!60, thin] (0,0) grid (4,8);
        \draw[blue!60, line width = 1.6 pt] (0,0) -- (0,1) -- (0,2) -- (1,2) -- (1,3) --  (1,4) -- (1,5) -- (2,5) --(2,6) -- (3,6) -- (3,7) -- (3,8) -- (4,8); 
        \begin{scope}[shift={(7,0)}]
            \fill[green!20] (0,2) rectangle (1,3);
            \fill[green!20] (0,3) rectangle (1,4);
            \fill[green!20] (1,5) rectangle (2,6);
            \draw[step=1.0, gray!60, thin] (0,0) grid (4,8);
            \draw[blue!60, line width = 1.6 pt] (0,0) -- (0,1) -- (0,2) -- (1,2) -- (1,3) --  (1,4) -- (1,5) -- (2,5) --(2,6) -- (3,6) -- (3,7) -- (3,8) -- (4,8);
            \begin{scope}[shift={(-0.05,0.05)}]
                \draw[red!60, line width = 1.6 pt] (0,-0.05) -- (0,1) -- (0,4) -- (1,4) -- (1,6) --  (3,6) -- (3,7) -- (3,8) -- (4.05,8); 
            \end{scope}
        \end{scope}
    \end{tikzpicture}
    \caption{The path $\delta(\gamma)$ is shown in blue, and a path in $\R(\gamma)$ is shown in red. The three green squares contribute to the area statistic.}
    \label{fig:small-path-example}
\end{figure}

\begin{theorem}[{\cite[Equation~(148)]{BHMPS2023ShuffleAnyLine}}]
    \label{thm:D.1_comb}
    Let $\gamma=(\gamma_1,\dots,\gamma_\ell)\vDash n$ be a composition of $n$. Then
    \[
        \left.\D_\gamma\cdot1\right\rvert_{q=1} = \sum_{\tau\in \R(\gamma)} t^{\area(\tau)} e_{\mu(\tau)}.
    \]
\end{theorem}

\begin{proof}
    We proceed by double induction on $\ell(\gamma)$ and $\lvert \gamma \rvert - \gamma_1$. Let \[ \R_{t,x}(\gamma) = \sum_{\tau \in \R(\gamma)} t^{\area(\tau)} e_{\mu(\tau)}. \]

    If $\gamma = (m)$, we have $\D_{(m)} \cdot 1 = e_m$ and $\R_{t,x}((m)) = e_m$, so the base case holds.

    Let $\gamma = (m, \gamma')$. Consider the lattice cell $c$ with coordinates $(0, m)$, that is, the first cell in the first column that is above $\delta(\gamma)$.
    Then \[ \R_{t,x}(\gamma) = \R_{t,x}((m))\R_{t,x}((\gamma')) + t \R_{t,x}((m+1, \gamma'_-)), \] where the first summand corresponds to paths that do not lie above $c$, the second to paths that do, and the factor $t$ counts the contribution of $c$ to the area. By induction,
    \[
        \R_{t,x}(\gamma) = \left. \D_{(m)} \cdot 1 \right\rvert_{q=1} \left. \D_{\gamma'} \cdot 1 \right\rvert_{q=1} + t \left. \D_{(m+1, \gamma'_-)} \cdot 1 \right\rvert_{q=1}
        = \left. \D_\gamma \cdot 1 \right\rvert_{q=1}
    \]
    by \Cref{lem:Dgamma_q1} and \Cref{thm:dgamma-recursion}. The claim follows.
\end{proof}

\section{Theta operators}
\label{sec:Theta}

In this section, we define an endomorphism $\Theta$ of $\A_{q,t}$, giving its action on the generators and then showing that the images satisfy the appropriate commutation relations.

We then show that this operator extends the $\Theta_f$ operators from \cite{DAdderioIraciVandenWyngaerd2021ThetaOperators} to the whole $\A_{q,t}$, and use this extension to prove commutation relations with the $\D_\gamma$ operators.

As an additional check on the computations in this section, the GitHub repository \cite{Iraci2026Code} contains a Lean formalization of several finite algebraic identities from this section, assisted by Harmonic's \emph{Aristotle} \cite{Achim2025Aristotle}.

We emphasize that this is \emph{not} a formalization of the entire paper or of the complete construction in this section: in particular, it does not construct the concrete symmetric-function action or the level-graded endomorphism on the completion.
Rather, the formalization checks the correctness of our computations and demonstrates the feasibility of formalizing this kind of mathematics in Lean.

We hope that this will inspire other mathematicians to contribute to the formalization of algebraic combinatorics, which is currently underrepresented in the Lean mathematical library \cite{Mathlib2020}.

\subsection{The Theta operator as endomorphism of \texorpdfstring{$\A_{q,t}$}{Aqt}}

This subsection is devoted to the proof of the following theorem.

\begin{theorem}
    \label{thm:Theta_def}
    The assignments
    \begin{itemize}
        \item $\Theta(T_i) = T_i$,
        \item $\Theta(d_-) = d_-$,
        \item $\Theta(d_+) = (1 - u (1 + u y_1)^{-1} y_1 \z_1) d_+$,
        \item $\Theta(d_+^*) = (1 - u (1 + u y_1)^{-1} \z_1 y_1)^{-1} (1 + uy_1)^{-1} d_+^*= (1 + u (1 - \z_1) y_1)^{-1} d_+^*$,
    \end{itemize}
    extend to an endomorphism $\Theta \colon \A_{q,t}[[u]] \rightarrow \A_{q,t}[[u]]$.
\end{theorem}

It is convenient to define
\[ s \coloneqq u (1 + u y_1)^{-1} y_1 \z_1 \quad \text{ and } \quad s^* \coloneqq u (1 - \z_1) y_1, \] so that
\[ \Theta(d_+) = (1-s) d_+ \quad \text{ and } \quad \Theta(d_+^*) = (1+s^*)^{-1} d_+^*. \]

\begin{remark}
    \label{rmk:u}
    Since $u$ always appears paired with $y_1$ and vice versa, its sole role is to keep track of the increase in (the second component of) the degree: if we set $\A_{q,t}^{(a)}$ to be the submodule of $\A_{q,t}$ consisting of elements that are homogeneous of degree $a$, then
    \[ \Theta \left( \A_{q,t}^{(a)} \right) \in \bigoplus_{k \geq 0} u^k \A_{q,t}^{(a+k)}. \]
\end{remark}

By \Cref{rmk:u}, instead of defining $\Theta$ on $\A_{q,t}[[u]]$, it is equivalent to set $u=1$ and work in the completion of $\A_{q,t}$ with respect to the degree. Throughout this section, we will do so to avoid dealing with $u$ (which adds no information) and to lighten the notation. The parameter $u$ will be useful when restricting to symmetric functions.

\begin{remark}
    The operator $\Theta$ becomes the identity if $u=0$.
\end{remark}

Notice that
\begin{equation}
    d_+ s = q^{-1} T_1 s T_1 d_+ \quad \text{ and } \quad d_+^* s^* = q T_1^{-1} s^* T_1^{-1} d_+^*.
\end{equation}

\begin{lemma}
    The following commutation relations hold:
    \begin{align}
        [T_1, (1-s)(1 - q^{-1} T_1 s T_1)] & = 0, \label{eq:T_comm} \\
        [T_1, (1+s^*)^{-1}(1 + q T_1^{-1} s^* T_1^{-1})^{-1}] & = 0 .\label{eq:T_comm*}
    \end{align}
\end{lemma}

\begin{proof}
    It is easy to check the following identities:
    \begin{align}
        & \z_1 T_1 y_1 = y_2 \z_1 T_1 & & y_1 T_1^{-1} \z_1 = \z_2 y_1 T_1^{-1} \label{eq:z1_y1_comm} \\
        & q s T_1 s T_1 = (-y_1) (1+y_1)^{-1} (-y_2)(1+y_2)^{-1} \z_1 \z_2 & & q s^* T_1^{-1} s^* T_1^{-1} = (1 - \z_1) (1 - \z_2) y_1 y_2 \\
        & [T_1, y_1+y_2] = 0 & & [T_1, y_1 y_2] = 0 \label{eq:T1_symy_comm} \\
        & [T_1, \z_1+\z_2] = 0 & & [T_1, \z_1 \z_2] = 0 \label{eq:T1_symz_comm}
    \end{align}
    In particular, $T_1$ commutes with all symmetric polynomials in $y_1$ and $y_2$ (resp.\ in $\z_1$ and $\z_2$) and so 
    \begin{align*}
        & [T_1, q^{-1} s T_1 s T_1] = 0 &&  [T_1, q s^* T_1^{-1} s^* T_1^{-1}] = 0
    \end{align*}
    because $q s^* T_1^{-1} s^* T_1^{-1} = (1-\z_1)(1-\z_2)y_1y_2$.
    Moreover,
    \begin{align*}
        & (q^{-1}T_1-T_1^{-1})s = \frac{1-q}{q} s = s(q^{-1}T_1-T_1^{-1})
    \end{align*}
    equivalently
    \[q^{-1} T_1 s T_1 T_1^{-1} + s T_1^{-1} = T_1^{-1} s + T_1^{-1} q^{-1} T_1 s T_1 \]
    and so
    \[ [T_1, s + q^{-1} T_1 s T_1] = 0. \]
    An analogous argument shows $[T_1, q T_1^{-1} s^* T_1^{-1} + s^*] = 0$.
    This concludes the proof.
\end{proof}

\begin{lemma}
    \label{lemma:proof_Q1}
    We have the identity
    \[ T_{i+1 \searrow 1} (1+s^*)^{-1} \z_1 T_{1 \nearrow i+1} (1-s) = (1-s) T_{i+1 \searrow 1} \left( 1 + \left( 1 - \z_1 q^{-1} T^2_1 \right) y_1 \right)^{-1} \z_1 T_{1 \nearrow i+1}.\]
\end{lemma}

\begin{proof}
    The elements $T_j$ for $j>1$ commute with $y_1$ and $\z_1$, so it suffices to prove the case $i=1$:
    \begin{equation}
        \label{eq:Q1_step1}
        T_1(1+s^*)^{-1} \z_1 T_1 (1-s) = (1-s) T_1 \left( 1 + \left( 1 - \z_1 q^{-1} T^2_1 \right) y_1 \right)^{-1} \z_1 T_1.
    \end{equation}
    Notice that
    \begin{align*}
        \left( 1 + \left( 1 - \z_1 q^{-1} T^2_1 \right) y_1 \right)^{-1} \z_1 T_1 &= (1+y_1)^{-1} \left( 1 - \z_1 q^{-1} T^2_1  y_1 (1+y_1)^{-1} \right)^{-1} \z_1 T_1 \\
        & = (1+y_1)^{-1} \z_1 T_1 ( 1 - q^{-1} T_1 s T_1)^{-1} 
    \end{align*}
    Our claim \eqref{eq:Q1_step1} becomes
    \begin{equation}
        \label{eq:Q1_step2}
        T_1(1+s^*)^{-1} \z_1 T_1 (1-s) (1 - q^{-1} T_1 s T_1) = (1-s) T_1 (1+y_1)^{-1} \z_1 T_1
    \end{equation}
    and the left-hand side of \eqref{eq:Q1_step2} is
    \begin{align*}
        T_1(1+s^*)^{-1} \z_1 T_1 (1-s) (1 - q^{-1} T_1 s T_1) &= T_1(1+s^*)^{-1} \z_1 (1-s) (1 - q^{-1} T_1 s T_1) T_1  \\
        &= T_1 (1 + y_1)^{-1} \z_1 ( 1 - q^{-1} T_1 s T_1) T_1 
    \end{align*}
    by \eqref{eq:T_comm} and \eqref{eq:z1(1-s)}.

    Now \eqref{eq:Q1_step2} simplifies to
    \[ T_1 (1 + y_1)^{-1} \z_1 (1 - q^{-1} T_1 s T_1) = (1-s) T_1 (1 + y_1)^{-1} \z_1 \]
    which is equivalent to 
    \[ T_1 (1 + y_1)^{-1} \z_1 q^{-1} T_1 s T_1 = s T_1 (1 + y_1)^{-1} \z_1. \]
    Up to a sign, this expands as
    \begin{equation}
        \label{eq:Q1_step3}
        T_1 (1 + y_1)^{-1} \z_1 q^{-1} T_1 y_1 (1 + y_1)^{-1} \z_1 T_1 = y_1 (1 + y_1)^{-1} \z_1 T_1 (1 + y_1)^{-1} \z_1.
    \end{equation}
    Using \eqref{eq:z1_y1_comm} the left-hand side of \eqref{eq:Q1_step3} is
    \begin{align*}
        T_1 (1 + y_1)^{-1} \z_1 q^{-1} T_1 y_1 (1 + y_1)^{-1} \z_1 T_1  & = q^{-1} T_1 y_2 (1 + y_1)^{-1} \z_1  T_1  (1 + y_1)^{-1} \z_1 T_1 \\
        & = y_1 T_1^{-1} (1 + y_1)^{-1} \z_1  T_1 (1 + y_1)^{-1} \z_1 T_1
    \end{align*}
    It remains to prove that $[T_1, (1 + y_1)^{-1} \z_1 T_1 (1 + y_1)^{-1} \z_1] = 0$.
    Finally,
    \begin{align*}
        (1 + y_1)^{-1} \z_1 T_1 (1 + y_1)^{-1} \z_1 &= (1 + y_1)^{-1} (1 + y_2)^{-1} \z_1 T_1 \z_1 \\
        & = q (1 + y_1)^{-1} (1 + y_2)^{-1} \z_1 \z_2 T_1^{-1}
    \end{align*}
    and $T_1$ commutes with symmetric polynomials in $y_1, y_2$ and $\z_1, \z_2$. In particular, it commutes with $(1 + y_1)^{-1} (1 + y_2)^{-1}$ and $\z_1  \z_2$, and this completes the proof.
\end{proof}

\begin{lemma} 
    \label{lemma:proof_Q1*}
    We have the identity
    \[ qT^*_{i+1 \searrow 1} (1-s) y_1 T^*_{1 \nearrow i+1} (1+s^*)^{-1} = (1+s^*)^{-1} T^*_{i+1 \searrow 2} \left( 1 - y_2 (1 + y_2)^{-1} T_1^{-1} \z_1 T_1 \right) y_2 T^*_{2 \nearrow i+1}.\]
\end{lemma}

\begin{proof}
    The elements $T_j$ for $j>1$ commute with $y_1$ and $\z_1$, so it suffices to prove the case $i=1$:
    \begin{equation}
        \label{eq:Q1*_step1}
        qT^{-1}_1 (1-s) y_1 T^{-1}_1 (1+s^*)^{-1} = (1+s^*)^{-1}  \left(1 - y_2 (1 + y_2)^{-1} T_1^{-1} \z_1 T_1 \right) y_2
    \end{equation} 
    Notice that
    \begin{align*}
        \left( 1 - y_2 (1 + y_2)^{-1} T_1^{-1} \z_1 T_1 \right) y_2 
        &= (1 + y_2)^{-1} \left( 1 + y_2 \left( 1 - T_1^{-1} \z_1 T_1 \right) \right)  y_2 \\
        &= (1 + y_2)^{-1} y_2 \left( 1 + \left( 1 - T_1^{-1} \z_1 T_1 \right) y_2 \right)  \\
        &= (1 + y_2)^{-1} y_2 \left( 1 + qT_1^{-1} s^* T_1^{-1} \right). 
    \end{align*}
    Our claim \eqref{eq:Q1*_step1} becomes
    \begin{equation}
        \label{eq:Q1*_step2}
        qT^{-1}_1 (1-s) y_1 T^{-1}_1 (1+s^*)^{-1} \left( 1 + qT_1^{-1} s^* T_1^{-1} \right)^{-1} = (1+s^*)^{-1}  (1 + y_2)^{-1} y_2
    \end{equation} 
    and the left-hand side of \eqref{eq:Q1*_step2} is
    \[ qT^{-1}_1 (1-s) y_1 T^{-1}_1 (1+s^*)^{-1} \left( 1 + qT_1^{-1} s^* T_1^{-1} \right)^{-1} = T_1^{-1} (1 + y_1)^{-1} y_1 \left( 1 + qT_1^{-1} s^* T_1^{-1} \right)^{-1} T^{-1}_1 \]
    by \eqref{eq:T_comm*} and \eqref{eq:y1(1+s*)}.
    Now \eqref{eq:Q1*_step2} simplifies to
    \begin{equation}
        \label{eq:Q1*_step3}
        qT_1^{-1} (1 + y_1)^{-1} y_1 = (1+s^*)^{-1} (1 + y_2)^{-1} y_2 T_1 \left( 1 + qT_1^{-1} s^* T_1^{-1} \right)
    \end{equation}
    and the right-hand side of \eqref{eq:Q1*_step3} can be manipulated using the identity
    \begin{align*}
        y_2 T_1 \left( 1 + qT_1^{-1} s^* T_1^{-1} \right)
        & = q T_1^{-1} y_1 \left( 1 + qT_1^{-1} s^* T_1^{-1} \right) \\
        & = q T_1^{-1}  \left( 1 + (1 - \z_2) y_2 \right) y_1 
    \end{align*}
    which follows by \eqref{eq:z1_y1_comm}.
    We will prove the stronger equality
    \[ T_1^{-1} (1 + y_1)^{-1} =  (1+s^*)^{-1}  (1 + y_2)^{-1} T_1^{-1}  \left( 1 + (1 - \z_2) y_2 \right) \]
    which is equivalent to
    \begin{equation}
        \label{eq:Q1*_step4}
        T_1 (1 + y_2) (1+s^*) = \left( 1 + (1 - \z_2) y_2 \right) (1 + y_1) T_1.
    \end{equation}
    The above equation \eqref{eq:Q1*_step4} follows from the identities
    \begin{align}
        T_1 (1 + y_2) (1 + y_1) &= (1 + y_2) (1 + y_1) T_1 \\
        T_1 (1 + y_2) \z_1 y_1 &= \z_2 y_2 (1 + y_1) T_1
    \end{align}
    The former identity holds because $T_1$ commutes with symmetric polynomials in $y_1$ and $y_2$, see \eqref{eq:T1_symy_comm}.
    The latter holds because $T_1 \z_1 y_1 = \z_2 y_2 T_1$ and $T_1 y_2 \z_1 y_1 = T_1 \z_1 T_1 y_1 T_1^{-1} y_1 = \z_2 y_1 y_2 T_1$ (by \eqref{eq:z1_y1_comm}).
    This completes the proof.
\end{proof}

\begin{lemma}
    \label{lem:theta_q2}
    We have the identities
    \begin{equation}
        \label{eq:y1(1+s*)}
        y_1 (1 + s^*) = (1 + y_1) (1-s) y_1
    \end{equation}
    and
    \begin{equation}
        \label{eq:z1(1-s)}
        \z_1 (1-s) = (1 + s^*) (1 + y_1)^{-1} \z_1.
    \end{equation}
\end{lemma}

\begin{proof}    
    Recall that \[ s = (1 + y_1)^{-1} y_1 \z_1 \quad \text{ and } \quad s^* = (1 - \z_1) y_1. \]
    We have

    \begin{minipage}{0.48\textwidth}
        \begin{align*}
            y_1 (1 + s^*) & = y_1 (1 + (1 - \z_1) y_1) \\
            & = y_1 (1 + y_1 - \z_1 y_1) \\
            & = (1 + y_1 - y_1 \z_1) y_1 \\
            & = (1 + y_1) (1 - s) y_1
        \end{align*}      
    \end{minipage}%
    \hfill
    \begin{minipage}{0.48\textwidth}
        \begin{align*}
            \z_1 (1-s) & = \z_1 (1 - (1 + y_1)^{-1} y_1 \z_1) \\
            & = (1 - \z_1 y_1 (1 + y_1)^{-1}) \z_1 \\
            & = (1 + y_1 - \z_1 y_1) (1 + y_1)^{-1} \z_1 \\
            & = (1 + s^*) (1 + y_1)^{-1} \z_1
        \end{align*}        
    \end{minipage}

    as expected.
\end{proof}

\begin{proof}[Proof of {\Cref{thm:Theta_def}}]
    We need to check that $\Theta$ preserves the relations from \Cref{def:Aq,def:Aqt}.
    
    Relations \eqref{eq:skein}, \eqref{eq:braid}, \eqref{eq:commuting}, \eqref{eq:r1}, and \eqref{eq:r4} are preserved since $\Theta$ acts as the identity on $T_i$ and $d_-$. We now go through the remaining relations.

    For~\eqref{eq:r2}, we have
    \begin{align*}
        \Theta(d_+ T_i) & = (1-s) d_+ T_i \\
        & = (1-s) T_{i+1} d_+  \\
        & = T_{i+1} (1-s) d_+ \\
        & = \Theta(T_{i+1} d_+).
    \end{align*}

    For~\eqref{eq:r3}, we have
    \begin{align*}
        \Theta(T_1 d_+^2) & = T_1(1-s) d_+ (1-s) d_+ \\
        & = T_1 (1-s)(1-q^{-1}T_1sT_1) d_+^2 \\
        & = (1-s)(1 - q^{-1} T_1 s T_1) T_1 d_+^2 \\
        & = (1-s)(1 - q^{-1} T_1 s T_1) d_+^2 \\
        & = (1-s) d_+ (1-s) d_+ \\
        & = \Theta(d_+^2).
    \end{align*}

    For~\eqref{eq:r5}, we have
    \begin{align*}
        \Theta(d_- [d_+, d_-] T_{k-1}) & = d_- [(1-s) d_+, d_-] T_{k-1} \\
        & = (1-s) d_- [d_+, d_-] T_{k-1} \\
        & = (1-s) q [d_+, d_-] d_- \\
        & = q [(1-s) d_+, d_-] d_- \\
        & = \Theta(q [d_+, d_-] d_-).
    \end{align*}

    For~\eqref{eq:r6}, we have
    \begin{align*}
        \Theta(T_1 [d_+, d_-] d_+) & = T_1 [(1-s) d_+, d_-] (1-s) d_+ \\
        & = T_1^{-1}(1-s)(1 - q^{-1} T_1 s T_1) [d_+,d_-] d_+ \\
        & = (1-s)(1 - q^{-1} T_1 s T_1) T_1^{-1} [d_+,d_-] d_+ \\
        & = (1-s)(1 - q^{-1} T_1 s T_1) q d_+[d_+,d_-] \\
        & = q (1-s) d_+ [(1-s) d_+, d_-] \\
        & = \Theta(q d_+ [d_+, d_-]).
    \end{align*}

    The proofs of the dual relations for $\A_{q^{-1}}$ are analogous to those for \eqref{eq:r1}--\eqref{eq:r6}, and we omit them.

    For~\eqref{eq:Q1}, we have
    \begin{align*}
        \Theta(\z_{i+1} d_+)
        & = q^{-i} T_{i+1 \searrow 1} (1+s^*)^{-1} \z_1 T_{1 \nearrow i+1}  (1-s) d_+ \\
        & = q^{-i} (1-s) T_{i+1 \searrow 1} \left( 1 + \left( 1 - \z_1 q^{-1} T^2_1 \right) y_1 \right)^{-1} \z_1 T_{1 \nearrow i+1} d_+ \\
        & = (1-s) T_{i+1 \searrow 1} \left( 1 + \left( 1 - \z_1 q^{-1} T^2_1 \right) y_1 \right)^{-1}  T_{1 \nearrow i+1}^* \z_{i+1} d_+ \\
        & = (1-s) d_+ T_{i \searrow 1} (1+s^*)^{-1} T_{1 \nearrow i}^* \z_i \\
        & = \Theta(d_+ \z_i).
    \end{align*}
    
    For~\eqref{eq:Q1*}, we have
    \begin{align*}
        \Theta(y_{i+1} d^*_+)
        & = q^{i} T_{i+1 \searrow 1}^* (1-s) y_1 T_{1 \nearrow i+1}^* (1+s^*)^{-1} d^*_+ \\
        & = q^{i-1} (1+s^*)^{-1} T^*_{i+1 \searrow 2} \left( 1 - y_2 (1 + y_2)^{-1} T_1^{-1} \z_1 T_1 \right) y_2 T^*_{2 \nearrow i+1} d^*_+ \\
        & = (1+s^*)^{-1} T^*_{i+1 \searrow 2} \left( 1 - y_2 (1 + y_2)^{-1} T_1^{-1} \z_1 T_1 \right) T^*_{2 \nearrow i+1} y_{i+1} d^*_+ \\
        & = (1+s^*)^{-1} d^*_+ T^*_{i \searrow 1} (1-s) T^*_{1 \nearrow i} y_{i} \\
        & = \Theta(d^*_+ y_i).
    \end{align*}

    For~\eqref{eq:Q2}, we have
    \begin{align*}
        \Theta(\z_1 d_+) & = (1+s^*)^{-1} \z_1 (1-s) d_+ \\
        & = (1 + y_1)^{-1} \z_1 d_+ && \text{by } \eqref{eq:z1(1-s)} \\
        & = -q^k (1 + y_1)^{-1} y_1 d_+^* && \text{by } \eqref{eq:Q2} \\
        & = -q^k (1-s) y_1 (1+s^*)^{-1} d_+^* && \text{by } \eqref{eq:y1(1+s*)} \\
        & = \Theta(-q^k y_1 d_+^*).
    \end{align*}

    This concludes the proof.
\end{proof}

\begin{corollary}
    \label{cor:theta_yz}
    We have the identities \[ \Theta(\z_1 d_+) = (1+y_1)^{-1} \z_1 d_+ \quad \text{ and } \quad \Theta(y_1 d_+^*) = (1+y_1)^{-1} y_1 d_+^*. \]
    The same identities hold after replacing $d_+$ with $y_1$ and $d_+^*$ with $\z_1$.
\end{corollary}

\begin{proof}
    Using \Cref{lem:theta_q2}, we have
    \[ \Theta(\z_1 d_+) = (1+s^*)^{-1} \z_1 (1-s) d_+ = (1+s^*)^{-1} (1 + s^*) (1 + y_1)^{-1} \z_1 d_+ = (1+y_1)^{-1} \z_1 d_+ \]
    and
    \[ \Theta(y_1 d_+^*) = (1-s) y_1 (1+s^*)^{-1} d_+^* = (1 + y_1)^{-1} y_1 (1 + s^*) (1+s^*)^{-1} d_+^* = (1+y_1)^{-1} y_1 d_+^*, \]
    as desired. The identities obtained by replacing $d_+$ with $y_1$ and $d_+^*$ with $\z_1$ follow immediately because $d_-$ commutes with $y_1$ and $\z_1$.
\end{proof}

\begin{proposition}
    The assignment
    \[ \Theta(Y) = (1+s^*)^{-1} (1+y_1) Y \]
    extends $\Theta$ to an endomorphism of $\A_{q,t}^\pm[[u]]$. 
\end{proposition}

\begin{proof}
    We only need to check that $\Theta(y_1 Y) = \Theta(Y y_1) = 1$. Indeed,
    \begin{align*}
        \Theta(y_1 Y) & = (1-s) y_1 (1+s^*)^{-1} (1+y_1) Y \\
        & = (1+y_1)^{-1} y_1 (1+s^*) (1+s^*)^{-1} (1+y_1) Y \\
        & = y_1 Y = 1 \\
    \end{align*}
    and
    \begin{align*}
        \Theta(Y y_1) & = (1+s^*)^{-1} (1+y_1) Y (1-s) y_1 \\
        & = (1+s^*)^{-1} (1+y_1) Y y_1 (1+y_1)^{-1} (1+s^*) \\
        & = (1+s^*)^{-1} (1+y_1) (1+y_1)^{-1} (1+s^*) = 1 \\
    \end{align*}
    as desired.    
\end{proof}

\subsection{The action on symmetric functions}

We will now show that this operator induces an operator on symmetric functions, which we use to recover the $\Theta_f$ operator from \cite[Equation~(28)]{DAdderioIraciVandenWyngaerd2021ThetaOperators}.

We start with the following lemma.

\begin{lemma}
    \label{lem:s*_acts_as_0}
    For $k \geq 1$, we have $s^* (d_+^*)^k \epsilon_0 = 0$.
\end{lemma}

\begin{proof}
    We have $s^* = (1 - \z_1) y_1$, so it suffices to show that $\z_1$ acts as the identity on $y_1 (d_+^*)^k \epsilon_0$. Indeed,
    \begin{align*}
        \z_1 y_1 (d_+^*)^k \epsilon_0
        & = -q^{1-k} \z_1 d_+ (d_+^*)^{k-1} \epsilon_0 && \text{by \eqref{eq:I2}} \\
        & = y_1 (d_+^*)^k \epsilon_0 && \text{by \eqref{eq:Q2}} \\
    \end{align*}
    as desired.
\end{proof}

The following result proves that $\Theta$ induces an operator on $V$.

\begin{lemma}
    \label{lem:preserve_kernel}
    The operator $\Theta$ preserves the kernel of the action of $\A_{q,t}$ on $V$.
\end{lemma}

\begin{proof}
    We need to show that $\Theta$ preserves the relations \eqref{eq:I1} and \eqref{eq:I2}.

    First, we check \eqref{eq:I1}. We have
    \[ \Theta((d_- d_+^* - 1) (d_+^*)^k \epsilon_0) = (d_- (1+s^*)^{-1} d_+^* - 1) ((1+s^*)^{-1} d_+^*)^k \epsilon_0, \]
    but by \Cref{lem:s*_acts_as_0}, we may omit $s^*$ throughout and obtain
    \[ \Theta((d_- d_+^* - 1) (d_+^*)^k \epsilon_0) = (d_- d_+^* - 1) (d_+^*)^k \epsilon_0 = 0, \]
    so the relation is exactly preserved.
    
    Next, we check \eqref{eq:I2}. We have
    \[ \Theta((d_+ + q^k y_1 d_+^*) d_+^{*k} \epsilon_0) = ((1-s)d_+ + q^k (1-s) y_1 (1+s^*)^{-1} d_+^*) ((1+s^*)^{-1} d_+^*)^k \epsilon_0, \]
    and again by \Cref{lem:s*_acts_as_0} we may omit $s^*$ throughout and obtain
    \[ \Theta((d_+ + q^k y_1 d_+^*) d_+^{*k} \epsilon_0) = (1-s) (d_+ + q^k y_1 d_+^*) (d_+^*)^k \epsilon_0 = 0, \]
    as desired.
\end{proof}

\begin{theorem}
    \label{thm:theta_e1_D1}
    For any $L \in \epsilon_0 \A_{q,t} \epsilon_0$, we have the identity
    \[ \Theta(L) = \Theta(u) L \Theta(u)^{-1} \]
    as operators acting on $\Lambda$, where the product is the composition of operators and
    \[ \Theta(u) = \sum_{n \geq 0} u^n \Theta_{e_n} \qquad \text{ and } \qquad \Theta(u)^{-1} = \sum_{n \geq 0} (-u)^n \Theta_{h_n}. \]
\end{theorem}

\begin{proof}
    As in \Cref{rmk:u}, the parameter $u$ merely keeps track of the degree, so it is enough to prove the statement when $u=1$.

    By \Cref{lem:preserve_kernel}, $\Theta$ induces an endomorphism of $V$, denoted by $\Theta'$, such that $\Theta' L = \Theta(L) \Theta'$ as operators on $V$. 

    Denote by $e_1^\bullet$ the operator of multiplication by $e_1$. It is easy to check that, on $V_0$, $d_- d_+ = e_1^\bullet$.
    
    Applying this identity to $e_1^\bullet$ and $\D_1$, we obtain on $V_0$
    \begin{align*}
        \Theta' e_1^\bullet & = \Theta' d_- d_+ \\
        & = d_- (1 - (1+y_1)^{-1} y_1 \z_1) d_+ \Theta' \\
        & = d_- d_+ \Theta' + d_- (1+y_1)^{-1} (-y_1) \z_1 d_+ \Theta' \\
        & = d_- d_+ \Theta' + d_- (1+y_1)^{-1} (-y_1)^2 d_+^* \Theta' \\
        & = e_1^\bullet \Theta' + \sum_{n \geq 2} \D_n \Theta',
    \end{align*}
    and
    \[ \Theta' \D_1 = \Theta' d_- (-y_1) d_+^* = d_- (1+y_1)^{-1} (-y_1) d_+^* \Theta' = \sum_{n \geq 1} \D_n \Theta', \]
    so \[ [\Theta', e_1^\bullet] = \sum_{n \geq 2} \D_n \Theta' \quad \text{ and } \quad [\Theta', \D_1] = \sum_{n \geq 2} \D_n \Theta'. \]
    By \cite[Proposition~10.1]{DAdderioIraciVandenWyngaerd2021ThetaOperators} and \cite[Proposition~1.2]{Romero2022}, and up to our redefinition of the signs, these are exactly the commutation relations satisfied by $\Theta(u)$ (the operator on $\Lambda$). Since $e_1^\bullet$ and $\D_1$ generate $\Lambda$ \cite[Theorem~2.1]{GarsiaHaimanTesler1999ExplicitPlethysticFormulas}, we deduce that $\Theta = \Theta'$, as desired.
\end{proof}

\subsection{Relations between \texorpdfstring{$\Theta$}{Theta} and \texorpdfstring{$\D_\gamma$}{Dgamma}}
\label{section:theta_D}

The goal of this subsection is to derive the commutation relations between $\Theta$ and $\D_\gamma$, which we will then exploit to prove the Theta conjecture.

\begin{definition}
    Let $\alpha \vDash n$ and $\ell = \ell(\alpha)$. We define 
    \[ w(\alpha) = \mathsf{rev}(1 \cdots 1\, 2 \cdots 2\, 3 \cdots 3\, \cdots\, (\ell-1) \cdots (\ell-1)\, \ell \cdots \ell), \]
    where $i$ occurs $\mathsf{rev}(\alpha)_i$ times, and $\mathsf{rev}$ denotes the reverse of a word.
\end{definition}

For instance, if $\alpha=(2,3)$, then $\mathsf{rev}(\alpha)=(3,2)$, hence $w(\alpha)=\mathsf{rev}(11122)=22111$.

Given two words $v$ and $w$ without common letters, we denote by $v \shuffle w$ the set of all words obtained by \emph{shuffling} the letters of $v$ with those of $w$, while preserving the relative order in the original words. 

For instance, if $v=21$ and $w=\textcolor{red}{3}\textcolor{red}{4}$, then
\[v \shuffle w = \{21\textcolor{red}{3}\textcolor{red}{4}, 2\textcolor{red}{3}1\textcolor{red}{4},2\textcolor{red}{3}\textcolor{red}{4}1, \textcolor{red}{3}21\textcolor{red}{4},\textcolor{red}{3}2\textcolor{red}{4}1,\textcolor{red}{3}\textcolor{red}{4}21\}. \]
Finally, we will also use the standard exponential notation for repeated letters. For example, we can denote the word $1111221332$ in the alphabet $\{1,2,3\}$ also by $1^42^21^13^22^1$, or $1^42^213^22$.

\begin{theorem}
    \label{thm:shuffle_theta_D}
    Let $\gamma = (\gamma_1, \dots, \gamma_\ell) \vDash n$. Then
    \[ \Theta \D_\gamma = \sum_{k=0}^\infty \sum_{\substack{w \in (\ell+1)^k \shuffle w(\gamma) \\ w_1 = \ell}} \D_{\Des(w)} \Theta. \]
    In particular, when applied to $1$, $\Theta_{e_k} \D_\gamma \cdot 1$ is the sum over all $\D_{\tilde{\gamma}} \cdot 1$ for which ${\tilde{\gamma}}$ is the descent composition of a word obtained by shuffling $k$ occurrences of $\ell+1$ into the word $w(\gamma)$.
\end{theorem}

\begin{example}
    Consider our operators applied to the constant $1$ and a fixed degree $k$. If $\gamma = (1,2)$ and $k=2$, we have $w(\gamma) = 211$, and the set of words appearing in the sum is \[ 2|1133, \; 2|13|13, \; 23|113, \; 2|133|1, \; 23|13|1, \; 233|11, \] where the vertical bars denote descents. It follows that \[ \Theta_{e_2} \D_{12} \cdot 1 = (\D_{14} + \D_{122} + \D_{23} + \D_{131} + \D_{221} + \D_{32}) \cdot 1. \]
\end{example}

\begin{proof}
    By definition, \[ \D_\gamma = d_- (-y_1)^{\gamma_1-1} \z_1 (-y_1)^{\gamma_2} \z_1 \cdots (-y_1)^{\gamma_\ell} \z_1 d_+. \]
    If we think of $\gamma$ as the descent composition of a word, then each $(-y_1)$ corresponds to a letter; we must add an initial letter, and there is an occurrence of $\z_1$ whenever there is a descent or at the end of the word.

    By \Cref{cor:theta_yz}, we have 
    $\Theta(\z_1 y_1) = (1+y_1)^{-1} \z_1 y_1$ and $\Theta(\z_1 d_+) = (1+y_1)^{-1} \z_1 d_+$, so it is convenient to rewrite
    \[ \D_\gamma = d_- (-y_1)^{\gamma_1-1} \z_1 (-y_1) (-y_1)^{\gamma_2-1} \cdots \z_1 (-y_1) (-y_1)^{\gamma_\ell-1} \z_1 d_+. \] Thus,
    \begin{align*} 
        \Theta (\D_\gamma) & =\Theta(d_-) \Theta(-y_1)^{\gamma_1-1} \Theta(\z_1 (-y_1)) \Theta(-y_1)^{\gamma_2-1} \cdots \Theta(\z_1 (-y_1)) \Theta(-y_1)^{\gamma_\ell-1} \Theta(\z_1 d_+) \\
        & = d_- \Theta(-y_1)^{\gamma_1-1} (1+y_1)^{-1} \z_1 (-y_1) \cdots \Theta(-y_1)^{\gamma_\ell-1} (1+y_1)^{-1} \z_1 d_+.
    \end{align*}
    The factors $(1+y_1)^{-1}$ let us insert an arbitrary number of occurrences of $(-y_1)$ before each $\z_1$, and these correspond to inserting an arbitrary number of occurrences of $\ell+1$ before each descent or at the end of the word.

    Consider a single occurrence of $\Theta(-y_1)$. This occurrence corresponds to a letter that is not a descent, as we have already isolated the last occurrence of each letter. We have
    \[ \Theta(-y_1) = (1 - (1 + y_1)^{-1} y_1 \z_1) (-y_1), \]
    which means that the term $-y_1$ can remain fixed or be preceded by a positive number (say, $r$) of occurrences of $(-y_1)$ and then a $\z_1$. In the word, this corresponds to inserting $r$ occurrences of $\ell+1$ into that position, creating a descent. Insertion at the starting position is not allowed because there is no $y_1$ that matches the first letter of the word. This completes the proof.
\end{proof}

Given $l,k \in \mathbb{N}$, we denote by $W\left(0^l,1^k\right)$ the set of all words in the alphabet $\{0,1\}$ consisting of $l$ letters equal to $0$ and $k$ letters equal to $1$.

\begin{corollary}
    \label{cor:theta_theta_e1}
    For $k, l \in \mathbb{N}$, we have
    \[ \Theta_{e_k} \Theta_{e_l} e_1 = \sum_{w\in W\left(0^l,1^k\right)} \D_{\Des(0w)} \cdot 1. \]
\end{corollary}

\begin{proof}
    We have 
    \begin{align*}
    \Theta_{e_l} e_1 & =\Theta_{e_l} e_1^\bullet \cdot 1=\langle u^l\rangle \Theta(d_-d_+) \cdot 1 \\
    & = \langle u^l \rangle d_- (1 - u(1+uy_1)^{-1} y_1\z_1) d_+ \cdot 1 \\
    & = d_-(-y_1)^l\z_1 d_+\cdot 1 \\
    & = d_-(-y_1)^{l+1}d_+^* \cdot 1 \\
    & = \D_{(l+1)} \cdot 1.
    \end{align*}
    By \Cref{thm:shuffle_theta_D}, up to a shift in the letters,
    \[ \Theta_{e_k} \D_{(l+1)} \cdot 1 = \sum_{\substack{w \in 1^k \shuffle \, 0^{l+1} \\ w_1 = 0}} \D_{\Des(w)} \cdot 1 
        = \sum_{w\in W\left(0^l,1^k\right)} \D_{\Des(0w)} \cdot 1, \]
    as desired.
\end{proof}

\section{The Theta conjecture}
\label{sec:theta_conjecture}

We now use the results developed in the previous sections to prove the Theta conjecture \cite[Conjecture~9.1]{DAdderioIraciVandenWyngaerd2021ThetaOperators}. Namely, we prove the following result.

\begin{theorem}
    \label{thm:theta_q1}
    For $n,k,l \in \mathbb{N}$ with $n > k+l$, we have
    \[ \left.\Theta_{e_k} \Theta_{e_l} \nabla e_{n-k-l} \right\rvert_{q=1} = \sum_{\pi \in \LD(n)^{\ast k, \bullet l}} t^{\area(\pi)} x^\pi. \]
\end{theorem}

The proof proceeds as follows. First, we give all the combinatorial definitions appearing in this theorem; then, we show that the statement holds for $k+l=n-1$; next, we use this result to prove the statement for paths with touch number $n-k-l$; finally, we give a combinatorial interpretation of a symmetric function identity that allows us to complete the proof.

\subsection{The combinatorics of the Theta conjecture}

We begin by giving precise definitions for the combinatorics of the Delta and Theta conjectures.

\begin{definition}
    A \emph{Dyck path} of size $n$ is a lattice path starting at $(0,0)$, ending at $(n,n)$, using only unit up (vertical) steps and right (horizontal) steps, and staying weakly above the line $x=y$. A \emph{labeled Dyck path} is a Dyck path $\pi$ together with a labeling $w \colon [n] \to \mathbb{Z}_+$ of its vertical steps, such that consecutive vertical steps have strictly increasing labels (from bottom to top). We will draw the labels of the vertical steps in the square immediately to the right of each step.

    A \emph{rise} of a labeled Dyck path is a vertical step that is preceded by another vertical step.

    A \emph{valley} of a labeled Dyck path is a vertical step $v$ preceded by at least one horizontal step. A valley $v$ is \emph{contractible} if it is preceded by either two horizontal steps or a horizontal step that is itself preceded by a vertical step whose label is strictly smaller than $v$'s label.

    A \emph{decorated labeled Dyck path} $\pi$ is a labeled Dyck path together with a choice of rises and contractible valleys to be \emph{decorated}.
    We set
    \begin{align*}
         & \mathsf{dr}(\pi) = \{i \in [n] \mid \text{the $i\th$ vertical step of $\pi$ is a decorated rise}\}     \\
         & \mathsf{dv}(\pi) = \{i \in [n] \mid \text{the $i\th$ vertical step of $\pi$ is a decorated valley}\}.
    \end{align*}
    We decorate rises with a $\ast$ and valleys with a $\bullet$, and these decorations are displayed in the square to the left of the vertical step. The set of decorated labeled Dyck paths of size $n$ with $k$ decorated rises and $l$ decorated valleys is denoted by $\LD(n)^{\ast k, \bullet l}$.
\end{definition}

See \Cref{fig:path-example} for an example of an element of $\LD(8)^{\ast 2, \bullet 2}$.

\begin{figure}
    \centering
    \begin{tikzpicture}[scale=.72]
        \draw[step=1.0, gray!60, thin] (0,0) grid (8,8);
        \draw[gray!60, thin] (0,0) -- (8,8);
        \draw[blue!60, line width = 1.6 pt] (0,0) -- (0,1) -- (0,2) -- (1,2) -- (1,3) -- (2,3) -- (2,4) -- (2,5) -- (2,6) -- (3,6) -- (4,6) -- (4,7) -- (5,7) -- (6,7) -- (7,7) -- (7,8) -- (8,8);
        \node at (0.5,0.5) {$2$};
        \draw (0.5,0.5) circle (.4cm);
        \node at (0.5,1.5) {$3$};
        \draw (0.5,1.5) circle (.4cm);
        \node at (1.5,2.5) {$4$};
        \draw (1.5,2.5) circle (.4cm);
        \node at (2.5,3.5) {$1$};
        \draw (2.5,3.5) circle (.4cm);
        \node at (2.5,4.5) {$2$};
        \draw (2.5,4.5) circle (.4cm);
        \node at (2.5,5.5) {$4$};
        \draw (2.5,5.5) circle (.4cm);
        \node at (4.5,6.5) {$3$};
        \draw (4.5,6.5) circle (.4cm);
        \node at (7.5,7.5) {$2$};
        \draw (7.5,7.5) circle (.4cm);
        \node at (1-1-0.5,1+0.5) {$\ast$};
        \node at (5-3-0.5,5+0.5) {$\ast$};
        \node at (2-1-0.5,2+0.5) {$\bullet$};
        \node at (6-2-0.5,6+0.5) {$\bullet$};
    \end{tikzpicture}
    \caption{An element of $\LD(8)^{\ast 2, \bullet 2}$.}
    \label{fig:path-example}
\end{figure}
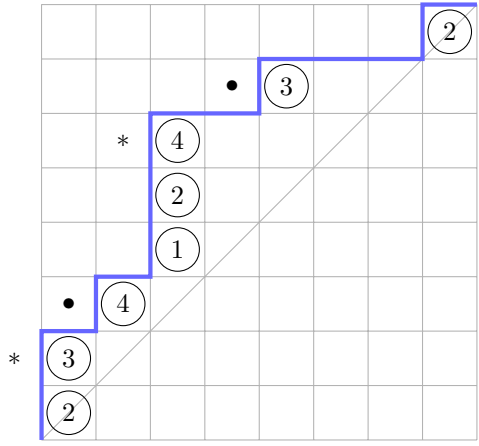

\begin{definition}[Monomial]
    \label{def:monomial}
    For a decorated labeled Dyck path $\pi$, we define the monomial $x^\pi = \prod_{i=1}^n x_{w_i}$, where $w_i$ is the label of the $i\th$ vertical step of $\pi$.
\end{definition}

\begin{definition}[Area]
    \label{def:area}
    Given a decorated labeled Dyck path $\pi$ of size $n$, its \emph{area word} is the word of nonnegative integers whose $i\th$ letter equals the number of unit squares between the $i\th$ vertical step of the path and the line $x=y$. If $a$ is the area word of $\pi$, the \emph{area} of $\pi$ is \[\area(\pi) \coloneqq \sum_{i \in [n] \setminus \mathsf{dr}(\pi)} a_i.\]
\end{definition}

For example, the decorated labeled Dyck path in \Cref{fig:path-example} has area $6$. Since we only consider the specialization $q=1$, we disregard the dinv statistic.

\begin{definition}[Touch number]
    The \emph{touch number} of a path $\pi \in \LD(n)^{\ast k, \bullet l}$ is the number of non-decorated vertical steps that touch the main diagonal. We denote it by $\touch(\pi)$, and we let $\LD(n,r)^{\ast k, \bullet l}$ be the subset of $\LD(n)^{\ast k, \bullet l}$ of paths with touch number $r$.
\end{definition}

Notice that $1 \leq \touch(\pi) \leq n-k-l$. 

For example, the path in \Cref{fig:path-example} has touch number $2$.

\begin{remark}
	\label{rmk:rises-falls-correspondence}
	Note that, if we call a horizontal step followed by another horizontal step a \emph{fall}, then there is a natural bijection between rises and falls.
    Indeed, the joining point of the two vertical steps of a rise is a point where a path $\pi$ vertically crosses a certain diagonal parallel to the main diagonal.
    Since the path must end at the main diagonal, it must cross the same diagonal horizontally at least once, via a fall. We map the rise to the first such fall; this yields a bijection (see \Cref{fig:rises-falls-correspondence}).
    Therefore, we might equivalently decorate falls instead of the corresponding rises.
\end{remark}

\begin{figure}[ht]
	\centering
	\begin{tikzpicture}[scale=0.6]
		\draw[gray!60, thin] (0,0) grid (11,11) (0,0) -- (11,11) (-1,1) -- (10,12);
		\draw[blue!60, ultra thick] (0,0) -- (0,2) -- (1,2) -- (1,3) -- (2,3) -- (2,6) -- (4,6) -- (4,9) -- (5,9) -- (5,10) -- (7,10) -- (7,11) -- (11,11);
		\draw[ultra thick] (2,3) -- (2,5) (8,11) -- (10,11);
		\fill[pattern=north west lines, pattern color=gray] (2,4) rectangle (4,5) (8,11) rectangle (9,9);
		\draw  (1.5,4.5) node {$\ast$};
	\end{tikzpicture} 
	\caption{Correspondence between rises and falls.}
	\label{fig:rises-falls-correspondence}
\end{figure}
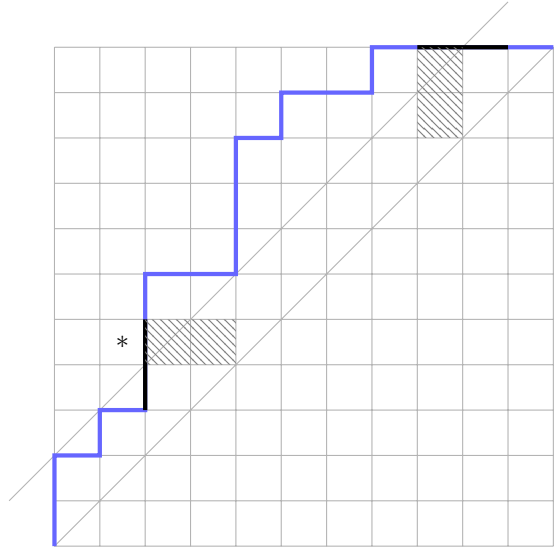

\subsection{Case \texorpdfstring{$k+l=n-1$}{k+l=n-1}}

In this subsection, we prove that the statement holds when $k+l=n-1$, that is, when the first vertical step is the only non-decorated one.

\begin{theorem}
    \label{thm:theta_theta_e1}
    For $k,l \in \mathbb{N}$, we have
    \[ \left.\Theta_{e_k} \Theta_{e_l} e_1 \right\rvert_{q=1} = \sum_{\pi \in \LD(k+l+1)^{\ast k, \bullet l}} t^{\area(\pi)} x^\pi. \]
\end{theorem}

Combining the commutation relations between $\Theta$ and $\D_\gamma$ established in \Cref{cor:theta_theta_e1} with the combinatorial interpretation of $\left.\D_\gamma\cdot 1\right\rvert_{q=1}$ given in \Cref{thm:D.1_comb}, we obtain a combinatorial interpretation for $\left.\Theta_{e_k} \Theta_{e_l} e_1 \right\rvert_{q=1}$.

\begin{proposition}
    For $k,l \in \mathbb{N}$, we have
    \[ 
        \left.\Theta_{e_k} \Theta_{e_l} e_1 \right\rvert_{q=1} 
        = \sum_{w\in W\left(0^l,1^k\right)} \left. \D_{\Des(0w)}\cdot 1\right\rvert_{q=1}=\sum_{w\in W\left(0^l,1^k\right)}\sum_{\tau\in \R(\Des(0w))}t^{\area(\tau)}e_{\mu(\tau)}
    \]
\end{proposition}

\begin{proof}
    The first equality follows by specializing \Cref{cor:theta_theta_e1} at $q=1$, while the second follows by applying \Cref{thm:D.1_comb} to each term of the sum.
\end{proof}

\begin{definition}
    Given a decorated Dyck path $\pi$, we define its \emph{contracted path $\cc(\pi)$} as the lattice path obtained by deleting every horizontal step that precedes a decorated valley of $\pi$. 
\end{definition}

\begin{example}
    The contracted path of $\pi$ in \Cref{fig:path-example} is the path $(3,6,7,7,7,8)$, expressed in the notation of \Cref{rmk:heights}.
\end{example}

\begin{remark}
    In the definition of $\LD(n)^{\ast k, \bullet l}$, the restrictions on the labels depend solely on the rise composition of the contracted paths. In particular, by grouping labeled paths according to their contracted paths and observing that the area statistic is independent of the labels, we may rewrite
    \[
        \sum_{\pi \in \LD(k+l+1)^{\ast k, \bullet l}} t^{\area(\pi)} x^\pi
        = \sum_{\pi \in \DP(k+l+1)^{\ast k, \bullet l}} t^{\area(\pi)} e_{\mu(\cc(\pi))}
    \]
    where $\DP(n)^{\ast k, \bullet l}$ is the set of decorated Dyck paths without labels and, in the unlabeled setting, every valley is considered contractible.
\end{remark}

From now on, we will refer to the elements of $\DP(k+l+1)^{\ast k, \bullet l}$ by indicating only their paths, with the understanding that every vertical step other than the first is decorated.

To prove \Cref{thm:theta_theta_e1}, it is enough to find a bijection $\Psi$ between $\DP(k+l+1)^{\ast k, \bullet l}$ and the set
\[ \{ (w, \tau) \mid w\in W \left(0^l,1^k\right), \tau \in \R(\Des(0w)) \} \]
that preserves both the area and the rise composition. In other words, for every $\pi \in \DP(k+l+1)^{\ast k, \bullet l}$ with $\Psi(\pi) = (w,\tau)$ for some $w \in W\left(0^l,1^k\right)$ and $\tau \in \R(\Des(0w))$, we want to have $\area(\pi)=\area(\tau)$ and $\mu(\cc(\pi))=\mu(\tau)$.

Notice that $\area(\tau)$ depends on $\Des(0w)$, and thus on $w$. We now define such a map $\Psi$ recursively via the following algorithm. The definition of $\Psi$ and the fact that it provides a bijection have been implemented and formally verified in Lean 4 \cite{Iraci2026Code}.

\begin{definition}
    \label{def:psi}
    Let $\pi \in \DP(k+l+1)^{\ast k, \bullet l}$. We construct a word $w_0 = 0w$, with $w \in W(0^l, 1^k)$, and a path $\tau \in \R(\Des(0w))$.
    
    We start with $(\varepsilon, \varnothing)$, where $\varepsilon$ denotes the empty word and $\varnothing$ is the empty path. We then apply the following algorithm.

    \begin{enumerate}
      \setcounter{enumi}{-1}
        \item Append $0$ to $w_0$. \label{step:0}

        \item Let $a$ be the maximum integer such that $\pi = (UR)^a \pi_0$, for some nonempty Dyck path $\pi_0$. Append $0^a$ to $w_0$. \label{step:valleys}

        \item Let $b$ be the maximum integer such that $\pi_0 = U \pi_1 \,U^b R^{b+1} \pi_2$, for some lattice paths $\pi_1$ and $\pi_2$ such that $\pi_1$ contains no two consecutive horizontal steps. Append $1^b$ to $w_0$. \label{step:rises}

        \item Let $h$ be the height (that is, the number of vertical steps) of $\pi_1$. Append $(\ell(0w)+h)$ to $\tau$ and repeat the process with $\pi' = \pi_1 \pi_2$. \label{step:induction}
    \end{enumerate}

    The algorithm terminates when $\pi_1 \pi_2$ is empty, and we define $\Psi(\pi) = (w,\tau)$, where $w_0 = 0w$.
\end{definition}

The path $\pi$ is transformed as follows. In Step~\ref{step:valleys} we remove $a$ valleys from the beginning of $\pi$; in Step~\ref{step:rises} we remove the first $b$ falls of $\pi$; in Step~\ref{step:induction} we remove the first $U$ step and the $R$ step that is a fall (overall, we remove a valley, cf.\ \Cref{lem:well-defined}).

At the end of each iteration, we append a nondecreasing string to $w$ and determine the height of a column of $\tau$. Inductively, we decompose $0w$ as $0 0^a 1^b w'$, where $w'$ is either empty or starts with a $0$. We now prove some properties of $\Psi$.

\begin{lemma}
    \label{lem:well-defined}
    If $\pi \in \DP(k+l+1)^{\ast k, \bullet l}$ and $\Psi(\pi) = (w,\tau)$, then $w \in W(0^l,1^k)$ and $\tau \in \R(\Des(0w))$.
\end{lemma}

\begin{proof}
    To prove this statement, it is convenient to change perspective and work with falls and valleys instead, as in \Cref{rmk:rises-falls-correspondence}. We also identify the valleys of $\pi$ with the horizontal steps that precede them, and compute the area by counting the number of unit squares below the valleys.

    In Step~\ref{step:valleys} of \Cref{def:psi}, we remove exactly $a$ valleys from $\pi$ and append $0^a$ to $w$. In Step~\ref{step:rises}, we remove exactly $b$ decorated falls and append $1^b$. Thus, in both cases the numbers agree.

    Finally, we need to show that, in Step~\ref{step:induction} of \Cref{def:psi},
    we either complete the process by going from $UR$ to the empty path, or we remove exactly one valley: we can equivalently think of its contribution to $w$ as either the starting $0$ in $0w$ (which does not correspond to a valley) or as the starting $0$ of $w'$, if it is nonempty.
    
    If the $R$ step between $\pi_1$ and $\pi_2$ is a valley in $\pi_0$, then we remove exactly one valley. If it is not a valley, then it is a fall and is therefore followed by a horizontal step. By construction, $\pi_1$ must then end with a horizontal step (otherwise $b$ is not maximal), and that step is a valley that becomes a fall in $U \pi_1 R \pi_2$ after removing $U^b R^b$, so we ultimately remove exactly one valley.

    If $\pi_1 \pi_2$ is empty, the claim follows. Otherwise, by induction, $\Psi(\pi_1 \pi_2) = (w',\tau')$ for some $w' \in W(0^{l-a},1^{k-b})$ and $\tau' \in \R(\Des(0w'))$. Since $w = 0^a 1^b 0 w'$, we have $w \in W(0^l,1^k)$. It remains to verify that $\tau_1 \leq \tau'_1 + a + b + 1$, which is true because $\tau_1 = a + b + h + 1$ and $h \leq \tau'_1$ by construction. Therefore, $\tau \in \R(\Des(0w))$, as desired.
\end{proof}

\begin{lemma}
    \label{lem:area}
    If $\pi \in \DP(k+l+1)^{\ast k, \bullet l}$ and $\Psi(\pi) = (w,\tau)$, then $\area(\pi) = \area(\tau)$.
\end{lemma}

\begin{proof}
    In Step~\ref{step:0}, the path does not change; in Step~\ref{step:valleys} of \Cref{def:psi}, the area does not change because we remove only steps on the main diagonal; in Step~\ref{step:rises}, we remove only decorated falls, and these do not contribute to the area.

    Let $v$ be the number of valleys in $\pi_1$. Each of these valleys now contributes one less to the area, since we removed the first vertical step. The valley that is removed (or converted into a rise) in Step~\ref{step:induction} is at height $h$, and it is preceded by $v$ horizontal steps, so it contributes $h-v$ to the area. Altogether, the area decreases by $h$.

    Since there are $h$ cells in the first column between $\tau$ and $\delta(\Des(0w))$, the claim follows by induction.
\end{proof}

\begin{lemma}
    \label{lem:comp}
    If $\pi \in \DP(k+l+1)^{\ast k, \bullet l}$ and $\Psi(\pi) = (w,\tau)$, then $\mu(\cc(\pi)) = \mu(\tau)$.
\end{lemma}

\begin{proof}
    By construction, $\mu(\cc(\pi))_1 = \mu(\tau)_1 = a+b+h+1$, and inductively $\mu(\cc(\pi')) = \mu(\tau')$.
    
    We have $\pi = (UR)^a U \pi_1 U^b R^{b+1} \pi_2$. By construction, $\pi_1$ has no two consecutive $R$ steps, so $\cc(\pi) = U^a U U^h U^b R^b \cc(R \pi_2)$; in particular, note that
    \[ \mu(\cc(\pi)) = (a+b+h+1, \mu(\cc(\pi_2))). \]    
    We distinguish two cases.

    \textbf{Case 1: $\pi_2$ starts with an $R$ step.}    
    If $\pi_2$ starts with an $R$ step, then $\pi_1$ must end with an $R$ step (by maximality of $b$), so $\pi_1$ and $\pi_2$ are separated by at least two $R$ steps, and $\mu(\cc(\pi')) = (h, \mu(\cc(\pi_2)))$.
    By construction, $\tau'_1 = \mu(\cc(\pi'))_1 = h$, and since $\tau'_1 = \tau_2 - (a+b+1)$, we have $\tau_1 = \tau_2$ and so
    \[ \mu(\tau) = (a+b+h+1, \mu(\tau')_2, \dots). \]

    \textbf{Case 2: $\pi_2$ starts with a $U$ step.}
    If $\pi_2$ starts with a $U$ step instead, then $\pi_1$ and $\pi_2$ are separated by at most one $R$ step in $\pi'$, so their contractions join and we have
    \[ \mu(\cc(\pi')) = (h+\mu(\cc(\pi_2))_1, \mu(\cc(\pi_2))_2, \dots). \]
    By construction, $\tau'_1 = \mu(\cc(\pi'))_1 = h + \mu(\cc(\pi_2))_1$, so $\mu(\cc(\pi_2))_1=\tau'_1-h$. This also implies that $\tau_1 < \tau_2$, and hence
    \[ \mu(\tau) = (a+b+h+1, \mu(\tau')_1-h, \mu(\tau')_2, \dots). \]

    In either case, by induction we have $\mu(\cc(\pi')) = \mu(\tau')$, and so $\mu(\cc(\pi)) = \mu(\tau)$, as desired.
\end{proof}

\begin{figure}[ht]
	\centering
	\begin{tikzpicture}[scale=0.6]
		\draw[gray!60, thin] (0,0) grid (10,10) (0,0) -- (10,10);
        \draw[valleys, ultra thick, -sharp >, sharp angle = 45] (0,0) -- (0,1) -- (1,1) -- (1,2) -- (2,2);
		\draw[ud, ultra thick, sharp <-, sharp angle = 45] (2,2) -- (2,3);        
        \draw[rho, ultra thick] (2,3) -- (2,4) -- (3,4) -- (3,5);
        \draw[rises, ultra thick] (3,5) -- (3,8) -- (6,8);
        \draw[ud, ultra thick, -sharp >, sharp angle = 45] (6,8) -- (7,8);
        \draw[rhoprime, ultra thick, sharp <-, sharp angle = 45] (7,8) -- (7,9) -- (8,9) -- (8,10) -- (10,10);
    \end{tikzpicture}
	\caption{The decomposition of $\pi = \textcolor{valleys}{{(UR)}^2} U \textcolor{rho}{{URU}} \textcolor{rises}{U^3 R^3} R \textcolor{rhoprime}{U R U R R}$ into the subpaths prescribed by \Cref{def:psi}.}
    \label{fig:pp0}
\end{figure}

\begin{example}
    \label{ex:psi}
    Let us compute $\Psi$ for the path $\pi=(1,2,4,7,7,7,8,10,10,10)$, shown in \Cref{fig:pp0}. Notice that $\pi$ has $4$ rises, $5$ valleys, $\area(\pi) = 3$, and $\mu(\cc(\pi)) = (8,2)$. We will compute $\Psi(\pi) = (w,\tau)$, where $w \in W(0^5, 1^4)$ and $\tau \in \R(\Des(0w))$.

    We start with the empty word $\varepsilon$ and the empty path $\tau=()$. After Step~\ref{step:0}, we have $0w = 0$.

    We then write $\pi=(UR)^2\pi_0$ where $\pi_0$ is the path $(2,6,6,6,6,7,8,8)$, so $a=2$. After Step~\ref{step:valleys}, we obtain $0w = 000$.

    Then, we write $\pi_0 = U \pi_1 U^3 R^4 \pi_2$ where $\pi_1 = URU$ and $\pi_2 = URURR$ are the subpaths in orange and violet, respectively, in \Cref{fig:pp0}. We have $b=3$, so we update $0w$ to $0w = 000111$.

    Next, in Step~\ref{step:induction}, we compute the height of $\pi_1$, which is $h=2$, and we append $6+2=8$ to $\tau$. We restart the iteration with the path $\pi' = \pi_1\pi_2 = URUURURR$.

    Its decomposition is $\pi' = (UR)^1 U (UR) (U^1 R^1) R$, so we have $a=1$, $b=1$, and $h = 1$, and we update $0w$ to $0w = 000111001$ and $\tau$ to $\tau = (8, 10)$. We have $\pi'_1 = UR$ and $\pi'_2 = \varnothing$.

    Iterating one more time, we get $\pi'' = \pi'_1\pi'_2 = UR$, so $a=0$, $b=0$, and $h=0$. We update $0w$ to $0w = 0001110010$ and $\tau$ to $\tau = (8, 10, 10)$. At this point, $\pi''_1\pi''_2$ is empty, so we stop and obtain $\Psi(\pi) = (w,\tau)$ with $w = 001110010$ and $\tau=(8,10,10)$.

    Indeed, $\Psi(\pi) = (w,\tau)$ satisfies the required properties: $w \in W(0^5,1^4)$, $\Des(0w) = (6, 9, 10)$ and $\tau \in \R(\Des(0w))$. The area of $\tau$ is $(8-6) + (10-9) = 2 + 1 = 3$, which equals $\area(\pi)$, and the rise composition of $\tau$ is $(8,2)$, which equals the rise composition of $\cc(\pi)$.
\end{example}

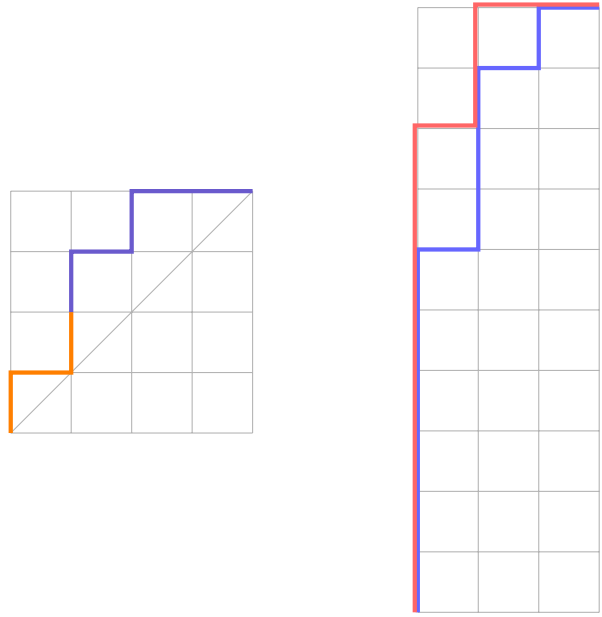
\begin{figure}
    \centering
    \begin{tikzpicture}[scale=0.8]
        \draw[draw=none, use as bounding box] (0,-3) rectangle (4,7);
		\draw[gray!60, thin] (0,0) grid (4,4) (0,0) -- (4,4);
        \draw[rho, ultra thick] (0,0) -- (0,1) -- (1,1) -- (1,2);
        \draw[rhoprime, ultra thick] (1,2) -- (1,3) -- (2,3) -- (2,4) -- (4,4);
    \end{tikzpicture}\hspace{2 cm}
    \begin{tikzpicture}[scale=0.8]
        \draw[step=1.0, gray!60, thin] (0,0) grid (3,10);
        \draw[blue!60, line width = 1.6 pt] (0,0) -- (0,6) -- (1,6) -- (1,9) -- (2,9) -- (2,10) -- (3,10);
        \begin{scope}[shift={(-0.05,0.05)}]
            \draw[red!60, line width = 1.6 pt] (0,-0.05) -- (0,8) -- (1,8) -- (1,10) -- (3.05,10);
        \end{scope}
    \end{tikzpicture}
    \caption{The path $\pi'$ obtained from $\pi$ after one iteration of the algorithm in \Cref{def:psi} (left), and the path $\tau$ in red above $\delta(\Des(0w))$ in blue (right).}
\end{figure}

We can finally prove the main statement of this section.

\begin{theorem}
    The map $\Psi$ is a bijection between $\DP(k+l+1)^{\ast k, \bullet l}$ and the set
    \[ \{ (w, \tau) \mid w \in W \left(0^l,1^k\right), \tau \in \R(\Des(0w)) \} \]
    that preserves both the area and the rise composition.
\end{theorem}

\begin{proof}
    We prove that $\Psi$ is bijective by exhibiting an inverse, using induction on the number of descents of $w$. Let $w \in W(0^l, 1^k)$ and $\tau \in \R(\Des(0w))$.
    
    If $w$ has no descents, then $w = 0^l 1^k$ for some $k,l \in \mathbb{N}$, and $\tau = (k+l+1)$. In this case, we can recover $\pi$ as $\pi = (UR)^l U^{k+1} R^{k+1}$, which has $l$ valleys, $k$ rises, area $0$, and rise composition $(k+l+1)$, as desired.

    Otherwise, let $w = 0^a 1^b 0 w'$ for some $a,b \in \mathbb{N}$ and some word $w' \in W \left(0^{l-a},1^{k-b}\right)$, and let $\tau' \in \R(\Des(0w'))$ be such that $\tau - \delta(\Des(0w)) = (h, \tau' - \delta(\Des(0w')))$ for some $h \in \mathbb{N}$. 

    These vectors encode the area of the paths, namely the number of area cells in each column, so $\area(\tau) = \area(\tau') + h$.

    By induction, we can find a unique path $\pi' \in \DP(k-b+l-a+1)^{\ast k-b, \bullet l-a}$ such that $\Psi(\pi') = (w',\tau')$, with $\area(\pi') = \area(\tau')$ and $\mu(\cc(\pi')) = \mu(\tau')$.

    By construction, since $\mu(\cc(\pi')) = \mu(\tau')$, $\pi'$ has no two consecutive horizontal steps in the first $\mu(\tau')_1 \geq h$ rows. Consider the vertical step in the $h\th$ row of $\pi'$, and call it $U_h$. If it is followed by a vertical step, let $\pi_1$ be the path ending with $U_h$, and $\pi_2$ be the path starting after $U_h$; if it is followed by a horizontal step, call it $R_h$, and let $\pi_1$ be the path ending with $R_h$, and $\pi_2$ be the path starting after $R_h$. Split $\pi' = \pi_1 \pi_2$. By construction, $\pi_1$ has height $h$ and no two consecutive horizontal steps.

    Let $\pi = (UR)^a U \pi_1 U^b R^{b+1} \pi_2$. By construction, $\Psi(\pi) = (w, \tau)$: the only nontrivial check is that $b$ is maximal, but either $\pi_1$ ends with an $R$ step or $\pi_2$ starts with a $U$ step, so $b$ is indeed maximal.

    The result now follows by \Cref{lem:well-defined,lem:area,lem:comp}.
\end{proof}

\subsection{Case with touch number \texorpdfstring{$n-k-l$}{n-k-l}}

In this subsection, we give a combinatorial interpretation of a symmetric-function identity that allows us to prove the general case from the case $k+l=n-1$.

First, we need to prove that specialization at $q=1$ commutes with the Theta operators. This is not immediate, since the Theta operators include a plethystic evaluation $f[X/M]$, which is not defined at $q=1$.

\begin{lemma}
    \label{lem:theta_commutes_with_q1_eval}
    If $F \in \Lambda$ is regular at $q=1$ (i.e.\ its coefficients when expanded in the Schur basis have no poles at $q=1$), then $\Theta_{e_k} F$ is also regular at $q=1$, and we have
    \[ \left. \left( \Theta_{e_k} F \right) \right\rvert_{q=1} = \left. \left( \Theta_{e_k} \left( \left. F \right\rvert_{q=1} \right) \right) \right\rvert_{q=1}. \]
\end{lemma}

\begin{proof}
    Recall that the modified Macdonald polynomials are Schur-positive \cite{Haiman2001nFactorial}, so the transition matrix from the Macdonald basis to the Schur basis has entries in $\mathbb{N}[q,t]$ (and thus no poles at $q=1$).
    By Cramer's rule, the inverse transition matrix (from the Schur basis to the Macdonald basis) is also regular at $q=1$. Therefore, if $F$ is regular at $q=1$, its expansion $F = \sum_\nu c_\nu(q,t) \Ht_\nu$ in the Macdonald basis has coefficients $c_\nu(q,t)$ that are regular at $q=1$.

    It suffices to prove regularity for $F = \Ht_\nu$, since the $\Ht_\nu$ form a basis of $\Lambda$. By \Cref{def:theta}, we have
    \begin{align*}
        \Theta_{e_k} \Ht_\nu & = \mathbf{\Pi} e_k \left[\frac{X}{M}\right] \mathbf{\Pi}^{-1} \Ht_\nu \\
        & = \mathbf{\Pi} e_k \left[\frac{X}{M}\right] \Pi_{\nu}^{-1} \Ht_\nu \\
        & = \mathbf{\Pi} \sum_{\mu} d^{(k)}_{\mu,\nu}(q,t) \Pi_{\nu}^{-1}  \Ht_\mu \\
        & = \sum_{\mu} \frac{\Pi_{\mu}}{\Pi_{\nu}} d^{(k)}_{\mu,\nu}(q,t) \Ht_\mu,
    \end{align*}
    where the sum is over partitions $\mu \supset \nu$ with $|\mu| - |\nu| = k$, and $d^{(k)}_{\mu,\nu}(q,t)$ is the Macdonald Pieri coefficient for $e_k^\ast \Ht_\nu = \sum_\mu d^{(k)}_{\mu,\nu} \Ht_\mu$.
    
    We refer to \cite{GarsiaHaglundXinZabrocki2016PieriRules} for notation. Explicitly, by \cite[Equation~(3.23)]{GarsiaHaglundXinZabrocki2016PieriRules} we have
    \[ d^{(k)}_{\mu,\nu}(q,t) = \frac{w_\nu}{w_\mu} c^{(k)}_{\mu,\nu}(q,t), \]
    where $w_\nu$ is the Macdonald weight of $\nu$, and by \cite[Theorem~3.2]{GarsiaHaglundXinZabrocki2016PieriRules},
    \[ c^{(k)}_{\mu,\nu}(q,t) = \frac{1}{B_{\mu/\nu}} \sum_{\nu \subset_1 \alpha \subset_{k-1} \mu} c^{(k-1)}_{\mu, \alpha}(q,t) c^{(1)}_{\alpha,\nu}(q,t) \frac{T_\alpha}{T_\nu}, \]
    where $B_{\mu/\nu} = B_\mu - B_\nu$.
    
    Combining these expressions, up to the global factor $B_{\mu/\nu}^{-1}$ (which is regular at $q=1$ since $B_{\mu/\nu}(1,t) \neq 0$), the coefficient $\frac{\Pi_\mu}{\Pi_\nu} d^{(k)}_{\mu,\nu}(q,t)$ of $\Ht_\mu$ in $\Theta_{e_k} \Ht_\nu$ decomposes into a sum of products over chains $\nu \subset_1 \dots \subset_1 \mu$ of terms of the form
    \[ \frac{\Pi_\beta T_\alpha w_\alpha}{\Pi_\alpha T_\beta w_\beta} c^{(1)}_{\beta, \alpha}(q,t) \qquad (\text{for } \alpha \subset_1 \beta). \]
    By the proof of \cite[Lemma~5.3]{DAdderioIraciLeBorgneRomeroVandenWyngaerd2022TieredTrees}, each of these terms is regular at $t=1$; applying partition conjugation $\mu \mapsto \mu'$ (which interchanges $q$ and $t$ in Macdonald polynomials) yields regularity at $q=1$.

    Finally, expanding any regular $F = \sum_\nu c_\nu(q,t) \Ht_\nu$ yields 
    \[ \left( \Theta_{e_k} F \right)_{q=1} = \sum_\nu c_\nu(1,t) \left( \Theta_{e_k} \Ht_\nu \right)_{q=1} = \left( \Theta_{e_k} \left( \left. F \right\rvert_{q=1} \right) \right)_{q=1}, \]
    as the Pieri coefficients for the modified Hall--Littlewood polynomials $\Ht_\mu[X;1,t]$ are the $q=1$ specializations of the corresponding Macdonald Pieri coefficients. This completes the proof.
\end{proof}

Next, we show that, when $q=1$, the operators involved are multiplicative. 

\begin{theorem}
    \label{thm:theta_q1_multiplicative}
    For $n,k,l \in \mathbb{N}$ with $k+l < n$, we have \[ \left. \Theta_{e_k} \Theta_{e_l} \nabla E_{n-k-l, n-k-l} \right\rvert_{q=1} = \left. \sum_{\sum k_i = k} \sum_{\sum l_i = l} \prod_{i=1}^{n-k-l} \Theta_{e_{k_i}} \Theta_{e_{l_i}} e_1 \right\rvert_{q=1}, \]
    where the $k_i$ and $l_i$ are nonnegative integers for $1 \leq i \leq n-k-l$.  
\end{theorem}

\begin{proof}
    By \Cref{lem:theta_commutes_with_q1_eval} and \Cref{rem:Enk},
    \[ \left. \Theta_{e_k} \Theta_{e_l} \nabla E_{n-k-l, n-k-l} \right\rvert_{q=1} = \left. \Theta_{e_k} \Theta_{e_l} e_1^{n-k-l} \right\rvert_{q=1}. \]
    When $q=1$, the $\D_a$ operators are multiplicative by \Cref{lem:Dgamma_q1} and $\D_1$ coincides with $e_1^\bullet$, so
    \[ \left. \Theta_{e_k} \Theta_{e_l} e_1^{n-k-l} \right\rvert_{q=1} = \left. \Theta_{e_k} \Theta_{e_l} \D_1^{n-k-l} \cdot 1\right\rvert_{q=1}. \]
    Introduce the generating series \[ \mathsf{\D}(u) = \sum_{a \geq 1} u^a \D_a. \] Passing to generating series and using the proofs of \Cref{thm:theta_e1_D1,thm:shuffle_theta_D}, we obtain
    \[ \Theta(v) \Theta(u) u \D_1 = \Theta(v) \mathsf{\D}(u) \Theta(u) = \sum_{k=0}^\infty \sum_{l=0}^\infty \sum_{w \in W(0^l, 1^k)} u^{l+1} v^k \D_{\Des(0w)} \Theta(v) \Theta(u), \]
    so, when $q=1$, multiplicativity gives
    \[ \left. \Theta(v) \Theta(u) u^{n-k-l} \D_1^{n-k-l} \cdot 1 \right\rvert_{q=1} = \left( \sum_{k'=0}^\infty \sum_{l'=0}^\infty \sum_{w \in W(0^{l'}, 1^{k'})} u^{l'+1} v^{k'} \D_{\Des(0w)} \cdot 1 \right)_{q=1}^{n-k-l}. \]
    The result follows from \Cref{cor:theta_theta_e1} by taking the coefficient of $u^{n-k} v^k$.
\end{proof}

In particular, we can conclude the following.

\begin{theorem}
    \label{thm:theta_theta_Enn}
    For $n,k,l \in \mathbb{N}$ with $n > k+l$, we have
    \[ \left. \Theta_{e_k} \Theta_{e_l} \nabla E_{n-k-l, n-k-l} \right\rvert_{q=1} = \sum_{\pi \in \LD(n, n-k-l)^{\ast k, \bullet l}} t^{\area(\pi)} x^\pi. \]
\end{theorem}

\begin{proof}
    By \Cref{thm:theta_q1_multiplicative}, we can rewrite the left-hand side as a sum of products of expressions of the form
    \[ \left.\Theta_{e_{k'}} \Theta_{e_{l'}} e_1 \right\rvert_{q=1} \]
    for some $k', l'$. By \Cref{thm:theta_theta_e1}, these can be interpreted as the generating series
    \[ \sum_{\pi \in \LD(k'+l'+1)^{\ast k', \bullet l'}} t^{\area(\pi)} x^\pi. \]
    A path with touch number exactly $n-k-l$ is the concatenation of $n-k-l$ paths in which only the first (vertical) step is not decorated. Since we are disregarding the dinv and the area of the larger path is the sum of the areas of the $n-k-l$ smaller paths, the result follows by multiplying the generating series.
\end{proof}

\subsection{Proof of the theorem}

The following symmetric function identity is \cite[Theorem~8.2]{DAdderioRomero2023ThetaIdentities}, rewritten in the notation of \cite{IraciNadeauVandenWyngaerd2024Smirnov}.

\begin{theorem}[{\cite[Theorem~8.2]{DAdderioRomero2023ThetaIdentities}}]
    \label{thm:full_recursion}
    For $n,k,l,j \in \mathbb{N}$ with $n > k+l$, we have
    \begin{align*}
        h_j^\perp & \Theta_{e_k} \Theta_{e_l} \nabla E_{n-k-l, n-k-l} = \sum_{r=0}^{j} \sum_{a=0}^{j} \sum_{b=0}^{j} \sum_{i=0}^{j} \qbinom{n-k-l-(j-r-a+b)-1}{i} \qbinom{n-k-l}{j-r-a+i+b} \\
        & \times q^{\binom{a-i-b}{2}} \qbinom{n-k-l-(j-r-a+i+b)}{a-i-b}  q^{\binom{r-i-b}{2}} \qbinom{n-k-l-(j-r-a+i+b)}{r-i-b} \\
        & \times \Theta_{e_{k-r}} \Theta_{e_{l-a}} \nabla E_{n-k-l-j+a+r,n-k-l-j+a+r-b}
    \end{align*}
\end{theorem}

We now prove the following result.

\begin{theorem}
    \label{thm:combinatorial_recursion}
    For $n,k,l,u \in \mathbb{N}$ with $n > k+l+u$, we have
    \[ \left.\Theta_{e_k} \Theta_{e_l} \nabla E_{n-k-l, n-k-l-u} \right\rvert_{q=1} = \sum_{\pi \in \LD(n, n-k-l-u)^{\ast k, \bullet l}} t^{\area(\pi)} x^\pi. \]  
\end{theorem}

For a similar but simpler argument, see \cite[Theorem~4.8]{IraciNadeauVandenWyngaerd2024Smirnov}.

\begin{proof}
    We proceed by strong induction on $u$. The base case $u=0$ is exactly the statement of 
    \Cref{thm:theta_theta_Enn}. If $u > 0$, set $j=u$ in \Cref{thm:full_recursion}. There is, among the terms whose summation index $b$ equals the induction parameter $u$, only one term for which $E_{m,m-b}$ (for some $m$) does not vanish, namely the term with $a=r=j$ and $i=0$ (so $m=n+j$). In every other case, at least one binomial coefficient vanishes.
    
    It follows that $\Theta_{e_{k-u}} \Theta_{e_{l-u}} \nabla E_{n+u,n}$, which represents the general case after a suitable choice of $n,k,l \in \mathbb{N}$, can be written in terms of expressions of the form $\Theta_{e_{k'}} \Theta_{e_{l'}} \nabla E_{n',n'-b'}$ for appropriate $n',k',l' \in \mathbb{N}$ and $b'<u$. Each such expression has a combinatorial interpretation by the induction hypothesis. 

    It remains to check that \Cref{thm:full_recursion} admits a combinatorial interpretation that is compatible with the statement of the theorem. We can interpret the operator $h_j^\perp$ as removing exactly the $j$ occurrences of the maximal label, say $M$, in a path with touch number $n-k-l$.
    
    We interpret the indices $b, i, r, a$, and the value $j-r-a+i-b$, as follows (see \Cref{fig:recursion_indices}):
    \begin{itemize}
        \item $b$ is the number of vertical steps labeled $M$ that are either decorated valleys off the main diagonal, or decorated rises labeled $M$ such that the vertical step in the next row is off the main diagonal;
        \item $i$ is the number of decorated rises labeled $M$ such that the vertical step in the next row is a decorated valley on the main diagonal;
        \item $r$ is the number of decorated rises labeled $M$, and in particular, $r-i-b$ is the number of such rises such that the vertical step in the next row is on the main diagonal, but it is not a decorated valley;
        \item $a$ is the number of decorated valleys labeled $M$, and in particular, $a-i-b$ is the number of such valleys on the main diagonal that are not preceded by a decorated rise labeled $M$;
        \item $j-r-a+i+b$, given the other interpretations, must be the number of non-decorated steps labeled $M$, which must all be on the main diagonal since they are not decorated.
    \end{itemize}

    In each case, the combinatorial interpretation of the identity is obtained by removing the vertical step labeled $M$, the following horizontal step, and any associated decoration. For $i$ and the second contribution to $b$, we also remove the decoration in the row immediately above $M$; the first contribution to $b$ requires additional care.

    Suppose that a maximal label $M$ is assigned to a decorated valley off the main diagonal. Since $M$ is maximal, it must be a peak. It must also be followed by another horizontal step: otherwise, the next vertical step would be a decorated, hence contractible, valley (because every step off the main diagonal is decorated), and would therefore require a label larger than $M$, a contradiction. Thus, $M$ is followed by a double fall and is paired with a double rise. Since all double rises are decorated, we remove that decoration as well.
 
    To understand the combinatorial interpretation, we start from a path in
    \[ \LD(n-j, n-k-l-j+r+a-b)^{\ast k-r, \bullet l-a} \]
    and try to obtain a path in $\LD(n, n-k-l)^{\ast k, \bullet l}$ with exactly $j$ maximal labels.
    
    First, for each of the $b$ non-decorated steps off the main diagonal, we decorate it appropriately: if it is a rise, we also insert a decorated valley labeled $M$ into the corresponding double fall; if it is a valley, we also insert a decorated rise labeled $M$ in the previous row. This can be done in a unique way. After this operation, all the non-decorated steps are on the main diagonal, the size increases by $b$, as do the numbers of decorated rises and valleys.
    
    Then, among the $n-k-l-j+r+a-b$ non-decorated steps on the main diagonal, we choose $i$ of them (excluding the first). For each chosen step, we insert a decorated rise labeled $M$ in the previous row and decorate the chosen step (they must now be contractible valleys, as we inserted a new horizontal step right before them). The number of possible choices is \[ \binom{n-k-l-(j-r-a+b)-1}{i}. \]
    Next, among the $n-k-l-j+r+a-b-i$ remaining non-decorated steps on the main diagonal, excluding the first step but allowing the endpoint of the path, we choose $r-i-b$ of them and insert decorated rises labeled $M$ in the row right before them. The number of ways in which this can be done is \[ \binom{n-k-l-(j-r-a+i+b)}{r-i-b}. \]
    Afterward, among the same $n-k-l-j+r+a-b-i$ remaining non-decorated steps on the main diagonal, again excluding the first step but allowing the endpoint of the path, we choose $a-i-b$ of them and insert decorated valleys labeled $M$ right before them. These valleys are necessarily contractible since their label is maximal. The number of possible choices is \[ \binom{n-k-l-(j-r-a+i+b)}{a-i-b}. \]
    Finally, again among the $n-k-l-j+r+a-b-i$ remaining non-decorated steps on the main diagonal, we choose $j-r-a+i+b$ spots, with repetition allowed and possibly including the endpoint of the path, and insert non-decorated steps labeled $M$ right before them. The number of ways in which this can be done is \[ \binom{n-k-l}{j-r-a+i+b}. \]
    We end up with a path in $\LD(n,n-k-l)^{\ast k, \bullet l}$ with exactly $j$ maximal labels, as expected. Since $q=1$, the binomials coincide with the $q$-analogs, matching the symmetric-function identity, which we can now interpret combinatorially.
\end{proof}

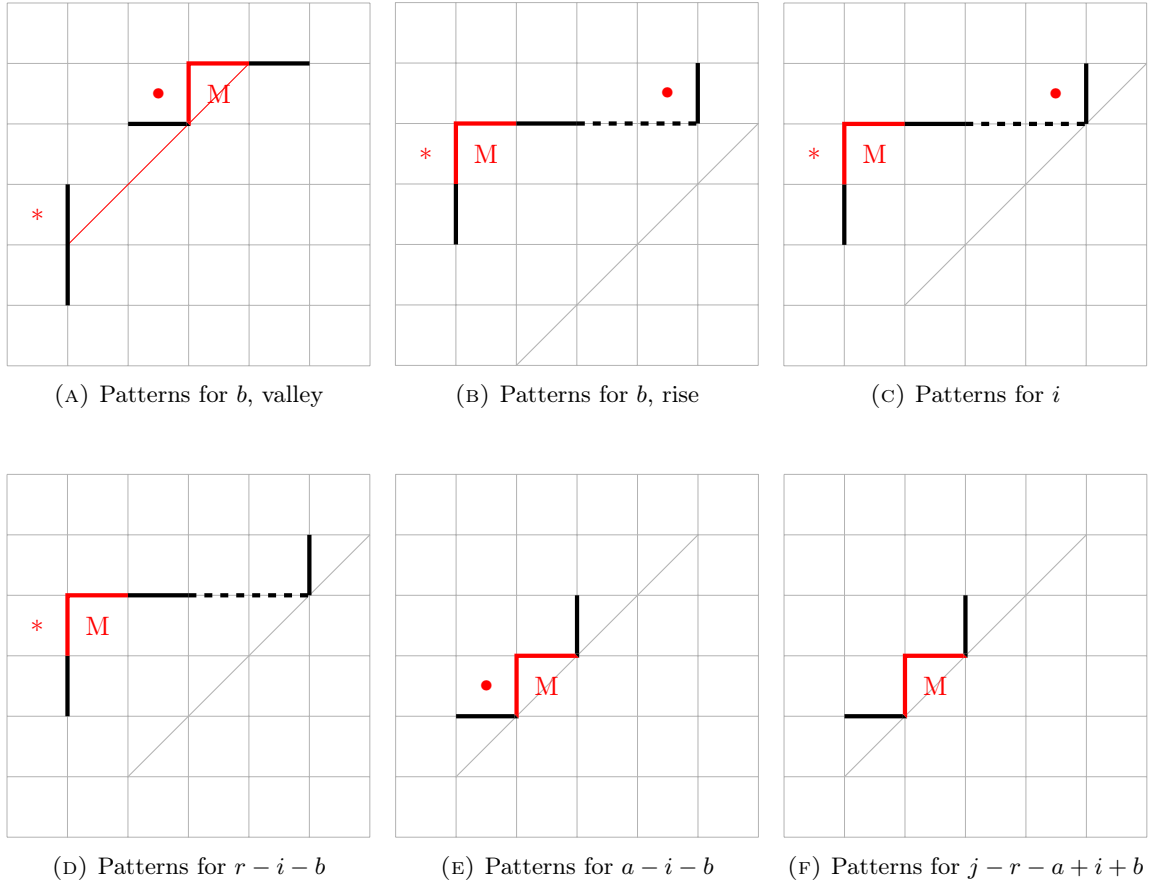
\begin{figure}
    \centering
    \begin{subfigure}{0.32\textwidth}
        \centering
        \begin{tikzpicture}[scale=0.8]
            \draw[step=1.0, gray!60, ultra thin] (0,0) grid (6,6);
            \draw[red] (1,2) -- (4,5);
            \draw[ultra thick] (1,1) -- (1,3);
            \draw[ultra thick, -sharp >, sharp angle = 45] (2,4) -- (3,4);
            \draw[ultra thick, red, sharp <-, sharp angle = 45] (3,4) -- (3,5) -- (4,5);
            \draw[ultra thick] (4,5) -- (5,5);

            \node[red] at (0.5, 2.5) {$\ast$};
            \node[red] at (2.5, 4.5) {$\bullet$};
            \node[red] at (3.5, 4.5) {M};
        \end{tikzpicture}
        \caption{Patterns for $b$, valley}
        \label{subfig:decorate-rise-add-valley}
    \end{subfigure}
    \hfill
    \begin{subfigure}{0.32\textwidth}
        \centering
        \begin{tikzpicture}[scale=0.8]
            \draw[step=1.0, gray!60, ultra thin] (0,0) grid (6,6);
            \draw[gray!60] (2,0) -- (6,4);
            \draw[ultra thick] (1,2) -- (1,3);
            \draw[ultra thick, red] (1,3) -- (1,4) -- (2,4);
            \draw[ultra thick] (2,4) -- (3,4);
            \draw[ultra thick, dashed] (3,4) -- (5,4);
            \draw[ultra thick] (5,4) -- (5,5);

            \node[red] at (0.5, 3.5) {$\ast$};
            \node[red] at (4.5, 4.5) {$\bullet$};
            \node[red] at (1.5, 3.5) {M};
        \end{tikzpicture}
        \caption{Patterns for $b$, rise}
        \label{subfig:decorate-valley-add-rise}
    \end{subfigure}
    \hfill
    \begin{subfigure}{0.32\textwidth}
        \centering
        \begin{tikzpicture}[scale=0.8]
            \draw[step=1.0, gray!60, ultra thin] (0,0) grid (6,6);
            \draw[gray!60] (2,1) -- (6,5);
            \draw[ultra thick] (1,2) -- (1,3);
            \draw[ultra thick, red] (1,3) -- (1,4) -- (2,4);
            \draw[ultra thick] (2,4) -- (3,4);
            \draw[ultra thick, dashed] (3,4) -- (5,4);
            \draw[ultra thick] (5,4) -- (5,5);

            \node[red] at (0.5, 3.5) {$\ast$};
            \node[red] at (4.5, 4.5) {$\bullet$};
            \node[red] at (1.5, 3.5) {M};
        \end{tikzpicture}
        \caption{Patterns for $i$}
        \label{subfig:decorate-valley-on-diagonal-add-rise}
    \end{subfigure}

    \vspace{0.8cm}

    \begin{subfigure}{0.32\textwidth}
        \centering
        \begin{tikzpicture}[scale=0.8]
            \draw[step=1.0, gray!60, ultra thin] (0,0) grid (6,6);
            \draw[gray!60] (2,1) -- (6,5);
            \draw[ultra thick] (1,2) -- (1,3);
            \draw[ultra thick, red] (1,3) -- (1,4) -- (2,4);
            \draw[ultra thick] (2,4) -- (3,4);
            \draw[ultra thick, dashed] (3,4) -- (5,4);
            \draw[ultra thick] (5,4) -- (5,5);

            \node[red] at (0.5, 3.5) {$\ast$};
            \node[red] at (1.5, 3.5) {M};
        \end{tikzpicture}
        \caption{Patterns for $r-i-b$}
        \label{subfig:add-rise}
    \end{subfigure}
    \hfill
    \begin{subfigure}{0.32\textwidth}
        \centering
        \begin{tikzpicture}[scale=0.8]
            \draw[step=1.0, gray!60, ultra thin] (0,0) grid (6,6);
            \draw[gray!60] (1,1) -- (5,5);
            \draw[ultra thick, -sharp >, sharp angle = 45] (1,2) -- (2,2);
            \draw[ultra thick, red, sharp <-sharp >, sharp angle = 45] (2,2) -- (2,3) -- (3,3);
            \draw[ultra thick, sharp <-, sharp angle = 45] (3,3) -- (3,4);

            \node[red] at (2.5, 2.5) {M};
            \node[red] at (1.5, 2.5) {$\bullet$};
        \end{tikzpicture}
        \caption{Patterns for $a-i-b$}
        \label{subfig:add-valley}
    \end{subfigure}
    \hfill
    \begin{subfigure}{0.32\textwidth}
        \centering
        \begin{tikzpicture}[scale=0.8]
            \draw[step=1.0, gray!60, ultra thin] (0,0) grid (6,6);
            \draw[gray!60] (1,1) -- (5,5);
            \draw[ultra thick, -sharp >, sharp angle = 45] (1,2) -- (2,2);
            \draw[ultra thick, red, sharp <-sharp >, sharp angle = 45] (2,2) -- (2,3) -- (3,3);
            \draw[ultra thick, sharp <-, sharp angle = 45] (3,3) -- (3,4);

            \node[red] at (2.5, 2.5) {M};
        \end{tikzpicture}
        \caption{Patterns for $j-r-a+i+b$}
        \label{subfig:add-non-dec}
    \end{subfigure}
    
    \caption[Indices of the combinatorial interpretation]{The interpretation of the indices in the combinatorics of \Cref{thm:full_recursion}. The parts in red (segments and decorations) are removed in the transformation.}
    \label{fig:recursion_indices}
\end{figure}

The proof of \Cref{thm:theta_q1} is now straightforward.

\begin{proof}[Proof of \Cref{thm:theta_q1}]
    The result follows from \Cref{thm:combinatorial_recursion} by taking the sum over $b$, recalling that
    \[ \sum_{b=0}^{n-k-l-1} E_{n-k-l,n-k-l-b} = \sum_{b=1}^{n-k-l} E_{n-k-l,b} = e_{n-k-l}. \]
\end{proof}

It is worth noting that \Cref{thm:combinatorial_recursion} provides substantial information about the form that a dinv statistic on our objects might take. Indeed, the operations in \Cref{subfig:decorate-rise-add-valley,subfig:decorate-valley-add-rise} should not change the dinv at all, while the contributions of the other operations should be described by $q$-binomial coefficients. For \Cref{subfig:add-rise,subfig:add-valley,subfig:add-non-dec} the interpretation is relatively straightforward, even for the dinv inherited from the valley version of the Delta conjecture (disregarding the decorations on the rises), but for \Cref{subfig:decorate-valley-on-diagonal-add-rise} the behavior should be significantly more complicated and similar to that of the sdinv statistic on segmented Smirnov words (cf.\ \cite[Lemma~3.7]{IraciNadeauVandenWyngaerd2024Smirnov}), or a variant thereof.

Indeed, if we had a global dinv statistic for the $k+l=n-1$ case, we could use this combinatorial interpretation to derive one in general. This reduces the problem to a much simpler one. Although that problem is solved when the area is $0$, a further step is still needed to formulate a conjecture in general.

\section{Acknowledgments}

The authors would like to thank Anton Mellit for providing a SageMath implementation of the $\A_{q,t}$ algebra, and Giovanni Paolini for some helpful discussions about the extension of $\Theta$ to $\A_{q,t}$.

The authors used OpenAI GPT-5.6 Sol and GPT-6 Astra to assist with proofreading, LaTeX consistency, and checking and clarifying the proof of \Cref{thm:same-operators}; the authors verified all resulting changes and take full responsibility for the final content.

D'Adderio, Iraci and Pagaria are partially supported by PRIN 2022A7L229 ALTOP, and by the GNSAGA--INdAM research group.

Iraci is partially supported by the GNSAGA--INdAM research group, CUP E53C25002010001.

\bibliographystyle{amsalpha}
\bibliography{references}

\begin{bibdiv}
\begin{biblist}

\bib{Achim2025Aristotle}{article}{
      author={Achim, Tudor},
      author={Best, Alex},
      author={Bietti, Alberto},
      author={others},
       title={Aristotle: {IMO}-level automated theorem proving},
        date={2025},
      eprint={2510.01346},
         url={https://arxiv.org/abs/2510.01346},
}

\bib{BenDaliBonzomDoulega2026PathOperators}{misc}{
      author={Ben~Dali, Houcine},
      author={Bonzom, Valentin},
      author={Do{\l}{\k{e}}ga, Maciej},
       title={Path operators and $(q,t)$-tau functions},
        date={2025},
         url={https://arxiv.org/abs/2506.06036},
}

\bib{BergeronGarsia1999ScienceFiction}{incollection}{
      author={Bergeron, Fran{\c{c}}ois},
      author={Garsia, Adriano~M.},
       title={Science fiction and {M}acdonald's polynomials},
        date={1999},
   booktitle={Algebraic methods and $q$-special functions},
      series={CRM Proceedings \& Lecture Notes},
      volume={22},
   publisher={American Mathematical Society},
     address={Providence, RI},
       pages={1\ndash 52},
}

\bib{BergeronGarsiaHaimanTesler1999IdentitiesPositivityConjectures}{article}{
      author={Bergeron, Fran{\c{c}}ois},
      author={Garsia, Adriano~M.},
      author={Haiman, Mark},
      author={Tesler, Glenn},
       title={Identities and positivity conjectures for some remarkable
  operators in the theory of symmetric functions},
        date={1999},
     journal={Methods and Applications of Analysis},
      volume={6},
      number={3},
       pages={363\ndash 420},
         url={https://doi.org/10.4310/MAA.1999.v6.n3.a7},
}

\bib{BergeronGarsiaSergelXin2015CompositionalShuffleConjectures}{article}{
      author={Bergeron, Fran{\c{c}}ois},
      author={Garsia, Adriano~M.},
      author={Sergel~Leven, Emily},
      author={Xin, Guoce},
       title={Compositional $(km, kn)$-shuffle conjectures},
        date={2016},
     journal={International Mathematics Research Notices},
      volume={2016},
      number={14},
       pages={4229\ndash 4270},
         url={https://doi.org/10.1093/imrn/rnv272},
}

\bib{BergeronGarsiaSergelXin2016PlethysticOperators}{article}{
      author={Bergeron, Fran{\c{c}}ois},
      author={Garsia, Adriano~M.},
      author={Sergel~Leven, Emily},
      author={Xin, Guoce},
       title={Some remarkable new plethystic operators in the theory of
  {M}acdonald polynomials},
        date={2016},
     journal={Journal of Combinatorics},
      volume={7},
      number={4},
       pages={671\ndash 714},
         url={https://doi.org/10.4310/JOC.2016.v7.n4.a6},
}

\bib{BHMPS2023ProofExtendedDelta}{article}{
      author={Blasiak, Jonah},
      author={Haiman, Mark},
      author={Morse, Jennifer},
      author={Pun, Anna},
      author={Seelinger, George~H.},
       title={A proof of the extended {D}elta conjecture},
        date={2023},
     journal={Forum of Mathematics, Pi},
      volume={11},
       pages={e6},
         url={https://doi.org/10.1017/fmp.2023.3},
}

\bib{BHMPS2023ShuffleAnyLine}{article}{
      author={Blasiak, Jonah},
      author={Haiman, Mark},
      author={Morse, Jennifer},
      author={Pun, Anna},
      author={Seelinger, George~H.},
       title={A shuffle theorem for paths under any line},
        date={2023},
     journal={Forum of Mathematics, Pi},
      volume={11},
       pages={e5},
         url={https://doi.org/10.1017/fmp.2023.4},
}

\bib{BHMPS2024LLT}{article}{
      author={Blasiak, Jonah},
      author={Haiman, Mark},
      author={Morse, Jennifer},
      author={Pun, Anna},
      author={Seelinger, George~H.},
       title={{LLT} polynomials in the {S}chiffmann algebra},
        date={2024},
     journal={Journal f{\"u}r die reine und angewandte Mathematik},
      volume={2024},
      number={811},
       pages={93\ndash 133},
         url={https://doi.org/10.1515/crelle-2024-0012},
}

\bib{BHMPS2025LoehrWarrington}{article}{
      author={Blasiak, Jonah},
      author={Haiman, Mark},
      author={Morse, Jennifer},
      author={Pun, Anna},
      author={Seelinger, George~H.},
       title={Dens, nests and the {L}oehr--{W}arrington conjecture},
        date={2025},
     journal={Journal of the American Mathematical Society},
      volume={38},
      number={4},
       pages={1049\ndash 1106},
         url={https://doi.org/10.1090/jams/1057},
}

\bib{BHMPS2025RaisingOperatorFormula}{article}{
      author={Blasiak, Jonah},
      author={Haiman, Mark},
      author={Morse, Jennifer},
      author={Pun, Anna},
      author={Seelinger, George~H.},
       title={A raising operator formula for {M}acdonald polynomials},
        date={2025},
     journal={Forum of Mathematics, Sigma},
      volume={13},
       pages={e47},
         url={https://doi.org/10.1017/fms.2025.8},
}

\bib{CarlssonMellit2018ShuffleConjecture}{article}{
      author={Carlsson, Erik},
      author={Mellit, Anton},
       title={A proof of the shuffle conjecture},
        date={2018},
     journal={Journal of the American Mathematical Society},
      volume={31},
      number={3},
       pages={661\ndash 697},
         url={https://doi.org/10.1090/jams/893},
}

\bib{DAdderioIraci2023}{article}{
      author={D'Adderio, Michele},
      author={Iraci, Alessandro},
       title={Some consequences of the valley {D}elta conjectures},
        date={2023},
     journal={Annals of Combinatorics},
      volume={27},
       pages={727\ndash 764},
         url={https://doi.org/10.1007/s00026-023-00663-1},
}

\bib{DAdderioIraciLeBorgneRomeroVandenWyngaerd2022TieredTrees}{article}{
      author={D'Adderio, Michele},
      author={Iraci, Alessandro},
      author={Le~Borgne, Yvan},
      author={Romero, Marino},
      author={Vanden~Wyngaerd, Anna},
       title={Tiered trees and {T}heta operators},
        date={2022},
     journal={International Mathematics Research Notices},
      volume={2022},
      number={20},
       pages={16231\ndash 16269},
         url={https://doi.org/10.1093/imrn/rnac258},
}

\bib{DAdderioIraciVandenWyngaerd2021ThetaOperators}{article}{
      author={D'Adderio, Michele},
      author={Iraci, Alessandro},
      author={Vanden~Wyngaerd, Anna},
       title={{T}heta operators, refined {D}elta conjectures, and
  coinvariants},
        date={2021},
     journal={Advances in Mathematics},
      volume={376},
       pages={107447},
         url={https://doi.org/10.1016/j.aim.2020.107447},
}

\bib{DAdderioMellit2022CompositionalDelta}{article}{
      author={D'Adderio, Michele},
      author={Mellit, Anton},
       title={A proof of the compositional {D}elta conjecture},
        date={2022},
     journal={Advances in Mathematics},
      volume={402},
       pages={108342},
         url={https://doi.org/10.1016/j.aim.2022.108342},
}

\bib{DAdderioRomero2023ThetaIdentities}{article}{
      author={D'Adderio, Michele},
      author={Romero, Marino},
       title={New identities for theta operators},
        date={2023},
     journal={Transactions of the American Mathematical Society},
      volume={376},
      number={8},
       pages={5875\ndash 5901},
         url={https://doi.org/10.1090/tran/8911},
}

\bib{GarsiaHaglundXinZabrocki2016PieriRules}{incollection}{
      author={Garsia, Adriano~M.},
      author={Haglund, James},
      author={Xin, Guoce},
      author={Zabrocki, Mike},
       title={Some new applications of the {S}tanley--{M}acdonald {P}ieri
  rules},
        date={2016},
   booktitle={The mathematical legacy of {R}ichard {P}. {S}tanley},
   publisher={American Mathematical Society},
     address={Providence, RI},
       pages={141\ndash 168},
}

\bib{GarsiaHaiman1993GradedRepresentationModel}{article}{
      author={Garsia, Adriano~M.},
      author={Haiman, Mark},
       title={A graded representation model for {M}acdonald's polynomials},
        date={1993},
     journal={Proceedings of the National Academy of Sciences},
      volume={90},
      number={8},
       pages={3607\ndash 3610},
         url={https://doi.org/10.1073/pnas.90.8.3607},
}

\bib{GarsiaHaimanTesler1999ExplicitPlethysticFormulas}{article}{
      author={Garsia, Adriano~M.},
      author={Haiman, Mark},
      author={Tesler, Glenn},
       title={Explicit plethystic formulas for {M}acdonald {$q,t$}-{K}ostka
  coefficients},
        date={1999},
     journal={S{\'e}minaire Lotharingien de Combinatoire},
      volume={42},
       pages={Art. B42m, 45 pp.},
         url={http://www.emis.de/journals/SLC/wpapers/s42garsia.html},
}

\bib{GillespieGorskyGriffin2025Skewing}{article}{
      author={Gillespie, Maria},
      author={Gorsky, Eugene},
      author={Griffin, Sean~T.},
       title={A combinatorial skewing formula for the {R}ise {D}elta
  {T}heorem},
        date={2025},
     journal={Combinatorial Theory},
      volume={5},
      number={3},
}

\bib{GorskyNegut2015RefinedKnotInvariants}{article}{
      author={Gorsky, Eugene},
      author={Negu\textcommabelow{t}, Andrei},
       title={Refined knot invariants and {H}ilbert schemes},
        date={2015-09},
     journal={Journal de Math{\'e}matiques Pures et Appliqu{\'e}es},
      volume={104},
      number={3},
       pages={403\ndash 435},
         url={https://doi.org/10.1016/j.matpur.2015.03.003},
}

\bib{Haglund2008Book}{book}{
      author={Haglund, James},
       title={The {$q,t$}-{C}atalan numbers and the space of diagonal
  harmonics},
      series={University Lecture Series},
   publisher={American Mathematical Society},
     address={Providence, RI},
        date={2008},
      volume={41},
        note={With an appendix on the combinatorics of Macdonald polynomials},
}

\bib{HaglundHaimanLoehrRemmelUlyanov2005ShuffleConjecture}{article}{
      author={Haglund, James},
      author={Haiman, Mark},
      author={Loehr, Nicholas~A.},
      author={Remmel, Jeffrey~B.},
      author={Ulyanov, Anatoly},
       title={A combinatorial formula for the character of the diagonal
  coinvariants},
        date={2005},
     journal={Duke Mathematical Journal},
      volume={126},
      number={2},
       pages={195\ndash 232},
         url={https://doi.org/10.1215/S0012-7094-04-12621-1},
}

\bib{HaglundMorseZabrocki2012CompositionalShuffleConjecture}{article}{
      author={Haglund, James},
      author={Morse, Jennifer},
      author={Zabrocki, Mike},
       title={A compositional shuffle conjecture specifying touch points of the
  {D}yck path},
        date={2012},
     journal={Canadian Journal of Mathematics},
      volume={64},
      number={4},
       pages={822\ndash 844},
         url={https://doi.org/10.4153/CJM-2011-078-4},
}

\bib{HaglundRemmelWilson2018DeltaConjecture}{article}{
      author={Haglund, James},
      author={Remmel, Jeffrey~B.},
      author={Wilson, Andrew~T.},
       title={The {D}elta conjecture},
        date={2018},
     journal={Transactions of the American Mathematical Society},
      volume={370},
      number={6},
       pages={4029\ndash 4057},
         url={https://doi.org/10.1090/tran/7096},
}

\bib{Haiman2001nFactorial}{article}{
      author={Haiman, Mark},
       title={Hilbert schemes, polygraphs and the {M}acdonald positivity
  conjecture},
        date={2001},
     journal={Journal of the American Mathematical Society},
      volume={14},
      number={4},
       pages={941\ndash 1006},
         url={https://doi.org/10.1090/S0894-0347-01-00373-3},
}

\bib{Haiman2002HilbertScheme}{article}{
      author={Haiman, Mark},
       title={Vanishing theorems and character formulas for the {H}ilbert
  scheme of points in the plane},
        date={2002},
     journal={Inventiones Mathematicae},
      volume={149},
      number={2},
       pages={371\ndash 407},
         url={https://doi.org/10.1007/s002220200219},
}

\bib{Iraci2026Code}{misc}{
      author={Iraci, Alessandro},
       title={Code accompanying ``{L}eaving the {H}all''},
        date={2026},
        note={GitHub repository,
  \url{https://github.com/SashaIr/leaving-the-hall}. Co-authored by Aristotle
  (Harmonic)},
}

\bib{IraciNadeauVandenWyngaerd2024Smirnov}{article}{
      author={Iraci, Alessandro},
      author={Nadeau, Philippe},
      author={Vanden~Wyngaerd, Anna},
       title={{S}mirnov words and the {D}elta conjectures},
        date={2024},
     journal={Advances in Mathematics},
      volume={452},
       pages={109793},
         url={https://doi.org/10.1016/j.aim.2024.109793},
}

\bib{IraciPagariaPaolini2026FallingStars}{article}{
      author={Iraci, Alessandro},
      author={Pagaria, Roberto},
      author={Paolini, Giovanni},
       title={Falling stars: a fall-decorated rational shuffle theorem},
        date={2026},
     journal={Forum of Mathematics, Sigma},
      volume={14},
       pages={e92},
}

\bib{Macdonald1995Book}{book}{
      author={Macdonald, Ian~G.},
       title={Symmetric functions and {H}all polynomials},
     edition={2},
      series={Oxford Mathematical Monographs},
   publisher={Clarendon Press},
     address={Oxford},
        date={1995},
        ISBN={0-19-853489-2},
        note={With contributions by A. Zelevinsky},
}

\bib{Mathlib2020}{inproceedings}{
      author={mathlib Community, The},
       title={The {L}ean mathematical library},
        date={2020-01},
   booktitle={Proceedings of the 9th acm sigplan international conference on
  certified programs and proofs (cpp '20)},
   publisher={ACM},
       pages={367\ndash 381},
}

\bib{Mellit2020SpringerFibers}{article}{
      author={Mellit, Anton},
       title={{P}oincar\'e polynomials of character varieties, {M}acdonald
  polynomials and affine {S}pringer fibers},
        date={2020-07},
        ISSN={0003-486X},
     journal={Annals of Mathematics},
      volume={192},
      number={1},
       pages={165\ndash 228},
}

\bib{Mellit2021Rational}{article}{
      author={Mellit, Anton},
       title={Toric braids and $(m,n)$-parking functions},
        date={2021-12},
        ISSN={0012-7094},
     journal={Duke Mathematical Journal},
      volume={170},
      number={18},
       pages={4123\ndash 4169},
}

\bib{Mellit2022HomologyTorusKnots}{article}{
      author={Mellit, Anton},
       title={Homology of torus knots},
        date={2022-04},
        ISSN={1465-3060},
     journal={Geometry \& Topology},
      volume={26},
      number={1},
       pages={47\ndash 70},
}

\bib{Negut2014ShuffleAlgebra}{article}{
      author={Negu\textcommabelow{t}, Andrei},
       title={The {S}huffle {A}lgebra {R}evisited},
        date={2014-08},
        ISSN={1073-7928},
     journal={International Mathematics Research Notices},
      volume={2014},
      number={22},
       pages={6242\ndash 6275},
}

\bib{Romero2022}{article}{
      author={Romero, Marino},
       title={A proof of the theta operator conjecture},
        date={2022-01},
        ISSN={0097-3165},
     journal={Journal of Combinatorial Theory, Series A},
      volume={185},
       pages={105535},
}

\bib{SchiffmannVasserot2011EllipticHallAlgebra}{article}{
      author={Schiffmann, Olivier},
      author={Vasserot, Eric},
       title={The elliptic {H}all algebra, {C}herednik {H}ecke algebras and
  {M}acdonald polynomials},
        date={2011},
        ISSN={0010-437X},
     journal={Compositio Mathematica},
      volume={147},
      number={1},
       pages={188\ndash 234},
}

\end{biblist}
\end{bibdiv}

\end{document}